\documentclass{article}[11pt]
\usepackage[utf8]{inputenc}
\usepackage[paper=a4paper,dvips,top=3cm,left=2.4cm,right=2.4cm,foot=1cm,bottom=3cm]{geometry}
\usepackage{amssymb}
\usepackage{amsmath}
\usepackage{authblk}
\usepackage{amsthm}
\usepackage{graphicx}
\usepackage{subfigure}
\usepackage{epstopdf}
\usepackage{color}
\usepackage{soul}
\usepackage{float}
\usepackage{booktabs}
\usepackage{multirow}
\usepackage{adjustbox}
\usepackage[graphicx]{realboxes}
\usepackage{rotating}
\usepackage{array}
\usepackage{makecell}

\usepackage{footmisc}
\usepackage{caption}

\usepackage{url}
\newcommand{\X}{{\mathcal{X}}}
\newcommand{\Y}{{\mathcal{Y}}}
\newcommand{\A}{{\mathcal{A}}}
\newcommand{\B}{{\mathcal{B}}}
\newcommand{\C}{{\mathcal{C}}}

\newcommand{\Q}{{\mathbb{H}}}
\newcommand{\pro}{{*_{\mu}}}
\newcommand{\proone}{{*_{\mu}^{1}}}
\newcommand{\protwo}{{*_{\mu}^{2}}}
\newcommand{\prothree}{{*_{\mu}^{3}}}

\newcommand{\ee}{{\bf{e}}}
\newcommand{\rank}{{\rm{rank}}}
\newcommand{\diag}{{\rm{diag}}}
\newcommand{\Circ}{{\rm{circ}}}
\newcommand{\prox}{{\rm{prox}}}
\newcommand{\fold}{{\rm{fold}}}
\newcommand{\unfold}{{\rm{unfold}}}
\newcommand{\fft}{{\rm{fft}}}
\newcommand{\ifft}{{\rm{ifft}}}
\newcommand{\Tr}{{\rm{tr}}}
\newcommand{\ii}{{\bf{i}}}
\newcommand{\jj}{{\bf{j}}}
\newcommand{\kk}{{\bf{k}}}
\newtheorem{theorem}{Theorem}[section]
\newtheorem{definition}{Definition}[section]
\newtheorem{lemma}{Lemma}[section]
\newtheorem{remark}{Remark}[section]
\newtheorem{example}{Example}[section]
\newtheorem{corollary}{Corollary}[section]

\usepackage{algorithm}
\usepackage{algorithmic}

\begin{document}

\title{\huge Low-rank quaternion tensor completion for color video inpainting via a novel factorization strategy}
\author[]{Zhenzhi Qin\protect\footnotemark[1],\quad Zhenyu Ming\protect\footnotemark[2],\quad Defeng Sun\protect\footnotemark[3],\quad Liping Zhang\protect\footnotemark[1]}

\maketitle

\date{}

\renewcommand{\thefootnote}{\fnsymbol{footnote}}
\footnotetext[1]{Department of Mathematical Sciences, Tsinghua University, Beijing 100084, China ({\tt qzz19@mails.tsinghua.edu.cn}; \quad {\tt lipingzhang@mail.tsinghua.edu.cn})}
\footnotetext[2]{Theory Lab, Central Research Institute, 2012 Labs, Huawei Technologies Co., Ltd., Hong Kong ({\tt mathmzy@163.com})}
\footnotetext[3]{Department of Applied Mathematics, The Hong Kong Polytechnic University,  Hung Hom, Kowloon, Hong Kong
({\tt defeng.sun@polyu.edu.hk})}
\renewcommand{\thefootnote}{\arabic{footnote}}

\vspace{-0.2in}

\begin{abstract}


 Recently, a quaternion tensor product named Qt-product was proposed, and then the singular value decomposition and the rank of a third-order quaternion tensor were given. From a more applicable perspective, we extend the Qt-product and propose a novel multiplication principle for third-order quaternion tensor named gQt-product. With the gQt-product, we introduce a brand-new singular value decomposition for third-order quaternion tensors named gQt-SVD and then define gQt-rank and multi-gQt-rank. We prove that the optimal low-rank approximation of a third-order quaternion tensor exists and some numerical experiments demonstrate the low-rankness of color videos. So, we apply the low-rank quaternion tensor completion to color video inpainting problems and present alternating least-square algorithms to solve the proposed low gQt-rank and multi-gQt-rank quaternion tensor completion models. The convergence analyses of the proposed algorithms are established and some numerical experiments on various color video datasets show the high recovery accuracy  and computational efficiency of our methods.

{\bf keywords}: quaternion tensor; color video inpainting; low-rank; singular value decomposition; tensor completion
\end{abstract}

\section{Introduction}
\noindent
\par
The purpose of this paper is to recover color videos via tensor completion. So far, tensors have been widely applied to signal processing \cite{Sidiropoulos2000}, computer vision \cite{2002Multilinear}, graph analysis \cite{PZhou2018,ZZhang2014,ZZhang2017} and data mining \cite{Savas2007}, to name a few. Tensor decomposition is a fundamental tool to cope with large-scale data which is arranged in tensor-based forms, since the scale of tensor data can be notably reduced while most inherent information are still preserved by using decomposition techniques. Four commonly used tensor decomposition methods are  CANDECOMP/PARAFAC (CP) decomposition \cite{Carroll1970,RHarshman1970},  Tucker decomposition \cite{Tucker1966}, tensor singular value decomposition (t-SVD) \cite{tproduct} and Triple Decomposition \cite{Qi2021tripledecomposition}, and the corresponding ranks are called CP rank \cite{Carroll1970,RHarshman1970}, Tucker rank \cite{Tucker1966}, tubal rank \cite{MKilmer2013} and Triple rank \cite{Qi2021tripledecomposition}, respectively.

For a positive integer $n$, $[n]\doteq \{ 1,2,\ldots,n \}$. Suppose that $\A\in\mathbb{R}^{n_1\times n_2\times \cdots \times n_p}$ is an $p$-th order tensor, where $n_1,\dots,n_p\in\mathbb{N}^{+}$. The CP decomposition is to decompose $\A$ as a sum of some outer products of $p$ vectors:
\begin{equation}\label{cpd}
    \A=\sum_{i=1}^{r}a_1^{(i)}\circ a_2^{(i)}\circ \cdots \circ a_p^{(i)},
\end{equation}
where the symbol ``$\circ$" denotes the outer product and $a_j^{(i)}\in\mathbb{R}^{n_j}$, $i\in [r]$, $j\in [p]$. The smallest $r$ required in CP decomposition \eqref{cpd} is defined as the CP rank of $\A$. It is learned from \cite{CHillar2013} that, in general, determining the CP rank of a given tensor whose order is no less than three is an NP-hard problem. In contrast to CP decomposition, Tucker decomposition is more computationally efficient.  Hence a number of low-rank tensor completion and recovery models are based on Tucker rank \cite{JCai2010,BRecht2010,ZWen2012}. Precisely, Tucker rank is a vector of the matrix ranks
\begin{equation*}
    {\rm rank}_{\rm TC}(\A)=\big( \rank(\A_{(1)}),\rank(\A_{(2)}),\ldots,\rank(\A_{(p)}) \big),
\end{equation*}
where $\A_{(i)}\in\mathbb{R}^{n_i\times (\Pi_{k=1}^{p}n_k/n_i)}$ is mode-$i$ matricization of tensor ($i\in [p]$). CP decomposition and Tucker decomposition are applicable to tensors with arbitrary orders. In 2011, Kilmer and Martin proposed a novel decomposition strategy specifically for third-order tensors \cite{tproduct}. Whereafter, the relevant tubal rank was introduced and studied in \cite{MKilmer2013} and testified to have excellent performance for image and video inpainting problems \cite{PZhou2018}.

Third-order tensors are the most widely used higher-order tensors in applications \cite{PZhou2018,ZZhang2014,ZZhang2017,MKilmer2013,EAcar2011,LYang2016}. For instance, a grey scale video can be viewed as a third-order tensor indexed by two spatial variables and one temporal variable. Unless otherwise specialized, tensors in this paper are of third-order. Low-rank tensor completion is one of the most important problems in tensor processing and analysis. It aims at filling in the missing entries of a partially observed low-rank tensor. Many practical datasets are highly structured in the sense that they can be approximately represented through a low-rank decomposition \cite{YXu2013,JiLiu2009}. As a consequence, the key idea of the recovery process is to find the low-rank approximation of the original tensor via the observed data, i.e.,
\begin{equation}\label{sec1model1}
    \min_{\C}\ \rank(\C),\quad \text{s.t.}\ P_{\Omega}(\C)=P_{\Omega}(\mathcal{M}),
\end{equation}
where $\rank(\cdot)$ is a certain tensor rank and $\Omega$ is the index set locating the observed data, $P_{\Omega}(\cdot)$ is a linear operator that extracts the entries in $\Omega$ and fills the entries not in $\Omega$ with zeros, and $\mathcal{M}$ is the raw tensor.

As previously mentioned, addressing model \eqref{sec1model1} with CP rank is an NP-hard problem. An alternative way is to employ Tucker rank instead of CP rank:
\begin{align}\label{sec1model2}
    \min_{\C}\ \sum_{i=1}^{p} \rank(\C_{(i)}),\quad \text{s.t.}\ P_{\Omega}(\C)=P_{\Omega}(\mathcal{M}),
\end{align}
and the nuclear-norm based convex relaxation of model \eqref{sec1model2} is considered as
\begin{align}\label{sec1model3}
    \min_{\C}\ \sum_{i=1}^{p} \|\C_{(i)}\|_{*},\quad \text{s.t.}\ P_{\Omega}(\C)=P_{\Omega}(\mathcal{M}).
\end{align}
However, Romera-Paredes et al. \cite{RParedes2013} proved that \eqref{sec1model3} is not a tight convex relaxation of \eqref{sec1model2}, and SVD is needed  to solve \eqref{sec1model3}, which will lead to high computational cost when coping with large-scale issues. To overcome the computational difficulty,  a matrix factorization method was designed by Xu et al. \cite{YXu2013}, which preserves the low-rank structure of the unfolded matrices, i.e.,
\begin{align}\label{sec1model4}
    \min_{X^{i},Y^{i},\C}\ \sum_{i=1}^{p} \alpha_{i}\|X^{i}Y^{i}-\C_{(i)}\|_{F}^2,\quad \text{s.t.}\ P_{\Omega}(\C)=P_{\Omega}(\mathcal{M}),
\end{align}
where $\alpha_i$ is a positive weight parameter satisfying  $\sum_{i=1}^{p}\alpha_i=1$. A similar third-order tensor recovery method based on Triple decomposition is proposed in \cite{Qi2021tripledecomposition}. As pointed in \cite{tproduct,MKilmer2013}, directly unfolding a tensor will destroy the multi-way structure of the original data, resulting in vital information loss and degraded recovery performance. Besides, solving \eqref{sec1model4} requires to deal with $p$ matrices and each matrix owns the same scale components as the original tensor. Thus the computational cost is relatively expensive. Instead, tubal rank has been adopted in \eqref{sec1model1} and testified to have not only promising recovery performance but also efficient computational process.  Semerci et al. \cite{Semerci2014} developed a new tensor nuclear norm (TNN) based on t-SVD, and subsequently Zhang et al. \cite{ZZhang2014} applied TNN to tensor completion problems. Zhou et al. \cite{PZhou2018} proposed the following model based on tubal rank and tensor product (t-product) to replace model \eqref{sec1model1}:
\begin{equation}\label{sec1model5}
    \min_{\A,\B,\C}\ \frac{1}{2}\|\A*\B-\C\|_{F}^2,\quad \text{s.t.}\ P_{\Omega}(\C)=P_{\Omega}(\mathcal{M}),
\end{equation}
where ``*" denotes the t-product \cite{tproduct}. From \cite{tproduct,MKilmer2013}, one can deal with t-product via fast Fourier transform (FFT) and block diagonalization of third-order tensors, which can significantly reduce the computational cost.

{\bf Motivation.} We now briefly describe the motivation of this paper here. All the above methods explore the approximate low-rank property of higher-order tensors. However, they are not good enough for the classic color video inpainting problems. In specific, the traditional t-SVD and Triple decomposition are specially designed for third-order tensors, while a color video can be naturally described as a fourth-order tensor, with four dimensions representing the length, width, frame numbers and RGB-channels of the considered color video, respectively. Moreover, the computational complexity of CP decomposition is NP-hard, and the unfolding operation in Tucker decomposition will destroy the original multi-way structure of the data. Therefore, it is desirable to
design a new type of tensor factorization strategy which can tackle the above issues in terms of the capability, the recovery performance and the computational cost. Notice that  the red, green and blue channel pixel values can be intuitively  encoded on the three imaginary parts of a quaternion. The use of quaternion matrices for color image representation has been fully  studied in the literature \cite{2007Hypercomplex,jia2019,jns2019,suba2011,1997A}. In 2022, we proposed a quaternion tensor product (Qt-product) and then introduced the singular value decomposition (Qt-SVD) and the rank of a third-order quaternion tensor (Qt-rank) by employing the discrete Fourier transformation (DFT) technique. They also proved that the existence of the best low-rank approximation of a third order quaternion tensor from the theoretical point and the low-rank of color videos from numerical experiments \cite{Qin2022}. But this is not applied to solve color video inpainting problems, and the DFT used in \cite{Qin2022} is very special. From a more applicable perspective, in this paper, we generalize the Qt-product and propose a novel multiplication principle for third-order quaternion tensor, and then establish low-rank quaternion tensor completion models to recover color videos.

{\bf Contribution.} By introducing an extensive quaternion discrete Fourier transformation (QDFT) based on a pure quaternion basis, we propose a novel multiplication principle for third-order quaternion tensor named gQt-product, and then a new SVD is given. With such SVD, we establish two low-rank quaternion tensor completion models to recover the incomplete color video data, and present an alternating least-squared (ALS) algorithm to solve the color video inpainting problems. The numerical experiments show that our methods outperform  other state-of-the-arts in the recovery accuracy  and computational efficiency. The main contributions are summarized as follows.
\begin{itemize}
\item A generalized QDFT based on a pure quaternion basis is introduced and a novel quaternion tensor product named gQt-product is proposed. With the gQt-product, we define identity quaternion tensor, unitary quaternion tensor, conjugate transpose and inverse of quaternion tensor. We prove that the collection of all invertible $n\times n\times l$ quaternion tensors forms a ring under standard tensor addition and the gQt-product.

\item A new gQt-product based SVD for quaternion tensors named gQt-SVD is given, and then the gQt-rank and the nuclear norm of  third-order quaternion tensor are defined. We prove that the optimal low-rank approximation of third-order quaternion tensor exists and some numerical experiments demonstrate the low-rankness of color videos. Note that gQt-rank is only defined on one mode of third order quaternion tensor without low rank structure in the other two modes, so we also introduce multi-gQt-rank.

\item To cope with color video inpainting problem, we construct low-rank quaternion tensor completion models \eqref{sec1model1} based on gQt-rank and multi-gQt-rank, and further propose their evolved forms \eqref{modelcompute} and \eqref{sec5model} via the gQt-product. We present an ALS algorithm to solve \eqref{modelcompute} and \eqref{sec5model}, and also show that the sequence generated by the ALS algorithm globally converges to a stationary point of the problem by using the Kurdyka--{\L}ojasiewicz property exhibited in the resulting problem. Extensive numerical experiments on various color video datasets show the high  recovery accuracy  and computational efficiency of our methods. Especially,  the criterion of the recovery performance illustrates our gQt-SVD-based method is superior to the commonly used t-SVD-based one.

\end{itemize}

The rest of this paper is organized as follows. In Section \ref{sec2}, we list some existing results for quaternion matrices and quaternion tensors.  and discuss the Fourier transform of quaternion tensors. In Section \ref{secQtproduct}, we introduce a generalized QDFT based on a pure quaternion basis, and then gQt-product, gQt-SVD, gQt-rank and multi-gQt-rank are defined. In Section \ref{sec4}, we establish related low-rank quaternion tensor completion models to recover the incomplete color video data, and present the ALS algorithm to solve the resulting problem. Moreover, its convergence rate analysis is also established. In Section \ref{sec6}, some numerical results are reported to confirm the advantages of gQt-SVD-based methods. The conclusions are drawn in Section \ref{sec7}.

\section{Preliminary}\label{sec2}
\subsection{Quaternions}
\noindent
\par
Let $\mathbb{R}$ and $\mathbb{C}$ denote the real field and the complex field, respectively. The quaternion field, denoted as $\Q$, is a four-dimensional vector space over real number field $\mathbb{R}$ with an ordered basis, denoted by $\textbf{1}$, $\textbf{i}$, $\textbf{j}$ and $\textbf{k}$. Here $\textbf{i}$, $\textbf{j}$ and $\textbf{k}$ are three imaginary units with the following multiplication laws:
\begin{flalign*}
	\textbf{i}^2=\textbf{j}^2=\textbf{k}^2=-1,\	\textbf{i}\textbf{j}=-\textbf{j}\textbf{i}=\textbf{k},\  \textbf{j}\textbf{k}=-\textbf{k}\textbf{j}=\textbf{i},\  \textbf{k}\textbf{i}=-\textbf{i}\textbf{k}=\textbf{j}.\
\end{flalign*}

Let $x=a+b\ii +c\jj +d\kk \in \Q$, where $a,b,c,d\in \mathbb{R}$, then the conjugate of $x$ is defined by $$x^{*}\doteq a-b\ii -c\jj -d\kk,$$ the norm of $x$ is $$|x|=|x^*|=\sqrt{xx^*}=\sqrt{x^{*}x}=\sqrt{a^2+b^2+c^2+d^2}.$$
and if $x\neq 0$, then $x^{-1}=\frac{x^*}{|x|^2}$.

\subsection{Quaternion matrix and quaternion tensor}
\noindent
\par
We give some notations here. Scalars, vectors, matrices and third-order tensors are denoted as lowercase letters ($a, b, \ldots$),  bold-case lowercase letters ($\textbf{a, b, \ldots}$), capital letters ($A, B, \ldots$) and  Euler script letters ($\mathcal{A, B, \ldots}$), respectively. We use $\bf{0}, \emph{O}$ and $\mathcal{O}$ to denote zero vector, zero matrix and zero tensor with appropriate dimensions. We use symbols $e$ to represent the vector whose elements are all 1, and $I$ and $\mathcal{I}$ to denote the identity matrix, and the identity tensor, respectively. The identity tensor $\mathcal{I}$ will be defined in Section \ref{secQtproduct}.

Then a quaternion matrix $A=(A_{ij})\in\Q^{n_1\times n_2}$ can be denoted as
$$A=A_{\bf{e}}+A_{\ii}\ii+A_{\jj}\jj+A_{\kk}\kk, $$
where $A_{\bf{e}}, A_{\ii}, A_{\jj}, A_{\kk}\in\mathbb{R}^{n_1\times n_2}$. The transpose of $A$ is $A^T=(A_{ji})$. The conjugate transpose of $A$ is $$A^*=(A_{ji}^*)=A_{\bf{e}}^T-A_{\ii}^T\ii-A_{\jj}^T\jj-A_{\kk}^T\kk .$$
The Frobenius norm of $A$ is
$$ \| A \|_{F}\doteq \sqrt{\sum_{i=1}^{n_1}\sum_{j=1}^{n_2}|A_{ij}|^2}. $$
With a simple calculation, it can be seen that
\begin{equation}\label{tracerepre}
     \| A\|_F^2=\Tr(AA^*)=\Tr(A^*A),
\end{equation}
where $\Tr(\cdot)$ is the trace of a matrix.

Let $A\in\Q^{n\times n}$.  $A$ is a unitary matrix if and only if $AA^*=A^*A=I_{n}$, where $I_{n}\in\mathbb{R}^{n\times n}$ is the real $n\times n$ identity matrix. $A$ is invertible if $AB=BA=I_n$ for some $B\in\Q^{n\times n}$. The following lemma gives Some properties of invertible quaternion matrix which can be found in \cite[Theorem 4.1]{qsvd}.
\begin{lemma}\label{lemma21}
Let $A\in\Q^{n_1\times n_2}$ and $B\in\Q^{n_2\times n_3}$, then

\emph{(i)} $(AB)^*=B^*A^*$,

\emph{(ii)} $(AB)^{-1}=B^{-1}A^{-1}$ if $A$ and $B$ are invertible,

\emph{(iii)} $(A^*)^{-1}=(A^{-1})^*$ if $A$ is invertible.
\end{lemma}

The following theorem for the SVD of quaternion matrix (QSVD) was given in \cite{qsvd}.
\begin{theorem}\label{thm21}
Any quaternion matrix $A\in\Q^{n_1\times n_2}$ has the following QSVD form
\begin{equation*}
    A=U\begin{bmatrix} \Sigma_r & O \\ O & O
    \end{bmatrix}V^* ,
\end{equation*}
where $U\in\Q^{n_1\times n_1}$ and $V\in\Q^{n_2\times n_2}$ are unitary, and $\Sigma_r=\diag\{ \sigma_1, \dots , \sigma_r \}$ is a real nonnegative diagonal matrix, with $\sigma_1\geq\cdots \geq\sigma_r$ as the singular values of $A$.
\end{theorem}

By Theorem \ref{thm21}, the nuclear norm of $A$ is defined as $\| A \|_*=\sum\limits_{i=1}^{r}\sigma_i.$ The quaternion rank of $A$ is the number of its singular values, denoted as rank($A$).
 It follows from \cite[Lemma 9]{connectnuclear} that we can prove the following lemma, which reveals the relationship between the  matrix-factorization and the nuclear norm of a quaternion matrix.
\begin{lemma}\label{lemman22}
    For a given quaternion matrix $A\in\Q^{n_1\times n_2}$,
    \begin{equation}\label{zlp1}
        \|A\|_*=\min_{A=XY^*}\|X\|_F\|Y\|_F=\min_{A=XY^*}\frac{1}{2}(\|X\|_F^2+\|Y\|_F^2)
    \end{equation}
\end{lemma}
\begin{proof}
    We denote three parts in \eqref{zlp1} as (i), (ii) and (iii) from left to right.\\
    (ii)$\leq$ (iii): This follows from the arithmetic mean and geometric mean inequality.\\
    (iii)$\leq$ (i): We decompose $A$ into the form in Theorem \ref{thm21}, and then set
    $$\tilde{X}=U\begin{bmatrix} {\Sigma_r}^{1/2} & O \\ O & O
    \end{bmatrix},\quad \tilde{Y}=\begin{bmatrix} {\Sigma_r}^{1/2} & O \\ O & O
    \end{bmatrix}V.$$
    Hence, $\tilde{X},\tilde{Y}$ are feasible matrices of (iii), and $\frac{1}{2}(\|\tilde{X}\|_F^2+\|\tilde{Y}\|_F^2)=\|A\|_*$, which implies (iii) $\leq$ (i).\\
    (i)$\leq$ (ii): For all $X,Y$ with $A=XY^*$, let $\textbf{u}_i,\ \textbf{v}_i$ be the $i$-th column of $U$ and $V$, respectively. Then
    \begin{align*}
        \|A\|_*&=\sum_{i=1}^{r}\textbf{u}_i^*A\textbf{v}_i =\sum_{i=1}^{r}(X^*\textbf{u}_i)^*(Y^*\textbf{v}_i) \\
        &\leq \sum_{i=1}^{r}\|X^*\textbf{u}_i\|_F\|Y^*\textbf{v}_i\|_F \\
        &\leq (\sum_{i=1}^{r}\|X^*\textbf{u}_i\|_F^2)^{1/2}(\sum_{i=1}^{r}\|Y^*\textbf{v}_i\|_F^2)^{1/2} \\
        &\leq \|X^*U\|_F\|Y^*V\|_F =\|X\|_F\|Y\|_F,
    \end{align*}
    where the first and second inequalities are from Cauchy-Schwarz inequality, which can be verified on quaternion, and the third inequality holds because we complete the rest part of $U$ and $V$.
\end{proof}

A third-order quaternion tensor $\mathcal{A}\in \mathbb{H}^{n_{1}\times n_{2}\times n_{3}}$ is expressed as $$\mathcal{A}=(\mathcal{A}_{i_{1}i_{2}i_{3}}),\  \mathcal{A}_{i_{1}i_{2}i_{3}}\in\mathbb{H},\ 1\leq i_{t}\leq n_{t},\ 1\leq t\leq 3.$$
Also, it can be expressed as
\begin{equation}\label{agqt}\A=\A_{\bf{e}}+\A_{\ii}\ii+\A_{\jj}\jj+\A_{\kk}\kk. \end{equation}
where $\A_{\bf{e}},\A_{\ii},\A_{\jj},\A_{\kk}\in\mathbb{R}^{n_1\times n_2\times n_3}$.
We use the Matlab notations $\mathcal{A}(i, :, :)$, $\mathcal{A}(:, j, :)$ and $\mathcal{A}(:, :, k)$ to denote its $i$-th horizontal, $j$-th lateral and $k$-th frontal slice, respectively. Let $A^{(k)}=\mathcal{A}(:, :, k)$ be the $k$-th ($k\in[n_3]$) frontal slice and $\A^*$ denote its conjugate transpose (see Section \ref{secQtproduct}). The Frobenius norm of $\mathcal{A}$ is the sum of all norms of its entries, i.e.,
\begin{flalign*}
	\| \mathcal{A} \|_{F}\doteq\sqrt{\sum_{i=1}^{n_{1}}\sum_{j=1}^{n_{2}}\sum_{k=1}^{n_{3}}|\mathcal{A}_{ijk}|^2}.
\end{flalign*}
The block circulant matrix ${\Circ}(\mathcal{A})\in \mathbb{H}^{n_{1}n_{3}\times n_{2}n_{3}}$ of $\mathcal{A}$ is defined as
\begin{equation*}
	\mbox{circ}(\mathcal{A})\doteq\begin{bmatrix}
		A^{(1)} & A^{(n_3)} & \cdots & A^{(2)} \\
		A^{(2)} & A^{(1)} & \cdots & A^{(3)} \\
		\vdots & \vdots & \ddots & \vdots \\
		A^{(n_3)} & A^{(n_3-1)} & \cdots & A^{(1)}
	\end{bmatrix}.
\end{equation*}
The operator ``$\unfold$" is defined as
$$\unfold(\A) \doteq [A^{(1)}; A^{(2)}; \dots ; A^{(n_3)}],$$ and its inverse operator ``$\fold$" is defined by $\fold(\unfold(\A))=\A$. The operator ``${\diag}$" of $\A$ is given as
\begin{equation*}
    {\diag}(\A)\doteq\begin{bmatrix}
		A^{(1)} &  &  &  \\
		 & A^{(2)} &  &  \\
		 &  & \ddots &  \\
		 &  &  & A^{(n_3)}
	\end{bmatrix}.
\end{equation*}

In 2022, Qin et al \cite{Qin2022}  noted that a quaternion $x=a+b\ii +c\jj +d\kk$ can be written as $x=(a+b\ii) +\jj(c-d\ii)$ and then the  quaternion tensor $\A$ can be written as the form of $\A=\A_{\bf{1},\ii}+\jj \A_{\jj,\kk}$ with $\A_{\bf{1},\ii},\A_{\jj,\kk}\in\mathbb{C}^{n_1\times n_2\times n_3}$. So, Qin et al \cite{Qin2022} defined the following quaternion tensor product named Qt-product, and they also give SVD of third-order quaternion tensor by using the Qt-product and the following DFT.

\begin{definition}[{\bf Qt-product} \cite{Qin2022}]\label{qtp}
	For $\A\in\Q^{n_1\times r\times n_3}$ and $\B\in\Q^{r\times n_2\times n_3}$, define
	\begin{equation*}
		\A*_{\Q}\B\doteq \fold((\Circ(\A_{\bf{1},\ii})+\jj \Circ(\A_{\jj,\kk})\cdot (P_{n_3}\otimes I_{r}))\cdot \unfold(\B))\in\Q^{n_1\times n_2\times n_3},
	\end{equation*}
where the symbol ``$\otimes$" means the Kronecker product, the matrix $P_{n_3}=(P_{ij})\in \mathbb{R}^{n_3 \times n_3}$ is a permutation matrix with $P_{11}=P_{ij}=1$ if $i+j={n_3}+2,\ 2\leq i,j \leq {n_3};\ P_{ij}=0,$ otherwise.
\end{definition}
The DFT used in \cite{Qin2022} is given as the form of the normalized DFT matrix $F_{n_3}\in \mathbb{C}^{n_3\times n_3}$ with
\begin{equation}\label{almdft}
F_{n_3}(i,j)=\frac{1}{\sqrt{n_3}}\omega^{(i-1)(j-1)},\quad \omega=\exp({-{2\pi \ii}/{n_3}}) \ \ \text{and} \ i,j\in[n_3].
\end{equation}
The DFT \eqref{almdft} plays an important role in the Qt-SVD of third-order quaternion tensor given in \cite{Qin2022}. As a generalization of the traditional Fourier transform, the quaternion Fourier transform was first defined by Ell  \cite{1993QDFT} to process quaternion signal.  Motivated by the idea of quaternion Fourier transform in \cite{1993QDFT} and from a more applicable perspective, in this paper we define a new quaternion DFT (QDFT) based on a unit pure quaternion $\mu=a\ii+b\jj+c\kk$ with $\mu^2=-1$,  which can be regards as a generalization of DFT \eqref{almdft}. And then, we introduce a novel multiplication principle for third-order quaternion tensor via the new defined QDFT.

\section{QDFT, gQt-Product and gQt-SVD}\label{secQtproduct}
\noindent
\par
One major contribution of this paper is
to introduce a novel multiplication principle for third-order quaternion tensor, named gQt-product,  via our new defined QDFT. The gQt-product is also the cornerstone of the quaternion tensor decomposition.
In this section, we will introduce  gQT-product and some relevant properties. Theorem \ref{mainthm} is one of the main results, which shows the relationship between gQt-product of quaternion tensors and  matrix product of their  QDFT matrices. Furthermore, we propose a new SVD of quaternion tensor via gQt-product, named gQt-SVD. With gQt-SVD, we can find the low-rank optimal approximation of quaternion tensor.

\subsection{QDFT: a new DFT of third-order quaternion tensor}
\noindent
\par
In this subsection, we define a new quaternion DFT (QDFT) based on a unit pure quaternion $\mu=a\ii+b\jj+c\kk$ with $\mu^2=-1$,  which can be regards as a generalization of DFT \eqref{almdft}.
Because the quaternion multiplication is not commutative, QDFT of vectors can be defined by the sum of components multiplied by the exponential kernel of the transform from the right or from the left.

Let $F_{\mu,n_3}\in\Q^{n_3\times n_3}$ be the normalized QDFT matrix with the $(i,j)$-th element as
\begin{equation}\label{qdftu} F_{\mu,n_3}(i,j)=\frac{1}{\sqrt{n_3}}\omega^{(i-1)(j-1)},\ \ i,j\in [n_3], \end{equation}
where the kernel $\omega$ is defined as
$$\omega=\exp(-\mu 2\pi/n_3)=\cos(2\pi/n_3)-\mu\sin(2\pi/n_3).$$
Obviously, it follows that
\begin{equation}\label{kernel}
   \omega^{p}=\exp(-\mu 2\pi p/n_3)=\cos(2\pi p/n_3)-\mu\sin(2\pi p/n_3),\quad  p\in \mathbb{Z}.
\end{equation}
Clearly, when $\mu=\ii$, QDFT matrix defined as \eqref{qdftu} is equal to DFT matrix given in \eqref{almdft}. So, QDFT is more applicable than DFT.

It is easy to see that the result of QDFT of $\A\in\Q^{n_1\times n_2\times n_3}$ is still a quaternion tensor $\hat{\A}\in\Q^{n_1\times n_2\times n_3}$ with
$$\hat{\A}(i,j,:)=\sqrt{n_3}F_{\mu,n_3}\A(i,j,:),\quad  i\in[n_1],\ j\in[n_2].$$
Moreover,
\begin{equation}\label{Qfrobeeq1}
    \|\hat{\A}\|_F^2=n_3\|\A\|_F^2.
\end{equation}
We also denote the QDFT of $\A$ as ${\fft}(\A)$, and its inverse operator ``${\ifft}$" is defined by ${\ifft}({\fft}(\A))=\A$.
It is known that any real circulant matrix can be diagonalized by the normalized DFT matrix \cite{cirfft}. For QDFT, by simple calculation, we also obtain the same result for any real tensor $\A\in\mathbb{R}^{n_1\times n_2\times n_3}$:
\begin{equation}\label{eqdiag}
    (F_{\mu,n_3}\otimes I_{n_1})\cdot {\Circ}(\A)\cdot (F_{\mu, n_3}^*\otimes I_{n_2})={\diag}(\hat{\A}).
\end{equation}

Denote $P_{n}=(P_{ij})\in \mathbb{R}^{n \times n}$ as a permutation matrix with $P_{11}=P_{ij}=1$ if $i+j=n+2,\ 2\leq i,j \leq n;\ P_{ij}=0,$ otherwise. We now give some special kernels of QDFT matrices, which will be used to introduce gQt-product. For a unit pure quaternion $\mu=a\ii+b\jj+c\kk$ with $\mu^2=-1$, set
\begin{equation}\label{kernel2}
    \mu_{\ii}=a\ii-b\jj-c\kk,\quad  \mu_{\jj}=-a\ii+b\jj-c\kk,\quad  \mu_{\kk}=-a\ii-b\jj+c\kk.
\end{equation}
The following lemma gives the relationship between QDFT matrices in the form of \eqref{qdftu} with different quaternion basis.
\begin{lemma}\label{lemmafm}
 For two unit pure quaternion numbers  $\mu_1$ and $\mu_2$ with $\mu_{1}^2=\mu_2^2=-1$,  it holds that
\begin{equation*} F_{\mu_1, n}^*F_{\mu_2, n}=\frac{1}{2}(1-\mu_1\mu_2)I_n+\frac{1}{2}(1+\mu_1\mu_2)P_n,\quad
 F_{\mu_1, n}F_{\mu_2, n}=\frac{1}{2}(1+\mu_1\mu_2)I_n+\frac{1}{2}(1-\mu_1\mu_2)P_n.
\end{equation*}
\end{lemma}
\begin{proof}
It follows from \eqref{qdftu} and \eqref{kernel} that
\begin{equation}\label{cs}
    F_{\mu_{t}, n}=C_n-\mu_{t} S_n, \quad  F_{\mu_{t}, n}^*=C_n+\mu_{t} S_n, \quad t=1,2,
\end{equation}
where $C_n=(C_{ij})\in\mathbb{R}^{n\times n}$ and $S_n=(S_{ij})\in\mathbb{R}^{n\times n}$ satisfy
$$C_{ij}=\frac{1}{\sqrt{n}}\cos(2\pi (i-1)(j-1)/n), \quad S_{ij}=\frac{1}{\sqrt{n}}\sin(2\pi (i-1)(j-1)/n).$$
By \eqref{almdft} and \eqref{qdftu}, the traditional DFT matrix $F_n$ can be written as $F_{\ii, n}$. By simple computation, we have
$$F_{\ii, n}^*F_{\ii, n}=I_n, \quad F_{\ii, n}F_{\ii, n}=P_n.$$
Combining \eqref{cs}, it is easy to get
$$ C_n^2+S_n^2=I_n, \quad C_n^2-S_n^2=P,\ C_nS_n=S_nC_n=O.$$
Hence,
\begin{align*}
    F_{\mu_1, n}^*F_{\mu_2, n}&=(C_n+\mu_1 S_n)(C_n-\mu_2 S_n)=C_n^2-\mu_1\mu_2S_n^2 \\
    &=\frac{1}{2}(1-\mu_1\mu_2)I_n+\frac{1}{2}(1+\mu_1\mu_2)P_n.
\end{align*}
In the same way, we can prove that $F_{\mu_1, n}F_{\mu_2, n}=\frac{1}{2}(1+\mu_1\mu_2)I_n+\frac{1}{2}(1-\mu_1\mu_2)P_n. $
\end{proof}

\subsection{gQt-product: a new product between third-order quaternion tensors}
\noindent
\par
In this subsection, based on QDFT \eqref{qdftu} and the form of quaternion tensor $\A$ as \eqref{agqt}, we introduce the concept of gQt-product and define identity quaternion tensor and inverse quaternion tensor. Moreover, the relation of gQt-product of quaternion tensors and  matrix product of their  QDFT matrices is given.

\begin{definition}[{\bf gQt-product}]\label{defpro}
For any given unit pure quaternion $\mu=a\ii+b\jj+c\kk$ with $\mu^2=-1$, let $\mu_{\ii},\mu_{\jj},\mu_{\kk}$ be given as \eqref{kernel2} and matrices $T_{\ii},\ T_{\jj},\ T_{\kk}$ be defined as $$T_{\ii}=F_{\mu_{\ii},n_3}^*F_{\mu,n_3},\quad  T_{\jj}=F_{\mu_{\jj},n_3}^*F_{\mu, n_3},\quad T_{\kk}=F_{\mu_{\kk},n_3}^*F_{\mu, n_3}. $$
Define gQt-product of
	$\A\in\Q^{n_1\times r\times n_3}$ and $\B\in\Q^{r\times n_2\times n_3}$ as
	\begin{equation*}
		\A\pro\B\doteq {\fold}\big( ({\Circ}(\A_{\bf{e}})+\ii {\Circ}(\A_{\ii})\cdot (T_{\ii}\otimes I_{r})+\jj {\Circ}(\A_{\jj})\cdot (T_{\jj}\otimes I_{r})+\kk {\Circ}(\A_{\kk})\cdot (T_{\kk}\otimes I_{r}) )\cdot {\unfold}(\B)\big).
	\end{equation*}
Clearly, $\A\pro\B\in\Q^{n_1\times n_2\times n_3}$.
\end{definition}

We have the following remarks for Definition \ref{defpro}.
\begin{remark}{\rm
    When $\mu=\ii$, gQt-product ``$\pro$" becomes Qt-product ``$*_{\Q}$" in Definition \ref{qtp} \cite{Qin2022}. Moreover, if we restrict the product to real tensors, gQt-product is just t-product defined in \cite{tproduct}. In this view, gQt-product is a generalization of t-product in a wider field.}
\end{remark}

\begin{remark}{\rm
 For any unit pure quaternion $\mu=a\ii+b\jj+c\kk$ with $a^2+b^2+c^2=1$, by Lemma \ref{lemmafm}, we can get
    \begin{align*}
        &F_{\mu_{\ii}, n_3}^*F_{\mu, n_3}=\big(a^2-a(b\kk-c\jj)\big)I_{n_3}+\big(1-a^2+a(b\kk-c\jj)\big)P_{n_3} , \\
        &F_{\mu_{\jj}, n_3}^*F_{\mu, n_3}=\big(b^2-b(c\ii-a\kk)\big)I_{n_3}+\big(1-b^2+b(c\ii-a\kk)\big)P_{n_3} , \\
        &F_{\mu_{\kk}, n_3}^*F_{\mu, n_3}=\big(c^2-c(a\jj-b\ii)\big)I_{n_3}+\big(1-c^2+c(a\jj-b\ii)\big)P_{n_3}.
    \end{align*}
 Hence, it is easy to see that  $T_{\ii}, T_{\jj}, T_{\kk}$ is the combination of identity matrix $I_{n_3}$ and permutation matrix $P_{n_3}$, with the sum of the coefficients  being one. In this vision, the operator from ${\Circ}(\A)$ to ${\Circ}(\A_{\bf{e}})+\ii {\Circ}(\A_{\ii})\cdot (T_{\ii}\otimes I_{r})+\jj {\Circ}(\A_{\jj})\cdot (T_{\jj}\otimes I_{r})+\kk {\Circ}(\A_{\kk})\cdot (T_{\kk}\otimes I_{r})$ will not change the magnitude, so there is no loss of information in gQt-product.}
\end{remark}

We now present the relation of gQt-product of quaternion tensors and  matrix product of their  QDFT matrices in the following theorem. For any given $\A\in\Q^{n_1\times n_2\times n_3}$, setting its QDFT tensor as $\hat{\A}=\fft(\A)$, then we have
\begin{equation}\label{eqfft1}
   {\unfold}(\hat{\A})=\sqrt{n_3}(F_{\mu, n_3}\otimes I_{n_1})\cdot {\unfold}(\A).
\end{equation}
\begin{theorem}\label{mainthm}
    Let $\A\in\Q^{n_1\times r\times n_3}$, $\B\in\Q^{r\times n_2\times n_3}$ and $\mathcal{C}\in \Q^{n_1\times n_2\times n_3}$, $\hat{\A},\hat{\B},\hat{\mathcal{C}}$ be their QDFT tensors. Then,  \begin{equation*}
    {\diag}(\hat{\mathcal{C}})={\diag}(\hat{\A})\cdot {\diag}(\hat{\B}) \quad \Leftrightarrow \quad \mathcal{C}=\A\pro\B.
    \end{equation*}
\end{theorem}
\begin{proof}
	Clearly, $\ii F_{\mu,n_3}=F_{\mu_{\ii}, n_3}\ii$. Thus, it follows from \eqref{eqfft1} and \eqref{eqdiag} that
	\begin{align}
	    \diag\big( fft(\A_{\ii}\ii)\big)&=\diag\Big( \fold\big( \sqrt{n_3}(F_{\mu, n_3}\otimes I_{n_1})\cdot \unfold(\A_{\ii}\ii) \big) \Big) \nonumber \\
	    &=\diag\Big( \fold\big( \sqrt{n_3}(F_{\mu, n_3}\otimes I_{n_1})\cdot \unfold(\A_{\ii}) \big) \Big)\ii \nonumber \\
	    &= \diag\big( \fft(\A_{\ii})\big)\ii \nonumber \\
	    &=(F_{\mu,n_3}\otimes I_{n_1})\cdot \Circ(\A_{\ii})\cdot (F_{\mu, n_3}^*\otimes I_{n_2})\ii \nonumber \\
	    &=(F_{\mu,n_3}\otimes I_{n_1})\cdot \ii \Circ(\A_{\ii})\cdot (F_{\mu_{\ii}, n_3}^*\otimes I_{n_2}), \label{eqii1}
	\end{align}
	where the second equality holds due to the homogeneity of operators ``$\fold$", ``$\unfold$" and ``$\diag$".
	Similarly, we also have
	\begin{align}
	    \diag\big( \fft(\A_{\jj}\jj)\big)&=(F_{\mu,n_3}\otimes I_{n_1})\cdot \jj \Circ(\A_{\jj})\cdot (F_{\mu_{\jj}, n_3}^*\otimes I_{n_2}), \label{eqjj1} \\
	   \diag\big( \fft(\A_{\kk}\kk)\big)&=(F_{\mu,n_3}\otimes I_{n_1})\cdot \kk \Circ(\A_{\kk})\cdot (F_{\mu_{\kk}, n_3}^*\otimes I_{n_2}). \label{eqkk1}
	\end{align}
	So, by \eqref{eqii1}, \eqref{eqjj1} and \eqref{eqkk1}, it holds that
	\begin{align}
	    \diag(\hat{\A})&=\diag\big(\fft(\A_{\bf{e}})\big)+\diag\big( \fft(\A_{\ii}\ii)\big)+\diag\big( \fft(\A_{\jj}\jj)\big)+\diag\big( \fft(\A_{\kk}\kk)\big) \nonumber \\
	    &=(F_{\mu,n_3}\otimes I_{n_1})\cdot \big( \Circ(\A_{\bf{e}})\cdot (F_{\mu, n_3}^*\otimes I_{n_2})+ \ii \Circ(\A_{\ii})\cdot (F_{\mu_{\ii}, n_3}^*\otimes I_{n_2}) \nonumber \\
	    &\quad +\jj \Circ(\A_{\jj})\cdot (F_{\mu_{\jj}, n_3}^*\otimes I_{n_2})+\kk \Circ(\A_{\kk})\cdot (F_{\mu_{\kk}, n_3}^*\otimes I_{n_2}) \big), \label{eqdiagA}
	\end{align}
which implies
	\begin{align*}
	    \unfold\big( \hat{\C}\big)&=\unfold\big( \fft(\A\pro\B)\big) \\
	    &=\sqrt{n_3}(F_{\mu, n_3}\otimes I_{n_1})\cdot \unfold(\A\pro\B) \\
	    &=\sqrt{n_3}(F_{\mu, n_3}\otimes I_{n_1})\cdot \big( (\Circ(\A_{\bf{e}})+\ii \Circ(\A_{\ii})\cdot (F_{\mu_{\ii},n_3}^*F_{\mu,n_3}\otimes I_{r})\\
	    &\quad +\jj \Circ(\A_{\jj})\cdot (F_{\mu_{\jj},n_3}^*F_{\mu,n_3}\otimes I_{r})+\kk \Circ(\A_{\kk})\cdot (F_{\mu_{\kk},n_3}^*F_{\mu,n_3}\otimes I_{r}) \big)\cdot \unfold(\B))\\
	    &=\diag(\hat{\A})\cdot \sqrt{n_3}(F_{\mu,n_3}\otimes I_{r})\cdot \unfold(\B)\\
	    &=\diag(\hat{\A})\cdot \unfold(\hat{\B}).
	\end{align*}
Thus, the proof is completed.
\end{proof}

We next to discuss the group-theoretical property of gQt-product in the following theorem. At first, we introduce the concepts of {\it identity quaternion tensor} and {\it inverse quaternion tensor}.
\begin{definition}
    The $n\times n\times l$  identity  quaternion tensor $\mathcal{I}_{nnl}$ is the tensor whose first frontal slice is the identity matrix and  others are all zeros.
\end{definition}
\begin{definition}
    An $n\times n\times l$ quaternion tensor $\A$ is said to be invertible if there exists a quaternion tensor $\B\in\Q^{n\times n\times l}$ such that
    $$\A\pro\B=\mathcal{I}_{nnl}=\B\pro\A.$$
The tensor $\B$ is called the inverse of $\A$, denoted as $\A^{-1}$.
\end{definition}
\begin{theorem}
The collection of all invertible $n\times n\times l$ quaternion tensors forms a group under the ``$\pro$" operation given  in Definition \ref{defpro}.
\end{theorem}
\begin{proof}
    It is easy to see that
    \begin{equation*}
        \diag(\hat{\mathcal{I}}_{nnl})=I_{nl}\in\mathbb{R}^{nl\times nl}.
    \end{equation*}
    For all $\A\in\Q^{n\times n\times l}$, it follows from  Theorem \ref{mainthm} that
    \begin{equation*}
        \diag(\hat{\A})\diag(\hat{\mathcal{I}}_{nnl})=\diag(\hat{\mathcal{I}}_{nnl})\diag(\hat{\A})=\diag(\hat{\A}) \quad \Leftrightarrow \quad \A\pro\mathcal{I}_{nnl}=\mathcal{I}_{nnl}\pro\A=\A,
    \end{equation*}
which implies that  $\mathcal{I}_{nnl}$ is the identical-element in group.

    Similarly, for all $\A,\B,\C\in\Q^{n\times n\times l}$, it holds that
    \begin{equation*}
        \big( \diag(\hat{\A})\diag(\hat{\B})\big) \diag(\hat{\C})= \diag(\hat{\A})\big( \diag(\hat{\B}) \diag(\hat{\C})\big) \quad \Rightarrow\quad (\A\pro\B)\pro\C=\A\pro(\B\pro\C).
    \end{equation*}
    So, the ``$\pro$" operation is associative. Hence, the collection of all invertible $n\times n\times l$ quaternion tensors forms a group under the ``$\pro$" operation.
\end{proof}

We can easily check that the collection of all invertible $n\times n\times l$ quaternion tensors forms a ring under  standard tensor addition and gQt-product.

\subsection{gQt-SVD: gQt-product based SVD of third-order quaternion tensor }
\noindent
\par
Our goal in this subsection is to build  SVD of third-order quaternion tensor based on gQt-product. To begin with, we introduce the concepts of {\it conjugate transpose} of third-order quaternion tensor and {\it unitary quaternion tensor}, which will be used in the sequent.

For real third-order tensors,  the definition of {\it conjugate transpose} was given in \cite[Definition 3.14]{tproduct}, and we recall it in Definition \ref{def314}.
\begin{definition}\label{def314}
    Let $\A\in\mathbb{R}^{n_1\times n_2\times n_3}$, then $\A^*\in \mathbb{R}^{n_2\times n_1\times n_3}$, the conjugate transpose of $\A$, is obtained by transposing each of the frontal slices and then reversing the order of transposed frontal slices 2 through $n_3$.
\end{definition}

\begin{example}{\rm
Let $\A\in\mathbb{R}^{n_1\times n_2\times 4}$ and its frontal slices be given by the $n_1\times n_2$ matrices $A^{(1)},A^{(2)},A^{(3)},A^{(4)}$. Then, the conjugate transpose of $\A$ is
\begin{equation*}
    \A^*=\fold\left( \begin{bmatrix}
		(A^{(1)})^{*} \\
		(A^{(4)})^{*} \\
	(A^{(3)})^{*} \\
	(A^{(2)})^{*}	
	\end{bmatrix}\right).
\end{equation*}}
\end{example}

For any $\A\in\mathbb{R}^{n_1\times n_2\times n_3}$, it is seen that
    \begin{equation}\label{eqcirc1}
        \Circ(\A)^*=\Circ(\A^*).
    \end{equation}

For any given $\A\in\Q^{n_1\times n_2\times n_3}$, by \eqref{agqt}, it can be written as $\A=\A_{\bf{e}}+\A_{\ii}\ii+\A_{\jj}\jj+\A_{\kk}\kk$ where $\A_{\bf{e}}, \A_{\ii}, \A_{\jj}$ and $\A_{\kk}$ are $n_1\times n_2\times n_3$ real tensors. Hence, based on Definition \ref{def314}, we can define the conjugate transpose of the quaternion tensor $\A$.
\begin{definition}
    The conjugate transpose of a quaternion tensor $\A=\A_{\bf{e}}+\A_{\ii}\ii+\A_{\jj}\jj+\A_{\kk}\kk\in\Q^{n_1\times n_2\times n_3}$ is also denoted as $\A^*\in\Q^{n_2\times n_1 \times n_3}$, which is defined by
    $$\unfold(\A^*)=\unfold(\A_{\bf{e}}^*)-(T_{\ii}^*\otimes I_{n_2})\unfold(\A_{\ii}^{*})\ii-(T_{\jj}^*\otimes I_{n_2})\unfold(\A_{\jj}^{*})\jj-(T_{\kk}^*\otimes I_{n_2})\unfold(\A_{\kk}^{*})\kk ,$$
    where $T_{\ii}^*=F_{\mu,n_3}^*F_{\mu_{\ii},n_3}$, $T_{\jj}^*=F_{\mu, n_3}^*F_{\mu_{\jj},n_3}$, and $T_{\kk}^*=F_{\mu, n_3}^*F_{\mu_{\kk},n_3}$.
\end{definition}
\begin{example}{\rm
    Let $\A\in\Q^{n_1\times n_2\times 3}$ and its frontal slices be given by the matrices $A^{(1)},A^{(2)},A^{(3)}\in\Q^{n_1\times n_2}$. Setting $\mu=\ii$, the conjugate transpose of $\A$ is given as
    \begin{equation*}
    \A^*=\fold\left( \begin{bmatrix}
		(A^{(1)}_{\bf{e}})^*-\ii (A^{(1)}_{\ii})^{*}-\jj (A^{(1)}_{\jj})^*-\kk (A^{(1)}_{\kk})^{*} \\
		(A^{(3)}_{\bf{e}})^*-\ii (A^{(3)}_{\ii})^{*}-\jj (A^{(2)}_{\jj})^*-\kk (A^{(2)}_{\kk})^{*} \\
    	(A^{(2)}_{\bf{e}})^*-\ii (A^{(2)}_{\ii})^{*}-\jj (A^{(3)}_{\jj})^*-\kk (A^{(3)}_{\kk})^{*} \\
	\end{bmatrix}\right).
\end{equation*}}
\end{example}

We give some properties of the conjugate transpose of a quaternion tensor $\A$.
\begin{theorem}\label{lemma32}
For $\A\in\Q^{n_1\times n_2\times n_3}$, we have $\diag\big(\fft(\A^*)\big)=\diag(\hat{\A})^*$.
\end{theorem}
\begin{proof}
 By \eqref{eqcirc1}, \eqref{eqdiagA} and Lemma \ref{lemma21}, we have
    \begin{align}
        \diag(\hat{\A})^*&=\Big( (F_{\mu,n_3}\otimes I_{n_1})\cdot \big( \Circ(\A_{\bf{e}})\cdot (F_{\mu, n_3}^*\otimes I_{n_2})+ \ii \Circ(\A_{\ii})\cdot (F_{\mu_{\ii}, n_3}^*\otimes I_{n_2}) \nonumber \\
	    &\quad +\jj \Circ(\A_{\jj})\cdot (F_{\mu_{\jj}, n_3}^*\otimes I_{n_2})+\kk \Circ(\A_{\kk})\cdot (F_{\mu_{\kk}, n_3}^*\otimes I_{n_2}) \big) \Big)^* \nonumber \\
	    &=(F_{\mu,n_3}\otimes I_{n_2})\cdot  \Circ(\A_{\bf{e}}^*)\cdot (F_{\mu, n_3}^*\otimes I_{n_1})-(F_{\mu_{\ii}, n_3}^*\otimes I_{n_2})\cdot \Circ(\A_{\ii}^*)(F_{\mu_{\ii}, n_3}^*\otimes I_{n_1})\ii \nonumber \\
	    &\quad -(F_{\mu_{\jj}, n_3}^*\otimes I_{n_2})\cdot \Circ(\A_{\jj}^*)(F_{\mu_{\jj}, n_3}^*\otimes I_{n_1})\jj -(F_{\mu_{\kk}, n_3}^*\otimes I_{n_2})\cdot \Circ(\A_{\kk}^*)(F_{\mu_{\kk}, n_3}^*\otimes I_{n_1})\kk. \label{eqdiaghatA}
    \end{align}

    It follows from \eqref{eqdiag} and \eqref{eqfft1} that
    \begin{align*}
        (F_{\mu_{\ii},n_3}\otimes I_{n_2})\cdot \Circ(\A_{\ii}^*)\cdot (F_{\mu_{\ii}, n_3}^*\otimes I_{n_2})\ii &= \diag\Big( \fold\big( \sqrt{n_3}(F_{\mu_{\ii}, n_3}\otimes I_{n_2})\cdot \unfold(\A_{\ii}^*) \big) \Big)\ii \\
        &=\diag\Big( \fold\big( \sqrt{n_3}(F_{\mu, n_3}F_{\mu, n_3}^*F_{\mu_{\ii}, n_3}\otimes I_{n_2})\cdot \unfold(\A_{\ii}^*)\ii \big) \Big)  \\
        &=\diag\bigg( \fft\Big( \fold\big( (F_{\mu, n_3}^*F_{\mu_{\ii}, n_3}\otimes I_{n_2})\cdot \unfold(\A_{\ii}^*)\ii \big) \Big) \bigg).
    \end{align*}
    Similarly,
    \begin{align*}
        (F_{\mu_{\jj},n_3}\otimes I_{n_2})\cdot \Circ(\A_{\jj}^*)\cdot (F_{\mu_{\jj}, n_3}^*\otimes I_{n_2})\jj &=\diag\bigg( \fft\Big( \fold\big( (F_{\mu, n_3}^*F_{\mu_{\jj}, n_3}\otimes I_{n_2})\cdot \unfold(\A_{\jj}^*)\jj \big) \Big) \bigg),  \\
        (F_{\mu_{\kk},n_3}\otimes I_{n_2})\cdot \Circ(\A_{\kk}^*)\cdot (F_{\mu_{\kk}, n_3}^*\otimes I_{n_2})\kk &=\diag\bigg( \fft\Big( \fold\big( (F_{\mu, n_3}^*F_{\mu_{\kk}, n_3}\otimes I_{n_2})\cdot \unfold(\A_{\kk}^*)\kk \big) \Big) \bigg),
    \end{align*}
 which, together with \eqref{eqdiaghatA}, imply that
    \begin{align*}
        \diag(\hat{\A})^*&=\diag\big( \fft(\A^*) \big).
    \end{align*}
    Thus, the proof is completed.
\end{proof}

\begin{corollary}\label{coro32}
    For any two quaternion tensors $\A,\B$ with adequate dimensions, we have
    $$(\A\pro\B)^*=\B^*\pro\A^*.$$
\end{corollary}
\begin{proof}
    Since QDFT is invertible, the equality $(\A\pro\B)^*=\B^*\pro\A^*$ can be written as $$\diag\Big( \fft\big( (\A\pro\B)^* \big) \Big)=\diag\big( \fft( \B^*\pro\A^* ) \big).$$
    By Theorems \ref{mainthm} and \ref{lemma32}, we have
    $$(\A\pro\B)^*=\B^*\pro\A^*~~~\Leftrightarrow\ \big(\diag(\hat{\A})\diag(\hat{\B})\big)^*=\diag(\hat{\B})^*\diag(\hat{\A})^*.$$
    This, together with Lemma \ref{lemma21}, completes the proof.
\end{proof}

We next to introduce the concepts of unitary quaternion tensor and partially unitary quaternion tensor. Some properties of unitary quaternion tensor are given, which are still true for partially unitary quaternion tensor.

\begin{definition}
    The $n\times n\times l$ quaternion tensor $\mathcal{U}$ is unitary if $\mathcal{U}^*\pro\mathcal{U}=\mathcal{U}\pro\mathcal{U}^*=\mathcal{I}_{nnl}$. In addition, $\mathcal{U}\in\Q^{m\times n\times l}$ is said to be partially unitary if $\mathcal{U}^*\pro\mathcal{U}=\mathcal{I}_{nnl}$.
\end{definition}

 By Theorem \ref{lemma32}, we immediately obtain the following result.
\begin{corollary}\label{coro31}
    A quaternion tensor $\mathcal{U}\in\Q^{n\times n\times n_3}$ is unitary  if and only if $\diag(\hat{\mathcal{U}})$ is a unitary matrix.
\end{corollary}

A nice feature of unitary quaternion tensor is to preserve the Frobenius norm.
\begin{theorem}\label{lemma33}
    Let  $\mathcal{U}$ be a unitary tensor and $\A\in\Q^{n_1\times n_2\times n_3}$ be a quaternion tensor with adequate dimensions. Then,
    $$\| \mathcal{U}\pro\A \|_F=\| \A \|_F.$$
\end{theorem}
\begin{proof}
    Since normalized QDFT will not change the Frobenius norm, we have
    \begin{equation}\label{QFrobeniuseq}
        \|\mathcal{U}\pro\A\|_{F}^{2}=\sum_{p,q}\| (\mathcal{U}\pro\A)(p,q,:) \|_{F}^{2}=\frac{1}{n_3}\sum_{p,q}\| \fft(\mathcal{U}\pro\A)(p,q,:) \|_{F}^{2}=\frac{1}{n_3}\|\diag\big(\fft(\mathcal{U}\pro\A)\big) \|_{F}^{2} .
    \end{equation}
    By Theorems \ref{mainthm} and \ref{lemma32}, Corollary \ref{coro32}, and \eqref{tracerepre}, we can obtain
    \begin{align*}
        \|\mathcal{U}\pro\A\|_{F}^{2}&=\frac{1}{n_3}\Tr\Big(\diag\big(\fft(\mathcal{U}\pro\A))^*\cdot \diag(\fft(\mathcal{U}\pro\A)\big)\Big) \\
        &=\frac{1}{n_3}\Tr\Big(\diag\big(\fft((\mathcal{U}\pro\A)^*)\big)\cdot \diag\big(\fft(\mathcal{U}\pro\A)\big)\Big) \\
        &=\frac{1}{n_3}\Tr\Big( \big(\diag(\hat{\A})\big)^*\big(\diag(\hat{\mathcal{U}})\big)^* \diag(\hat{\mathcal{U}}) \diag(\hat{\A}) \Big) \\
        &=\frac{1}{n_3}\| \diag(\hat{\A}) \|_F^2=\| \A \|_F^2.
    \end{align*}
    Thus, we complete the proof.
\end{proof}

We say a tensor is ``f-diagonal" if each frontal slice is diagonal, and then the following decomposition of quaternion tensor is proposed.
\begin{theorem}[{\bf gQt-SVD}]\label{thmqtsvd}
Any third-order quaternion tensor $\A\in\Q^{n_1\times n_2\times n_3}$  can be factorized as
    \begin{equation*}
        \A=\mathcal{U}\pro\mathcal{S}\pro\mathcal{V}^*,
    \end{equation*}
    where $\mathcal{S}\in\Q^{n_1\times n_2\times n_3}$ is an f-diagonal tensor, $\mathcal{U}\in\Q^{n_1\times n_1\times n_3}$ and $\mathcal{V}\in\Q^{n_2\times n_2\times n_3}$ are unitary. Moreover, $\|\A\|_F=\|\mathcal{S}\|_F$.
\end{theorem}
\begin{proof}
    With the quaternion matrix SVD in Theorem \ref{thm21}, we have
    \begin{equation*}
        \diag(\hat{\A})= \diag(\hat{\mathcal{U}})\diag(\hat{\mathcal{S}})\diag(\hat{\mathcal{V}})^*,
    \end{equation*}
    where $\diag(\hat{\mathcal{U}})$ and $\diag(\hat{\mathcal{V}})^*$ are unitary, and $\diag(\hat{\mathcal{S}})$ is diagonal.

    Let $\mathcal{U},\mathcal{S},\mathcal{V}$ be the quaternion tensors corresponding to $\diag(\hat{\mathcal{U}}),\ \diag(\hat{\mathcal{S}}),\ \diag(\hat{\mathcal{V}})$, respectively. That is, $$\mathcal{U}\doteq \fold\big( (F_{\mu,n_3}^*\otimes I_{n_1})\frac{1}{\sqrt{n_3}}\diag(\hat{\mathcal{U}})(e\otimes I_{n_1})\big),$$
    $$\mathcal{V}\doteq \fold\big( (F_{\mu,n_3}^*\otimes I_{n_2})\frac{1}{\sqrt{n_3}}\diag(\hat{\mathcal{V}})(e\otimes I_{n_2})\big),$$
    $$\mathcal{S}\doteq \fold\big( (F_{\mu,n_3}^*\otimes I_{n_1})\frac{1}{\sqrt{n_3}}\diag(\hat{\mathcal{S}})(e\otimes I_{n_2})\big),$$
    where $e$ is an $n_3$-dimensional column vector whose all elements are $1$.
    Clearly, $\mathcal{U},\mathcal{V}$ are unitary, and $\mathcal{S}$ is an f-diagonal tensor. It follows from Theorem \ref{mainthm} and Corollary \ref{coro31} that  $\A=\mathcal{U}\pro\mathcal{S}\pro\mathcal{V}^*$. By Theorem \ref{lemma33}, we immediately obtain $\|\A\|_F=\|\mathcal{S}\|_F$.
\end{proof}

By Theorem \ref{thmqtsvd}, any third-order quaternion tensor has gQt-SVD. So, in the next subsection we will define its rank based on such decomposition, and show the existence of  low-rank optimal approximation.

\subsection{gQt-rank, low-rank optimal approximation, and multi-gQt-rank}
\noindent
\par
For any third-order quaternion tensor $\A\in\Q^{n_1\times n_2\times n_3}$ and its gQt-SVD $\A=\mathcal{U}\pro\mathcal{S}\pro\mathcal{V}^*$, we regard $\mathcal{U}(:,i,:)\in \Q^{n_1\times 1\times n_3}$,  $\mathcal{V}(:,i,:)\in \Q^{n_2\times 1\times n_3}$, and $\mathcal{S}(i,i,:)\in \Q^{1\times 1\times n_3}$ as tensors. By QDFT \eqref{qdftu}, we have
$$\fft(\mathcal{U}(:,i,:))=\hat{\mathcal{U}}(:,i,:),\quad  \fft(\mathcal{S}(i,i,:))=\hat{\mathcal{S}}(i,i,:),\quad \fft(\mathcal{V}(:,i,:))=\hat{\mathcal{V}}(:,i,:).$$
It follows from Theorem \ref{thm21} that
\begin{equation*}
    \diag(\hat{\A})=\sum_{i=1}^{\min(n_1,n_2)}\diag\big(\fft(\mathcal{U}(:,i,:))\big)\diag\big(\fft(\mathcal{S}(i,i,:))\big)\diag\big(\fft(\mathcal{V}(:,i,:))\big)^*.
\end{equation*}
With simple calculation, we can get
\begin{equation}\label{approach}
    \A=\sum_{i=1}^m\mathcal{U}(:,i,:)\pro\mathcal{S}(i,i,:)\pro\mathcal{V}(:,i,:)^*,\quad m\doteq\min(n_1,n_2).
\end{equation}
Thus, $\A$ can be written as a finite sum of gQt-product of matrices. Naturally, according to \eqref{approach}, we can define the ``gQt-rank" for any third-order quaternion  tensor.

\begin{definition}[{\bf gQt-rank}]\label{defrank}
    Let $\A\in\Q^{n_1\times n_2\times n_3}$ and its gQt-SVD be given as $\A=\mathcal{U}\pro\mathcal{S}\pro\mathcal{V}^{*}$.  The number of nonzero elements of $\{\mathcal{S}(i,i,:)\}_{i=1}^{m}$ is called gQt-rank of $\A$, denote as $\rank_{gQt}(\A)$. That is,
    \begin{equation*}
        \rank_{gQt}(\A)\doteq\#\{i|\ \mathcal{S}(i,i,:)\neq 0\}=\#\{i|\ \hat{\mathcal{S}}(i,i,:)\neq 0\}.
    \end{equation*}
The $i$-th  singular value of $\A$ is defined as $$\sigma_i(\A)\doteq\frac{1}{n_3}\|\hat{S}(i,i,:)\|_1, \quad i\in[m],$$ and the nuclear norm of $\A$ is defined as $$\|\A\|_*\doteq \sum_{i=1}^m\sigma_i(\A).$$
\end{definition}

Similar to Lemma \ref{lemman22}, we have the following result for the nuclear norm $\|\A\|_*$.
\begin{lemma}\label{lemma34}
    Let $\A\in\Q^{n_1\times n_2\times n_3}$, then
    $$\|\A\|_*=\frac{1}{n_3}\sum_{i=1}^{n_3}\|\hat{\A}(:,:,i)\|_*=\min_{\X,\Y}\left\{\frac{1}{2}(\|\X\|_F^2+\|\Y\|_F^2)~:~~\A=\X\pro\Y^*\right\}.$$
\end{lemma}
\begin{proof}
    From Definition \ref{defrank}, we have
    \begin{align*}
        \|\A\|_*&=\sum_{j=}^{m}\sum_{i=}^{n_3}|\hat{\mathcal{S}}(j,j,i)|=\frac{1}{n_3}\sum_{i=}^{n_3}\|\hat{\A}(:,:,i)\|_* \\ &=\min_{\hat{\X},\hat{\Y}}\left\{\frac{1}{2n_3}\sum_{i=1}^{n_3}(\|\hat{\X}(:,:,i)\|_F^2+\|\hat{\Y}(:,:,i)\|_F^2)~:~\hat{\X}(:,:,i)(\hat{\Y}(:,:,i))^*=\hat{\A}(:,:,i), i\in[n_3]\right\} \\
        &=\min_{\X,\Y}\left\{\frac{1}{2}(\|\X\|_F^2+\|\Y\|_F^2)~:~~\A=\X\pro\Y^*\right\}.
    \end{align*}
\end{proof}

By \eqref{approach}, the following theorem shows the existence of  low-rank ($\rank_{gQt}(\A)=k<m$) optimal approximation.
\begin{theorem}
    Let $\A\in\Q^{n_1\times n_2\times n_3}$ and its gQt-SVD be given as $\A=\mathcal{U}\pro\mathcal{S}\pro\mathcal{V}^{*}$. For $k<m$, denote
    \begin{equation}\label{thm3}
        \A_k=\sum_{i=1}^{k}\mathcal{U}(:,i,:)\pro\mathcal{S}(i,i,:)\pro\mathcal{V}(:,i,:)^*.
    \end{equation}
    Then,
     $$\A_k\in \arg\min_{{\mathcal{C}}\in M_{k}}\| \A-{\mathcal{C}} \|_{F},\quad \text{where $M_{k}=\{ \mathcal{C}|\ \mathcal{X}\in \Q^{n_1\times k\times n_3}, \mathcal{Y}\in\Q^{k\times n_2\times n_3}, ~\mathcal{C}=\mathcal{X}\pro\mathcal{Y} \}$.}$$
\end{theorem}
\begin{proof}
    For all $\mathcal{C}\in M_{k}$, it is easy to see that the rank of each block matrix of $\diag(\hat{\mathcal{C}})$ is not greater than $k$. Since there exists rank-$k$ optimal approximation of quaternion matrix,  it follows that
    \begin{align*}
        \|\A-\mathcal{C}\|_{F}^{2}&=\frac{1}{n_3}\|\diag(\hat{\A})-\diag(\hat{\mathcal{C}}) \|_{F}^{2}\\
        &=\frac{1}{n_3}\sum_{q=1}^{n_3}\| \hat{\A}(:,:,q)-\hat{\mathcal{C}}(:,:,q) \|_{F}^{2}  \\
        &\geq \frac{1}{n_3}\sum_{q=1}^{n_3}\|\hat{\A}(:,:,q)- \sum_{p=1}^{k}\hat{\mathcal{U}}(:,p,q)\hat{\mathcal{S}}(p,p,q)\hat{\mathcal{V}}(:,p,q)^* \|_{F}^{2} \\
        &=\frac{1}{n_3}\sum_{q=1}^{n_3}\|\hat{\A}(:,:,q)- \hat{\A_{k}}(:,:,q) \|_{F}^{2} =\| \A-\A_{k} \|_{F}^{2},
    \end{align*}
    where the first equality is from \eqref{QFrobeniuseq} and the last two equalities hold due to \eqref{thm3} and Theorem \ref{mainthm}. Thus, this completes the proof.
\end{proof}

We now  investigate gQt-rank of third order quaternion tensor generated by color video, and show that the low gQt-rankness is actually an inherent property of many color videos. We select a video data ``News" in widely used YUV Video Sequences\footnote{\url{http://trace.eas.asu.edu/yuv/}}. We take all 300 frames of size 288 $\times$ 352  as a color video data $\tilde{\C}$, i.e., $\tilde{\C}\in\Q^{288\times 352\times 300}$, where we encode the red, green, and blue channel pixel values on the three imaginary parts of a quaternion. Figure \ref{fig:singularshow1} illustrates the low-rank structure of $\tilde{\C}$ from color videos.

\begin{figure}[H]
	\centering
	\subfigure[{\small Sampled frames in video}]{
		\begin{minipage}[t]{0.4\linewidth}
			\centering
			\includegraphics[width=1\linewidth]{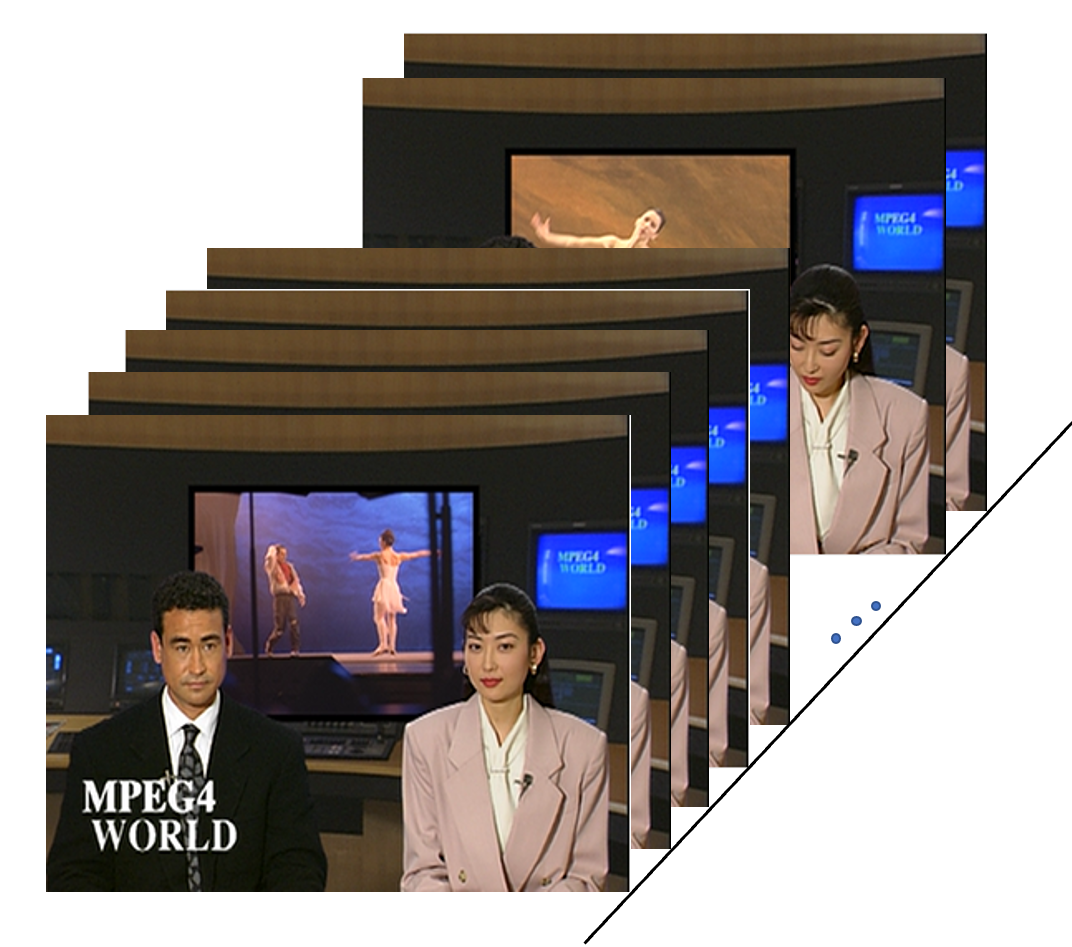}
		\end{minipage}
	}
	\subfigure[{\small Singular values of $\tilde{\C}$}]{
		\begin{minipage}[t]{0.5\linewidth}
			\centering
			\includegraphics[width=1\linewidth]{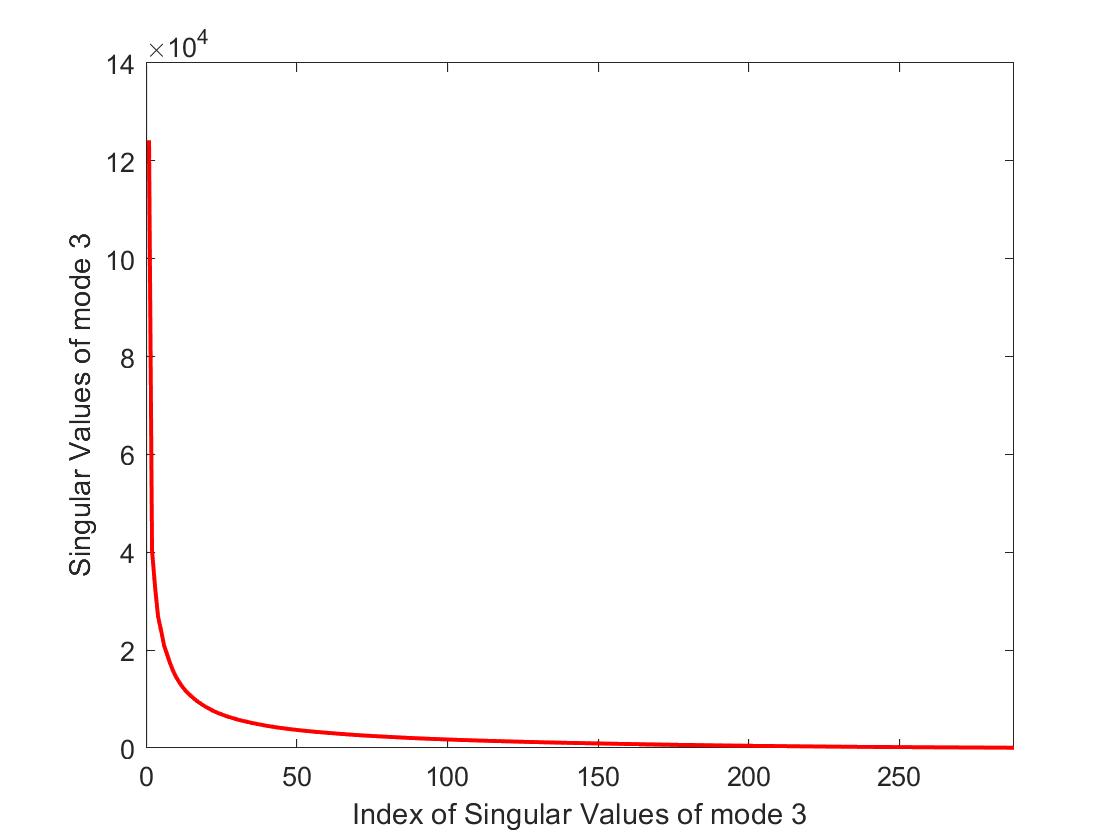}
		\end{minipage}
	}
\vspace{-0.1cm}
    \caption{{\small The sampled frames in video and singular values.}}
    \label{fig:singularshow1}
\end{figure}

Note that gQt-rank is only defined on one mode-$3$ of third order quaternion tensor, and the low-rank structure on the other two modes is missing. Motivated by this and {\it multi-tubal rank} given \cite{multirank}, we next to introduce ``multi-gQt-rank" for third-order quaternion tensor $\A\in\Q^{n_1\times n_2\times n_3}$ by extending QDFT \eqref{qdftu} from mode-$3$ to the other two modes.

To begin with some notations, we denote
\begin{equation*}
    A_{1}^{(i)}=\A(i,:,:),\quad  A_{2}^{(j)}=\A(:,j,:),\quad A_{3}^{(k)}=\A(:,:,k),\quad i\in [n_1],\ j\in [n_2],\ k\in [n_3].
\end{equation*}
Define QDFT of $\A$ along $w$-th mode ($w=1,2,3$) as $\hat{\A}_{1}$, $\hat{\A}_{2}$ and $\hat{\A}_{3}$,  which satisfy
\begin{align*}
    \hat{\A}_{1}(:,j,k)=F_{\mu,n_1}\A(:,j,k),\quad \hat{\A}_{2}(i,:,k)=F_{\mu,n_2}\A(i,:,k),\quad \hat{\A}_{3}(i,j,:)=F_{\mu,n_3}\A(i,j,:),
\end{align*}
for $i\in [n_1],\ j\in [n_2],\ k\in [n_3]$. Here, $F_{\mu,n_w}$ is defined similarly to \eqref{qdftu}. For simplicity, we define
\begin{align*}
    \hat{A}_{1}^{(i)}\doteq\hat{\A}_{1}(i,:,:),\quad \hat{A}_{2}^{(j)}\doteq\hat{\A}_{2}(:,j,:),\quad \hat{A}_{3}^{(k)}\doteq\hat{\A}_{3}(:,:,k),\quad i\in [n_1],\ j\in [n_2],\ k\in [n_3].
\end{align*}
We also define the following operators for $w=1,2,3$,
\begin{equation*}
    \diag_{w}(\A)\doteq\begin{bmatrix}
    A_{w}^{(1)} &  &  &  \\
		 & A_{w}^{(2)} &  &  \\
		 &  & \ddots &  \\
		 &  &  & A_{w}^{(n_u)}
    \end{bmatrix}, \quad
	\Circ_{w}(\mathcal{A})\doteq\begin{bmatrix}
		A_{w}^{(1)} & A_{w}^{(n_3)} & \cdots & A_{w}^{(2)} \\
		A_{w}^{(2)} & A_{w}^{(1)} & \cdots & A_{w}^{(3)} \\
		\vdots & \vdots & \ddots & \vdots \\
		A_{w}^{(n_3)} & A_{w}^{(n_3-1)} & \cdots & A_{w}^{(1)}
	\end{bmatrix},\end{equation*}
and \begin{align*}
    \unfold_{w}(\A) \doteq [A_{w}^{(1)}; A_{w}^{(2)}; \dots; A_{w}^{(n_w)}].
\end{align*}
The inverse operator ``$\fold_{w}$'' is defined by $\fold_{w}(\unfold_{w}(\A))=\A.$ With $\mu_{\ii},\mu_{\jj},\mu_{\kk}$ defined in \eqref{kernel2}, we set
$$T_{\ii,w}=F_{\mu_{\ii},n_w}^*F_{\mu,n_w},\quad   T_{\jj,w}=F_{\mu_{\jj},n_w}^*F_{\mu, n_w},\quad  T_{\kk,w}=F_{\mu_{\kk},n_w}^*F_{\mu, n_w}.$$

We now generalize gQt-product along three modes. For $\A\in\Q^{n_1\times n_2\times r}$ and $\B\in\Q^{n_1\times r\times n_3}$, define
	\begin{equation*}
		\A\proone\B\doteq \fold\big( (\Circ_{1}(\A_{\bf{e}})+\ii \Circ_{1}(\A_{\ii})\cdot (T_{\ii,1}\otimes I_{r})+\jj \Circ_{1}(\A_{\jj})\cdot (T_{\jj,1}\otimes I_{r})+\kk \Circ_{1}(\A_{\kk})\cdot (T_{\kk,1}\otimes I_{r}) )\cdot \unfold_{1}(\B)\big).
	\end{equation*}
For $\A\in\Q^{n_1\times n_2\times r}$ and $\B\in\Q^{r\times n_2\times n_3}$, define
	\begin{equation*}
		\A\protwo\B\doteq \fold\big( (\Circ_{2}(\A_{\bf{e}})+\ii \Circ_{2}(\A_{\ii})\cdot (T_{\ii,2}\otimes I_{r})+\jj \Circ_{2}(\A_{\jj})\cdot (T_{\jj,2}\otimes I_{r})+\kk \Circ_{2}(\A_{\kk})\cdot (T_{\kk,2}\otimes I_{r}) )\cdot \unfold_{2}(\B)\big).
	\end{equation*}
For $\A\in\Q^{n_1\times r\times n_3}$ and $\B\in\Q^{r\times n_2\times n_3}$, define
	\begin{equation*}
		\A\prothree\B\doteq \fold\big( (\Circ_{3}(\A_{\bf{e}})+\ii \Circ_{3}(\A_{\ii})\cdot (T_{\ii,3}\otimes I_{r})+\jj \Circ_{3}(\A_{\jj})\cdot (T_{\jj,3}\otimes I_{r})+\kk \Circ_{3}(\A_{\kk})\cdot (T_{\kk,3}\otimes I_{r}) )\cdot \unfold_{3}(\B)\big).
	\end{equation*}

We introduce the concept of ``multi-gQt-rank" for third-order quaternion tensor as follows.
\begin{definition}[{\bf multi-gQt-rank}]
    Let $\A\in\Q^{n_1\times n_2\times n_3}$ and $r_w^l=\rank(\hat{A}_w^{(l)})$ with $l\in[n_w]$ and $w\in[3]$. The multi-gQt-rank of $\A$ is defined as
    $$\rank_{mgQt}(\A)=(r_1(\A),r_2(\A),r_3(\A)),$$
    where $r_w(\A)=\max(r_w^1,r_w^2,\dots,r_w^{n_w})$.
\end{definition}
By Theorem \ref{mainthm} and Lemma \ref{lemma34}, we can easily get the following results.
\begin{theorem}\label{lemma51}
    Let $\A,\B$ be quaternion tensors and $\C=\A*_{\mu}^{w}\B$ be defined above for $w\in[3]$. If their QDFT tensors along mode-$w$ are $\hat{\A}_{w},\hat{\B}_{w}$, and $\hat{\mathcal{C}}_{w}$, respectively, then, it holds\\
    $\emph{(i)}$ $\|\A\|_F^2=\frac{1}{n_w}\|\hat{\A}_w\|_F^2$.\\
    $\emph{(ii)}$ $\diag_{w}(\hat{\mathcal{C}}_{w})=\diag_{w}(\hat{\A}_{w})\cdot \diag_{w}(\hat{\B}_{w}).$\\
    $\emph{(iii)}$  $r_w(\C)\leq \min( r_w(\A),r_w(\B) )$.
\end{theorem}

We now  investigate multi-gQt-rank of third order quaternion tensor generated by color video, and show that the low rankness is actually an inherent property of many color videos. We select a video data ``Coastguard" in YUV Video Sequences\footnote{\url{http://trace.eas.asu.edu/yuv/}} to group into the quaternion tensor $\tilde{\C}\in\Q^{288\times 352\times 300}$. Figure \ref{fig:singularshow2} shows the low-rank structures of tensor $\tilde{\C}$ in mode 1, 2 and 3.
\begin{figure}[H]
	\centering
	\subfigure[{\small Sampled frames in video}]{
		\begin{minipage}[t]{0.45\linewidth}
			\centering
			\includegraphics[width=1\linewidth]{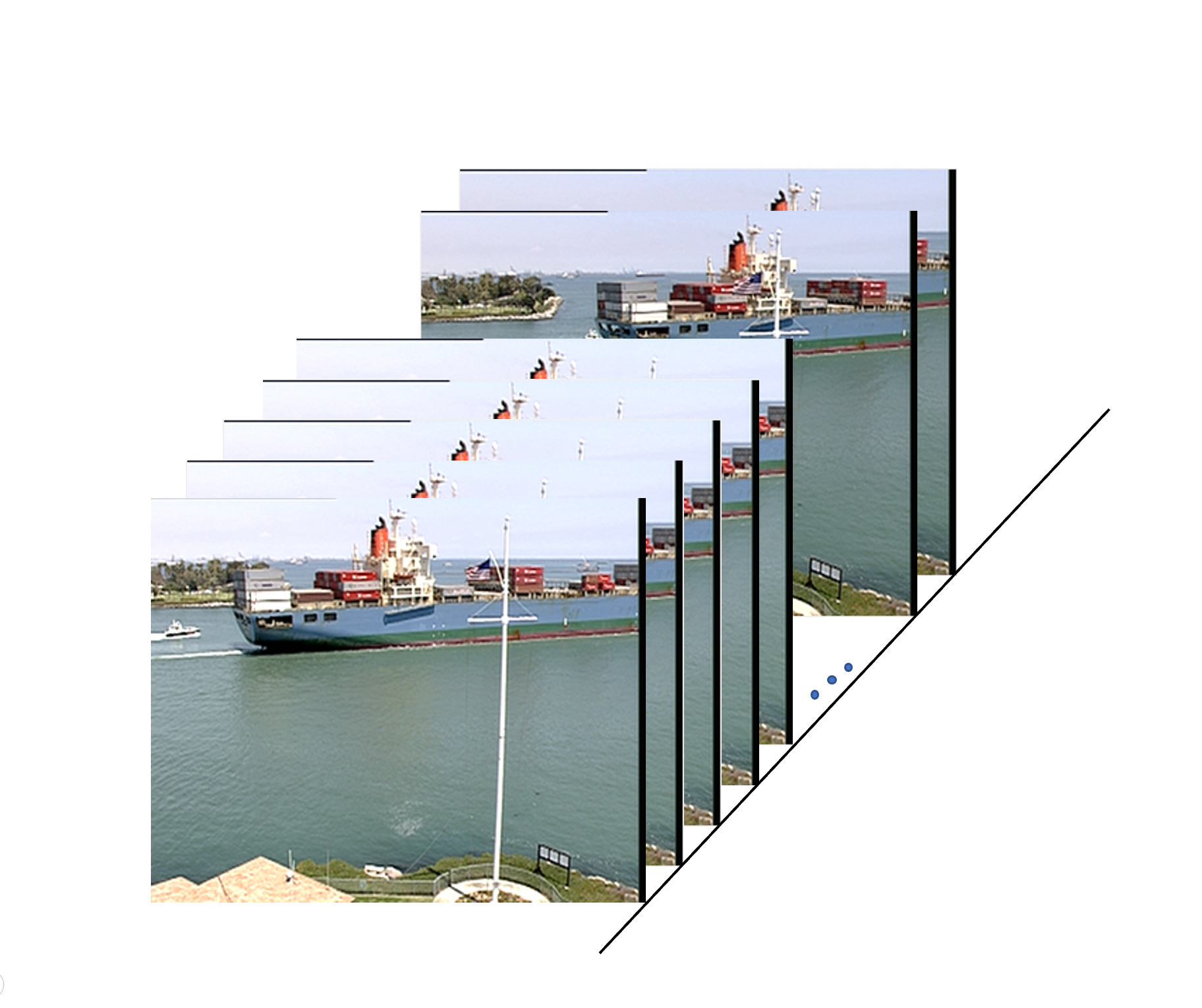}
		\end{minipage}
	}
	\subfigure[{\small Singular values of $\tilde{\C}$ in mode 1}]{
		\begin{minipage}[t]{0.45\linewidth}
			\centering
			\includegraphics[width=1\linewidth]{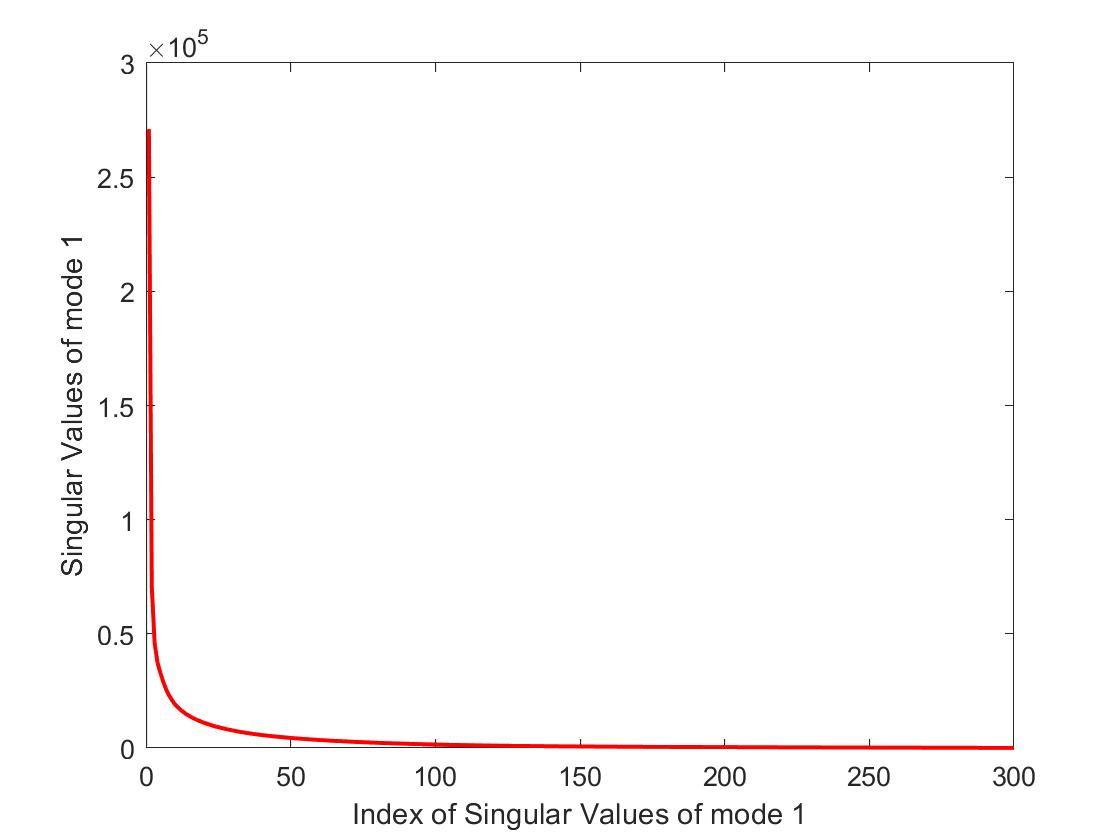}
		\end{minipage}
	}
	
	\subfigure[{\small Singular values of $\tilde{\C}$ in mode 2}]{
		\begin{minipage}[t]{0.45\linewidth}
			\centering
			\includegraphics[width=1\linewidth]{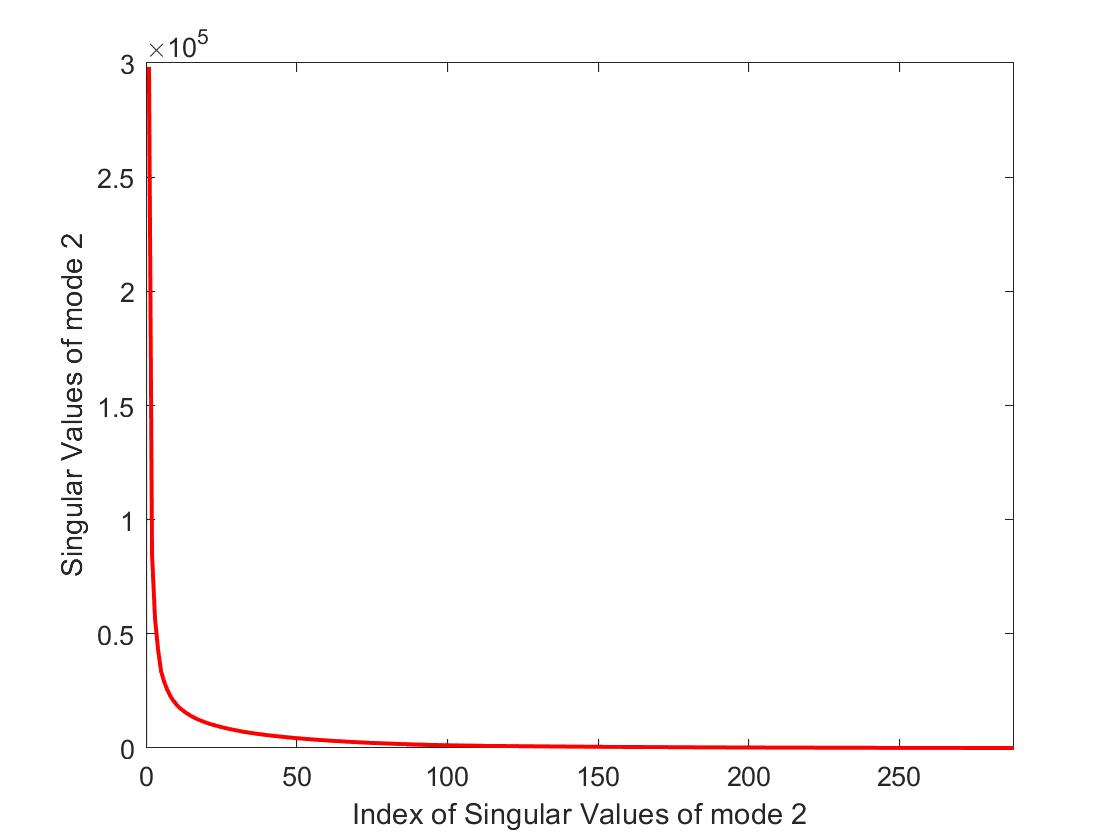}
		\end{minipage}
	}
	\subfigure[{\small Singular values of $\tilde{\C}$ in mode 3}]{
		\begin{minipage}[t]{0.45\linewidth}
			\centering
			\includegraphics[width=1\linewidth]{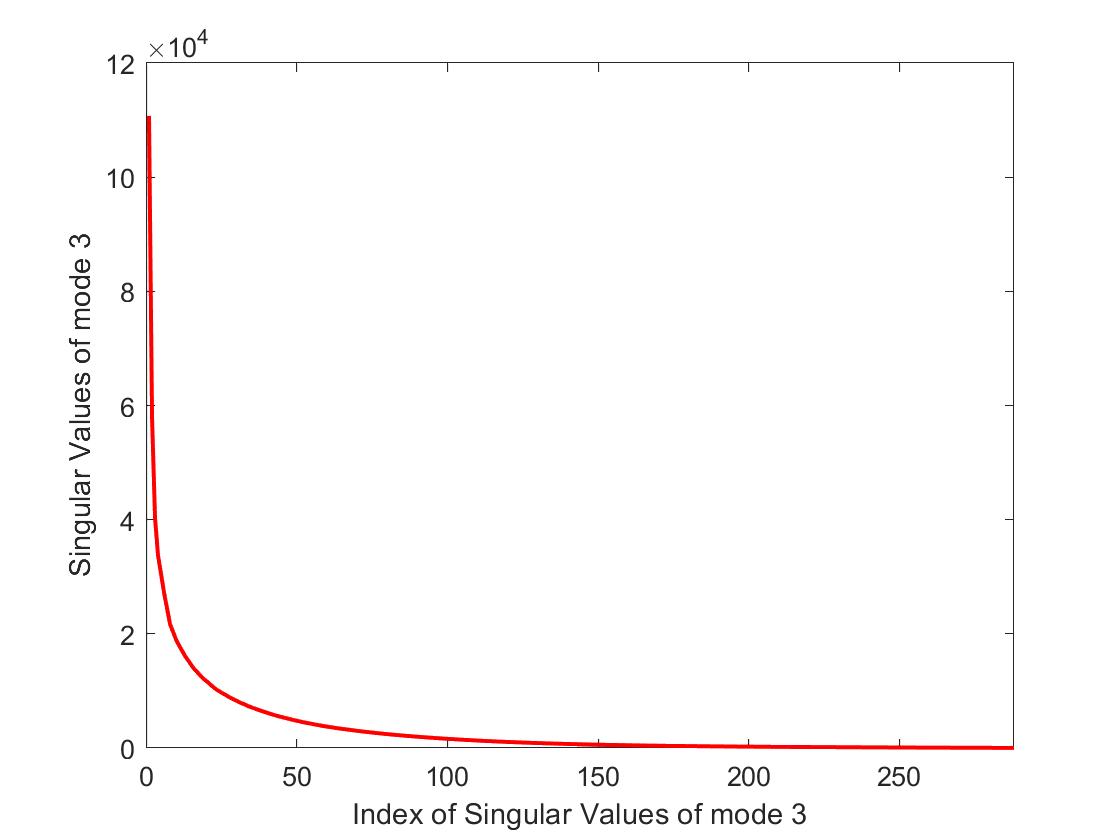}
		\end{minipage}
	}

    \caption{{\small The sampled frames in video and singular values.}}
    \label{fig:singularshow2}
\end{figure}

As mentioned in \cite{PZhou2018,MiaoJi2019,song2021}, when the color image or video data is regarded as quaternion matrices or real tensors, they lie on a union of low-rank subspaces approximately, which lead the low-rank structure of the real data. This is also true for third-order quaternion tensors data. Figures \ref{fig:singularshow1} and \ref{fig:singularshow2} indicate that third-order quaternion tensors generated by color videos in real life have an inherent low-rank property. Hence, in order to recover color videos with partial data loss, we can design tensor completion model as \eqref{sec1model1} via gQt-rqnk and multi-gQt-rank. This is the aim of next section.

\section{Quaternion Tensor Completion for Color Video Inpainting}\label{sec4}
\noindent
\par
In this section, we establish low-rank quaternion tensor completion models based on gQt-rank and multi-gQt-rank to recover color videos with partial data loss. We present an ALS algorithm to solve the proposed models and show that the generated sequence converges to the stationary point of our model.

\subsection{Low gQt-rank quaternion tensor completion model}
\noindent
\par
We first define the operator $\Re$ to get the real part of a quaternion tensor, and the operator $\Im$ to get the imaginary part. Our low gQt-rank quaternion tensor completion is to find the minimal gQt-rank solution satisfying the consistency with the observed data. Let $\mathcal{M}$ be the raw tensor and $\Omega$ be the index set locating the observed data. Then, by \eqref{sec1model1}, the low gQt-rank quaternion tensor completion can be modeled as
\begin{equation}\label{rankrc0}
    \min_{\C\in\Q^{n_1\times n_2\times n_3}}\rank_{gQt}(\C), \quad \text{s.t.}~\ P_{\Omega}(\C-\mathcal{M})=0,\ \Re(\C)=0,
\end{equation}
To solve \eqref{rankrc0} efficiently, the following relaxation via nuclear norm is considered
\begin{equation*}
    \min_{\C\in\Q^{n_1\times n_2\times n_3}}\|\C\|_*, \quad \text{s.t.}~\  P_{\Omega}(\C-\mathcal{M})=0,\ \Re(\C)=0.
\end{equation*}
By Theorem \ref{thmqtsvd}, computing $\|\C\|_*$ is via gQt-SVD of $\C$. However, its computational complexity will be  $O(n_1n_2n_3(\log(n_3)+\min(n_1,n_2)))$. In order to reduce the computational cost, by  Lemma \ref{lemma34}, we consider the following quaternion tensor factorization model:
\begin{equation*}
    \min_{\C, \A, \B}\frac{1}{2}(\|\A\|_F^2+\|\B\|_F^2), \quad \text{s.t.}~\ \C=\A\pro\B,\  P_{\Omega}(\C-\mathcal{M})=0,\ \Re(\C)=0.
\end{equation*}

In practical applications, each frame of the video data has spatial stability feature. We use total variation (TV) to capture these spatial correlation features and consider the square of total variation of data to keep objective function analytic, i.e.,
$$\|\C\times_{1}H_{n_1}\|_F^2+\|\C\times_{2}H_{n_2}\|_F^2=\sum_{k=1}^{n_1-1}\|\C(k,:,:)-\C(k+1,:,:)\|_F^2+\sum_{k=1}^{n_2-1}\|\C(:,k,:)-\C(:,k+1,:)\|_F^2,$$
where $H=\text{Toeplitz}(0,-1,1)$ be an $(n-1)\times n$ Toeplitz matrix, i.e., $$H_{n}=\begin{bmatrix}
		1 & -1 & 0 & \cdots & \cdots & 0 \\
		0 & 1 & -1 & \ddots & \ddots & 0 \\
		0 & 0 & 1 & \ddots & \ddots & 0 \\
		\vdots & \ddots & \ddots & \ddots & \ddots & \ddots \\
		0 &\cdots & \cdots & \cdots & 1 & -1
	\end{bmatrix}\in \mathbb{R}^{(n-1)\times n}.$$
In addition, the real color video quaternion tensor $\C$ is only approximately low gQt-rank, and hence it is likely to fail  to find a low gQt-rank solution strictly satisfying the restriction $\C=\A\pro\B$. Therefore, we usually penalize $\C=\A\pro\B$ into objective function. Thus, we get our final low gQt-rank quaternion tensor completion model as follows:
\begin{align}\label{modelcompute}
    &\min_{\C, \A, \B}\quad f(\C,\A,\B)\doteq\frac{1}{2}\|\A\pro\B-\C\|_F^2+\frac{\lambda}{2}(\|\A\|_F^2+\|\B\|_F^2)+\sum_{k=1}^{2}\lambda_k\|\C\times_{k}H_{n_k}\|_F^2,\nonumber \\
    &~~\text{s.t.}\quad P_{\Omega}(\C-\mathcal{M})=0,\ \Re(\C)=0,
\end{align}
where $\lambda$, $\lambda_1$ and $\lambda_2$ are the penalty parameters.

We will propose an algorithm for solving the model \eqref{modelcompute} in the next subsection.

\subsection{Solution method}\label{subsecsolving}
\noindent
\par
In this subsection, we present an ALS procedure to solve \eqref{modelcompute}. At each iteration, two variables of $\A$, $\B$, $\C$ are fixed and the other one is updated by solving the updated model \eqref{modelcompute}.

At the $t$-th iteration of our method, $\C^{t}$ is updated by
\begin{equation}\label{updatec}
    \C^{t}=\mathop{\arg\min}_{P_{\Omega}(\C-\mathcal{M})=0,\Re(\C)=0}\frac{1}{2}\|\A^{t-1}\pro\B^{t-1}-\C\|_F^2+\sum_{k=1}^{2}\lambda_k\|\C\times_{k}H_{n_k}\|_F^2,
\end{equation}
and $\A^t, \B^t$ are updated by the regularized version of \eqref{modelcompute} as follows:
\begin{align}
    \A^t=\mathop{\arg\min}\ &\frac{1}{2}\|\A\pro\B^{t-1}-\C^{t}\|_F^2+\frac{\lambda}{2}\|\A\|_F^2+\frac{\beta}{2}\|\A-\A^{t-1}\|_F^2, \label{updatea} \\
    \B^t=\mathop{\arg\min}\ &\frac{1}{2}\|\A^{t}\pro\B-\C^{t}\|_F^2+\frac{\lambda}{2}\|\B\|_F^2+\frac{\beta}{2}\|\B-\B^{t-1}\|_F^2, \label{updateb}
\end{align}
where $\beta>0$ is the regularization parameter.

We next to solve the subproblems \eqref{updatec}-\eqref{updateb}. First, we rewrite \eqref{updatec} as
\begin{equation}\label{updatec1}
    \C^{t}=\mathop{\arg\min}_{\Re(\C)=0}\frac{1}{2}\|\A^{t-1}\pro\B^{t-1}-\C\|_F^2+\sum_{k=1}^{2}\lambda_k\|\C\times_{k}H_{n_k}\|_F^2+\delta_{\{ P_{\Omega}(\C-\mathcal{M})=0\} }(\C),
\end{equation}
where $\delta_S(\cdot)$ is the indicator function, i.e., $\delta_S(x)=0$, if $x\in S$; $\delta_S(x)=+\infty$, if $x\notin S$.

For simplicity, set $\C=\C_{\ii}\ii+\C_{\jj}\jj+\C_{\kk}\kk$ whose real part is zero. Thus, the objective function in \eqref{modelcompute} can be written as $f(\C_{\ii},\C_{\jj},\C_{\kk}|\hat{\A},\hat{\B})$ due to $\sqrt{n_3}\|\A\|_F=\|\hat{\A}\|_F$. Denote
\begin{align}
    f_{1}(\C_{\ii},\C_{\jj},\C_{\kk})&=\frac{1}{2}\|\Im(\A^{t-1}\pro\B^{t-1})-\C\|_F^2+\sum_{k=1}^{2}\lambda_k\|\C\times_{k}H_{n_k}\|_F^2, \nonumber \\
    h_{1}(\C_{\ii},\C_{\jj},\C_{\kk})&=\delta_{\{ P_{\Omega}(\C-\mathcal{M})=0\} }(\C). \label{h1eq}
\end{align}
Then, \eqref{updatec1} is equivalent to the following unconstrained optimization problem:
\begin{align}\label{updatec2}
    \min_{\C_{\ii},\C_{\jj},\C_{\kk}}f_{1}(\C_{\ii},\C_{\jj},\C_{\kk})+h_{1}(\C_{\ii},\C_{\jj},\C_{\kk}).
\end{align}
Clearly, $f_1$ is differentiable and $h_1$ is a proper closed convex function. So, we can solve \eqref{updatec2} inexactly by the well-known proximal gradient  method (PGM). We employ the following PGM
with Barzilar-Borwein \cite{BB} line research rule to solve \eqref{updatec2}.
\begin{algorithm}
\caption{BB-PGM for \eqref{updatec2}}\label{alg:2}
\begin{algorithmic}
\STATE {{\bf Input.} The tensor data $\mathcal{M},\A^{t-1}\pro\B^{t-1}$, the observed set $\Omega$, initial step size $\alpha_0$, parameters $\lambda,\lambda_i,H_{n_i},1\leq i\leq 2$.}
\STATE {{\bf Step 0.} Initialize $z^0=[\C_{\ii}^0,\C_{\jj}^0,\C_{\kk}^0]$ which satisfies $P_{\Omega}(\C_{\ii}^0\ii+\C_{\jj}^0\jj+\C_{\kk}^0\kk-\mathcal{M})=0$, $z^{-1}=z^0$. Iterate the following steps for $k=0,1,2,\cdots$ while it does not satisfy stop criterion.}
\STATE {{\bf Step 1.} If $k\geq 1$, choose step size $$\alpha_k=\frac{\|z^{k}-z^{k-1}\|_F^2}{\langle z^{k}-z^{k-1}, \nabla f_1(z^{k})-\nabla f_1(z^{k-1})\rangle},\ \text{or } \alpha_k=\frac{\langle z^{k}-z^{k-1}, \nabla f_1(z^{k})-\nabla f_1(z^{k-1})\rangle}{\|\nabla f_1(z^{k})-\nabla f_1(z^{k-1})\|_F^2}.$$}
\STATE {{\bf Step 2.} Update $[\C_{\ii}^{k+1},\C_{\jj}^{k+1},\C_{\kk}^{k+1}]=\prox_{\alpha_kh_1}\big([\C_{\ii}^{k},\C_{\jj}^{k},\C_{\kk}^{k}]-\alpha_k\nabla f_1(\C_{\ii}^{k},\C_{\jj}^{k},\C_{\kk}^{k})\big)$ .}
\STATE {{\bf Output.} $\C_{\ii}^{k+1}\ii+\C_{\jj}^{k+1}\jj+\C_{\kk}^{k+1}\kk$.}
\end{algorithmic}
\end{algorithm}

Here, $\prox_{\alpha_kh_1}(\cdot)$ in Step 2 of Algorithm \ref{alg:2}  is the proximal mapping of $h_1$ with parameter $\alpha_k>0$, i.e.,
$$\prox_{\alpha_kh_1}(\X)=\mathop{\arg\min}_{\Y\in \mathbb{R}^{n_1\times n_2\times n_3\times 3}}\left\{ \alpha_kh_1(\Y)+\frac{1}{2}\|\Y-\X\|_F^2 \right\}, \ \forall \X\in\mathbb{R}^{n_1\times n_2\times n_3\times 3}.$$
It is known that the proximal mapping of indicator function is a projection, i.e.,
$$\prox_{\alpha_kh_1}(\X)=P_{\{ P_{\Omega}(\X_{\ii}\ii+\X_{\jj}\jj+\X_{\kk}\kk-\mathcal{M})=0\}}(\X),~~ \ \forall \X=[\X_{\ii},\X_{\jj},\X_{\kk}]\in\mathbb{R}^{n_1\times n_2\times n_3\times 3}.$$

For solving \eqref{updatea} and \eqref{updateb}, we consider their matrix versions because the updates  of $\A^t$ and $\B^t$ are just those of their QDFT tensors from the calculation perspective. Hence, by \eqref{Qfrobeeq1}, we can rewrite  \eqref{updatea} and \eqref{updateb} as the following corresponding matrix versions:
\begin{align}\label{restrictauu}
    \hat{\A}^t=\mathop{\arg\min}_{\hat{\A}}\sum_{l=1}^{n_3}\left(\frac{1}{2}\|\hat{\A}^{(l)}\hat{\B}^{t-1,(l)}-\hat{C}^{t,(l)}\|_F^2+\frac{\lambda}{2}\|\hat{\A}^{(l)}\|_F^2+\frac{\beta}{2}\|\hat{\A}^{(l)}-\hat{\A}^{t-1,(l)}\|_F^2 \right),
\end{align}
and
\begin{align}\label{restrictbuu}
    \hat{\B}^{t}&=\mathop{\arg\min}_{\hat{\B}}\sum_{l=1}^{n_3}\left(\frac{1}{2}\|\hat{\A}^{t,(l)}\hat{\B}^{(l)}-\hat{\C}^{t,(l)}\|_F^2+\frac{\lambda}{2}\|\hat{\B}^{(l)}\|_F^2+
    \frac{\beta}{2}\|\hat{\B}^{(l)}-\hat{\B}^{t-1,(l)}\|_F^2\right). \end{align}

To solve the problems \eqref{restrictauu} and \eqref{restrictbuu} with quaternion variables, we apply the following results, which were given in \cite{Yannan2020} to introduce the gradient for a quaternion matrix function and optimality condition for an equality-constrained quaternion matrix optimization.
\begin{definition}[\cite{Yannan2020} Definition 4.1]
    Let $f: \Q^{m\times n}\rightarrow \mathbb{R}$ and $X=X_{\ee}+X_{\ii}\ii+X_{\jj}\jj+X_{\kk}\kk$. $f$ is said to be differentiable at $X$ if $\frac{\partial f}{\partial X_{v}}$ exists at $X_v$ for $v=\ee,\ii,\jj,\kk$. Moreover, its gradient is defined as
    \begin{equation*}
        \nabla_{\Q} f(X)=\frac{\partial f}{\partial X_{\ee}}+\frac{\partial f}{\partial X_{\ii}}\ii+\frac{\partial f}{\partial X_{\jj}}\jj+\frac{\partial f}{\partial X_{\kk}}\kk.
    \end{equation*}
   $f$ is said to be continuously differentiable at $X$ if $\frac{\partial f}{\partial X_{v}}$ exists in a neighborhood of $X_v$ and is continuous at $X_v$ for $v=\ee,\ii,\jj,\kk$. Furthermore, $f$ is said to be continuously differentiable if $f$ is continuously differentiable at any $X\in\Q^{m\times n}$.
\end{definition}

\begin{theorem}[\cite{Yannan2020} Theorem 4.2] \label{Qprogramcondition}
    Suppose that $f: \Q^{m\times n}\rightarrow \mathbb{R}$ is continuously differentiable, and $X^{\#}\in\Q^{m\times n}$ is an optimal solution of $\min\{ f(X)\}$. Then, it holds
        $$\nabla_{\Q}f(X^{\#})=O.$$
\end{theorem}

By Theorem \ref{Qprogramcondition}, we can find the closed-form solutions to \eqref{updatea} and \eqref{updateb}. For $l\in[n_3]$,  $\hat{\A}^{t,(l)}$ is updated as
\begin{align}
    \hat{\A}^{t,(l)}&=\mathop{\arg\min}_{\hat{\A}}\frac{1}{2}\|\hat{\A}^{(l)}\hat{\B}^{t-1,(l)}-\hat{\C}^{t,(l)}\|_F^2+\frac{\lambda}{2}\|\hat{\A}^{(l)}\|_F^2+\frac{\beta}{2}\|\hat{\A}^{(l)}-\hat{\A}^{t-1,(l)}\|_F^2 \nonumber \\
    &=\big(\hat{\C}^{t,(l)}(\hat{\B}^{t-1,(l)})^*+\beta \hat{\A}^{t-1,(l)}\big)\big(\hat{\B}^{t-1,(l)}(\hat{\B}^{t-1,(l)})^*+(\lambda+\beta)I\big)^{-1}, \label{newupdatea}
\end{align}
and $\hat{\B}^{t,(l)}$ is updated as
\begin{align}
    \hat{\B}^{t,(l)}&=\mathop{\arg\min}_{\hat{\B}}\frac{1}{2}\|\hat{\A}^{t,(l)}\hat{\B}^{(l)}-\hat{\C}^{t,(l)}\|_F^2+\frac{\lambda}{2}\|\hat{\B}^{(l)}\|_F^2+\frac{\beta}{2}\|\hat{\B}^{(l)}-\hat{\B}^{t-1,(l)}\|_F^2 \nonumber \\
    &=\big((\hat{\A}^{t,(l)})^*\hat{\A}^{t,(l)}+(\lambda+\beta)I\big)^{-1}\big((\hat{\A}^{t,(l)})^*\hat{\C}^{t,(l)}+\beta \hat{\B}^{t-1,(l)}\big). \label{newupdateb}
\end{align}

Denote $\Omega^c$ as the complement of the set $\Omega$. Based on above discussions, we propose the following algorithm  to solve our model \eqref{modelcompute}.

\begin{algorithm}
\caption{gQt-Rank Tensor Completion (QRTC)}\label{alg:1}
\begin{algorithmic}
\STATE {{\bf Input.} The tensor data $\mathcal{M}\in\Q^{n_1\times n_2\times n_3}$, the observed set $\Omega$, the rank $\textbf{r}\in\mathbb{Z}_{+}^{n_3}$, parameters $\lambda,\lambda_i,H_{n_i},1\leq i\leq 3$ and $\epsilon$.}
\STATE {{\bf Step 0.} Initialize $\hat{\A}^0,\hat{\B}^0$ and $\C^0$ satisfying $P_{\Omega}(\C^0-\mathcal{M})=0,\ \Re(\C^0)=0$ and the rank of $\hat{\A}^0,\hat{\B}^0$ are less than $\textbf{r}$. Iterate the following steps for $t=1,2,\cdots$ while it does not satisfy stop criterion.}
\STATE {{\bf Step 1.} Compute
\begin{equation}\label{updatect}
	\C^{t}=\mathop{\arg\min}_{P_{\Omega}(\C-\mathcal{M})=0,\Re(\C)=0}\frac{1}{2}\|\A^{t-1}\pro\B^{t-1}-\C\|_F^2+\sum_{k=1}^{2}\lambda_k\|\C\times_{k}H_{n_k}\|_F^2+\sum_{\alpha=\ii,\jj,\kk}\langle \delta^t_{\alpha}, \C_{\alpha}\rangle
\end{equation}
via Algorithm \ref{alg:2}. That is, given parameters of Algorithm \ref{alg:2}, apply BB-PGM to find an approximate solution $\C^t$ of \eqref{updatec1} such that the error vector $\delta^t$ satisfies $\Re(\delta^t)=0,~P_{\Omega}(\delta^t)=0$ and the accuracy condition
\begin{equation}\label{deltaineq}
    \|P_{\Omega^c}(\delta^t)\|_F\leq \frac{1}{4}\|\C^t-\C^{t-1}\|_F.
\end{equation}
}
\STATE {{\bf Step 2.} Compute $\hat{\A}^t$ by  \eqref{newupdatea}.}
\STATE {{\bf Step 3.} Compute $\hat{\B}^t$ by  \eqref{newupdateb}.}
\STATE {{\bf Step 5.} Check the stop criterion: $|f(\C^{k},\A^{k},\B^{k})-f(\C^{k-1},\A^{k-1},\B^{k-1})|/|f(\C^{k-1},\A^{k-1},\B^{k-1})|<\epsilon$.}
\STATE {{\bf Output.} $\C^t$.}
\end{algorithmic}
\end{algorithm}
\subsection{Convergence analysis for QRTC}
\noindent
\par
We now analyze the convergence of QRTC. For convenience, we collect all variables as a real undetermined vector
$$z\doteq \big(\C_{\ii},\C_{\jj},\C_{\kk},\hat{\A}_{\ee},\hat{\A}_{\ii},\hat{\A}_{\jj},\hat{\A}_{\kk},\hat{\B}_{\ee},\hat{\B}_{\ii},\hat{\B}_{\jj},\hat{\B}_{\kk}\big)\in\mathbb{R}^{3n_1n_2n_3+4\|\textbf{r}\|_1(n_1+n_2)}.$$
And then \eqref{modelcompute} can be written as
\begin{align}\label{modelcompute2}
    \min_{z\in\Lambda}f(z),
\end{align}
where $$\Lambda=\{z\in\mathbb{R}^{3n_1n_2n_3+4\|\textbf{r}\|_1(n_1+n_2)}|\ P_{\Omega}(\C_{\ii}\ii+\C_{\jj}\jj+\C_{\kk}\kk-\mathcal{M})=0  \}.$$
Therefore, the projected gradient of $f$ at $z\in\Lambda$ is given as
\begin{equation*}
    \Pi_{\Lambda}(\nabla f(z))=\left(\begin{aligned}
    &\left[P_{\Omega^c}(\frac{\partial f(z)}{\partial \C_{\ii}});P_{\Omega^c}(\frac{\partial f(z)}{\partial \C_{\jj}});P_{\Omega^c}(\frac{\partial f(z)}{\partial \C_{\kk}})\right] \\
    &\qquad \left[\frac{\partial f(z)}{\partial \hat{\A}_{\ee}};\frac{\partial f(z)}{\partial \hat{\A}_{\ii}};\frac{\partial f(z)}{\partial \hat{\A}_{\jj}};\frac{\partial f(z)}{\partial \hat{\A}_{\kk}}\right] \\
    &\qquad \left[\frac{\partial f(z)}{\partial \hat{\B}_{\ee}};\frac{\partial f(z)}{\partial \hat{\B}_{\ii}};\frac{\partial f(z)}{\partial \hat{\B}_{\jj}};\frac{\partial f(z)}{\partial \hat{\B}_{\kk}}\right]
    \end{aligned}\right),
\end{equation*}
where $\Pi_{\Lambda}(\cdot)$ denotes the projection onto the feasible set $\Lambda$.

 \begin{definition}[{\bf stationary point}]
 The point $z^*\in\Lambda$ is said to be a stationary point of the low gQt-rank quaternion tensor completion model \eqref{modelcompute} if $\Pi_{\Lambda}(\nabla f(z^*))=0$. \end{definition}

The following theorems show that the sequence generated by Algorithm \ref{alg:1} is bounded and any accumulation point converges to a stationary point of \eqref{modelcompute}.

\begin{theorem}\label{thmconverge1}
	Let $\{ z^t \}$ be the sequence generated by QRTC. Then, there exists a constant $K_1$ such that
\begin{equation}\label{thm41eq}
f(z^t)-f(z^{t+1})\geq K_1\|z^t-z^{t+1}\|_2^2.
\end{equation}
Moreover, the sequence $\{ z^t \}$ is bounded.
\end{theorem}
\begin{proof}
	 Let $h_1$ be defined as \eqref{h1eq}. It is easy to see that $f(\C_{\ii},\C_{\jj},\C_{\kk}|\ \hat{\A},\hat{\B})+h_1(\C_{\ii},\C_{\jj},\C_{\kk})-\frac12\|P_{\Omega^c}(\Re(\C))\|_F^2$ is a  convex function. Since $h_1$ is an indicator function on affine space, we have
	$$\frac{\partial (f+h_1)}{\partial \C_{\alpha}}(\cdot)=P_{\Omega^c}\big(\frac{\partial f}{\partial \C_{\alpha}}(\cdot)\big),~~\ \alpha=\ii,\jj,\kk .$$
It follows from
	\begin{equation}\label{41thm2}
		h_1(\C_{\ii}^{t+1},\C_{\jj}^{t+1},\C_{\kk}^{t+1})=0,~\ h_1(\C_{\ii}^{t},\C_{\jj}^{t},\C_{\kk}^{t})=0,~\ \|\C^{t+1}-\C^{t}\|_F=\|P_{\Omega^c}(\C^{t+1}-\C^{t})\|_F,
	\end{equation}
and the convexity of function $f(\C_{\ii},\C_{\jj},\C_{\kk}|\ \hat{\A},\hat{\B})+h_1(\C_{\ii},\C_{\jj},\C_{\kk})-\frac12\|P_{\Omega^c}(\Re(\C))\|_F^2$ that
	\begin{align}
	f(\C^{t}| \hat{\A}^t,\hat{\B}^t)-f(\C^{t+1}|\hat{\A}^t,\hat{\B}^t)&=f(\C_{\ii}^{t},\C_{\jj}^{t},\C_{\kk}^{t}| \hat{\A}^t,\hat{\B}^t)-f(\C_{\ii}^{t+1},\C_{\jj}^{t+1},\C_{\kk}^{t+1}| \hat{\A}^t,\hat{\B}^t) \nonumber \\
	&\geq \sum_{\alpha=\ii,\jj,\kk} \Big( \left\langle \frac{\partial f+h_1}{\partial \C_{\alpha}}(\C_{\ii}^{t+1},\C_{\jj}^{t+1},\C_{\kk}^{t+1}| \hat{\A}^t,\hat{\B}^t),\C_{\alpha}^{t}-\C_{\alpha}^{t+1}\right\rangle +\frac{1}{2}\|\C_{\alpha}^{t}-\C_{\alpha}^{t+1}\|_F^2 \Big) \nonumber \\
	&=\sum_{\alpha=\ii,\jj,\kk}  \left\langle P_{\Omega^c}\big(\frac{\partial f}{\partial \C_{\alpha}}(\C_{\ii}^{t+1},\C_{\jj}^{t+1},\C_{\kk}^{t+1}| \hat{\A}^t,\hat{\B}^t)\big),P_{\Omega^c}(\C_{\alpha}^{t}-\C_{\alpha}^{t+1})\right\rangle \nonumber \\
	&\qquad +\frac{1}{2}\|\C^{t}-\C^{t+1}\|_F^2 \label{41thm1}
	\end{align}
	From \eqref{updatect},
	\begin{equation}\label{41thm7}
		P_{\Omega^c}(\delta^{t+1}_{\alpha})=P_{\Omega^c}\big(\frac{\partial f}{\partial \C_{\alpha}}(\C_{\ii}^{t+1},\C_{\jj}^{t+1},\C_{\kk}^{t+1}|\ \hat{\A}^t,\hat{\B}^t)\big),\ \alpha=\ii,\jj,\kk.
	\end{equation}
	With \eqref{41thm1}, we can obtain
	\begin{align}
		f(\C^{t}|\hat{\A}^t,\hat{\B}^t)-f(\C^{t+1}|\hat{\A}^t,\hat{\B}^t)&\geq \sum_{\alpha=\ii,\jj,\kk}\langle P_{\Omega^c}(\delta^{t+1}_{\alpha}), P_{\Omega^c}(\C_{\alpha}^{t}-\C_{\alpha}^{t+1})\rangle +\frac{1}{2}\|\C^{t}-\C^{t+1}\|_F^2 \nonumber \\
		&\geq -\|P_{\Omega^c}(\delta^{t+1})\|_F\|P_{\Omega^c}(\C^{t}-\C^{t+1})\|_F+\frac{1}{2}\|\C^{t}-\C^{t+1}\|_F^2 \nonumber \\
		&\geq \frac{1}{4}\|\C^{t}-\C^{t+1}\|_F^2, \label{41thm6}
	\end{align}
where the last inequality holds due to \eqref{41thm2} and \eqref{deltaineq}. Thus,
	\begin{align}
		f(\hat{\A}^{t}| \C^{t+1},\hat{\B}^t)-f(\hat{\A}^{t+1}| \C^{t+1},\hat{\B}^t)&=\frac{1}{n_3}\sum_{l=1}^{n_3}(\frac{1}{2}\|\hat{\A}^{t,(l)}\hat{\B}^{t,(l)}-\hat{C}^{t+1,(l)}\|_F^2+\frac{\lambda}{2}\|\hat{\A}^{t,(l)}\|_F^2 \nonumber \\
		&\qquad \quad -\frac{1}{2}\|\hat{\A}^{t+1,(l)}\hat{\B}^{t,(l)}-\hat{C}^{t+1,(l)}\|_F^2-\frac{\lambda}{2}\|\hat{\A}^{t+1,(l)}\|_F^2 ) \nonumber \\
		&=\frac{1}{n_3}\sum_{l=1}^{n_3}\bigg(\frac{1}{2}\|(\hat{\A}^{t,(l)}-\hat{\A}^{t+1,(l)})\hat{\B}^{t,(l)}\|_F^2+\frac{\lambda}{2}\|\hat{\A}^{t,(l)}-\hat{\A}^{t+1,(l)}\|_F^2 \nonumber \\
		&\qquad \quad \Re\Big(\Tr\big((\hat{\A}^{t+1,(l)}\hat{\B}^{t,(l)}-\hat{C}^{t+1,(l)})(\hat{\B}^{t,(l)})^*(\hat{\A}^{t,(l)}-\hat{\A}^{t+1,(l)})^*\big) \nonumber \\
		&\qquad \quad +\lambda \Tr\big( \hat{\A}^{t+1,(l)}(\hat{\A}^{t,(l)}-\hat{\A}^{t+1,(l)})^* \big) \Big)\bigg). \label{41thm3}
	\end{align}
	By Theorem \ref{Qprogramcondition} and \eqref{newupdatea}, it holds for any $l\in[n_3]$,
	\begin{equation*}
		(\hat{\A}^{t+1,(l)}\hat{\B}^{t,(l)}-\hat{C}^{t+1,(l)})(\hat{\B}^{t,(l)})^*+\lambda\hat{\A}^{t+1,(l)}=\beta (\hat{\A}^{t,(l)}-\hat{\A}^{t+1,(l)}),
	\end{equation*}
which, together with \eqref{41thm3}, implies
	\begin{align}
		f(\hat{\A}^{t}| \C^{t+1},\hat{\B}^t)-f(\hat{\A}^{t+1}| \C^{t+1},\hat{\B}^t)&=\frac{1}{n_3}\sum_{l=1}^{n_3}\Big(\frac{1}{2}\|(\hat{\A}^{t,(l)}-\hat{\A}^{t+1,(l)})\hat{\B}^{t,(l)}\|_F^2+(\frac{\lambda}{2}+\beta)\|\hat{\A}^{t,(l)}-\hat{\A}^{t+1,(l)}\|_F^2\Big) \nonumber \\
		&\geq \frac{1}{n_3}\sum_{l=1}^{n_3}(\frac{\lambda}{2}+\beta)\|\hat{\A}^{t,(l)}-\hat{\A}^{t+1,(l)}\|_F^2 \nonumber \\
		&=\frac{\lambda+2\beta}{2n_3}\|\hat{\A}^t-\hat{\A}^{t+1}\|_F^2. \label{41thm4}
	\end{align}
	Similarly, we have
		\begin{align}
		f(\hat{\B}^{t}| \C^{t+1},\hat{\A}^{t+1})-f(\hat{\B}^{t+1}| \C^{t+1},\hat{\A}^{t+1})\geq\frac{\lambda+2\beta}{2n_3}\|\hat{\B}^t-\hat{\B}^{t+1}\|_F^2. \label{41thm5}
	\end{align}
	Therefore, combining \eqref{41thm6}, \eqref{41thm4} and \eqref{41thm5}, we get
	\begin{equation*}
		f( \C^{t},\hat{\A}^{t},\hat{\B}^{t})-f( \C^{t+1},\hat{\A}^{t+1},\hat{\B}^{t+1})\geq \min\Big(\frac{1}{4},\frac{\lambda+2\beta}{2n_3}\Big)\|z^{t}-z^{t+1}\|_2^2.
	\end{equation*}
Taking $$K_1=\min\left(\frac{1}{4},\frac{\lambda}{2}+\frac{\lambda+2\beta}{2n_3}\right),$$ we prove that \eqref{thm41eq} holds and hence the sequence $\{ f(z^t) \}$ is monotonically decreasing.
	
It follows from $f\geq 0$ that
	$$\sum_{t=1}^{\infty}\big(f(z^t)-f(z^{t+1})\big)<\infty,\quad \sum_{t=1}^{\infty}\|z^t-z^{t+1}\|_2^2<\infty,\quad \lim_{t\rightarrow\infty}z^t-z^{t+1}=0.$$
	Since $$f(z^1)\geq f(z^t)\geq \frac{\lambda}{2}(\|\A^t\|_F^2+\|\B^t\|_F^2),$$ $\{ \A^t \}$, $\{ \B^t \}$ are bounded, and so $\{ \hat{\A}^t \}$, $\{ \hat{\B}^t \}$ are also bounded. Together with the fact that $$f(z^1)\geq f(z^t)\geq \|\A^t\pro\B^t-\C^t\|_F^2,$$ $\{ \C^t \}$ is also bounded, and hence $\{ z^t \}$ is bounded.
\end{proof}

\begin{theorem}\label{thmconverge11}
Let $\{ z^t \}$ be the sequence generated by QRTC. Then, there exists a constant $K_2$ such that
\begin{equation}\label{thm42eq}
K_2\|z^t-z^{t+1}\|_2\geq\|\Pi_{\Lambda}\big( \nabla f(z^t) \big)\|_F.
\end{equation}
Moreover, any accumulation point of $\{ z^t \}$ is a stationary point of \eqref{modelcompute}.
\end{theorem}
\begin{proof}	
	By Theorem \ref{thmconverge1}, $\{ z^t \}$ is bounded. So, there exists a compact convex set $Z$ such that $\{ z^t \}\subset Z$. Since $f$ is a quadratic polynomial in $Z$, its gradient is Lipschitz in $Z$ with the Lipschitz constant $L_f$, that is,
	$$\|\nabla f(z)-\nabla f(z')\|_2\leq L_f\|z-z'\|_2,\ \forall z,z'\in Z.$$
	By Theorem \ref{Qprogramcondition}, \eqref{newupdatea} and \eqref{newupdateb}, for any $l\in[n_3]$,  we have
	\begin{align*}
		\nabla_{\Q}f(\hat{\A}^{t+1,(l)}|\hat{\A}^{t+1,(-l)},\C^{t+1},\hat{\B}^{t})&=\frac{\beta}{n_3} (\hat{\A}^{t,(l)}-\hat{\A}^{t+1,(l)}), \\
		\nabla_{\Q}f(\hat{\B}^{t+1,(l)}|\hat{\B}^{t+1,(-l)},\C^{t+1},\hat{\A}^{t+1})&=\frac{\beta}{n_3} (\hat{\B}^{t,(l)}-\hat{\B}^{t+1,(l)}),
	\end{align*}
	where $\hat{\A}^{t+1,(-l)}$ and $\hat{\B}^{t+1,(-l)}$ denote $\hat{\A}^{t+1}$ and $\hat{\B}^{t+1}$ except $\hat{\A}^{t+1,(l)}$ and $\hat{\B}^{t+1,(l)}$, respectively. Set
	\begin{align*}
		\frac{\partial f}{\partial \hat{\A}}=\left[\frac{\partial f}{\partial \hat{\A}_{\ee}};\frac{\partial f}{\partial \hat{\A}_{\ii}};\frac{\partial f}{\partial \hat{\A}_{\jj}};\frac{\partial f}{\partial \hat{\A}_{\kk}}\right], \qquad \frac{\partial f}{\partial \hat{\B}}=\left[\frac{\partial f}{\partial \hat{\B}_{\ee}};\frac{\partial f}{\partial \hat{\B}_{\ii}};\frac{\partial f}{\partial \hat{\B}_{\jj}};\frac{\partial f}{\partial \hat{\B}_{\kk}}\right],
	\end{align*}
then,
	\begin{align}
		\|\frac{\partial f}{\partial \hat{\A}}(z^t)\|_F&\leq \|\frac{\partial f}{\partial \hat{\A}}(z^t)-\frac{\partial f}{\partial \hat{\A}}(\C^{t+1},\A^{t+1},\B^t)\|_F+\|\frac{\partial f}{\partial \hat{\A}}(\C^{t+1},\A^{t+1},\B^t)\|_F \nonumber \\
		&\leq L_f\|z^{t}-z^{t+1}\|_2+\sum_{l=1}^{n_3} \|\frac{\partial f}{\partial \hat{\A}^{(l)}}(\C^{t+1},\A^{t+1},\B^t)\|_F \nonumber \\
		&=L_f\|z^{t}-z^{t+1}\|_2+\sum_{l=1}^{n_3} \|\nabla_{\Q}f(\hat{\A}^{t+1,(l)}|\hat{\A}^{t+1,(-l)},\C^{t+1},\hat{\B}^{t})\|_F \nonumber \\
		&=L_f\|z^{t}-z^{t+1}\|_2+\sum_{l=1}^{n_3} \|\frac{\beta}{n_3} (\hat{\A}^{t,(l)}-\hat{\A}^{t+1,(l)})\|_F \nonumber \\
		&\leq (L_f+\frac{\beta}{n_3})\|z^{t}-z^{t+1}\|_2. \label{41thm81}
	\end{align}
	Similarly, we have
	\begin{equation}\label{41thm82}
		\|\frac{\partial f}{\partial \hat{\B}}(z^t)\|_F\leq (L_f+\frac{\beta}{n_3})\|z^{t}-z^{t+1}\|_2 .
	\end{equation}
It follows from \eqref{41thm7} that
	\begin{align}
		\sum_{\alpha=\ii,\jj,\kk}\| P_{\Omega^c}\frac{\partial f}{\partial \C_{\alpha}}(z^t)\|_F&\leq \sum_{\alpha=\ii,\jj,\kk}\Big(\| \big(P_{\Omega^c}\frac{\partial f}{\partial \C_{\alpha}}(z^t)-P_{\Omega^c}\frac{\partial f}{\partial \C_{\alpha}}(\C^{t+1},\hat{\A}^t,\hat{\B}^t)\big)\|_F+\|P_{\Omega^c}\frac{\partial f}{\partial \C_{\alpha}}(\C^{t+1},\hat{\A}^t,\hat{\B}^t)\|_F\Big) \nonumber \\
		&\leq \|\Pi_{\Lambda}\big( \nabla f(z^t) \big)-\Pi_{\Lambda}\big( \nabla f(\C^{t+1},\hat{\A}^t,\hat{\B}^t) \big)\|_F+\sum_{\alpha=\ii,\jj,\kk}\|P_{\Omega^c}(\delta_{\alpha}^{t+1})\|_F \nonumber \\
		&\leq L_f\|z^{t}-z^{t+1}\|_2+\frac{1}{4}\|\C^t-\C^{t+1}\|_F \nonumber \\
		&\leq (L_f+\frac{1}{4})\|z^{t}-z^{t+1}\|_2. \label{41thm83}
	\end{align}
	Combining \eqref{41thm81}, \eqref{41thm82} and \eqref{41thm83}, it is easy to see that \eqref{thm42eq} holds with $K_2=L_f+\max(\frac{1}{4},\frac{\beta}{n_3})$.
	
	Since $\{ z^t \}$ is bounded, there exists a convergent subsequence of $\{ z^t \}$. Without loss of generality, we assume that $\lim_{k\rightarrow\infty}z^{t_k}=z^*$. Then,
	\begin{align*}
		\|\Pi_{\Lambda}\big( \nabla f(z^*) \big)\|_F&\leq \|\Pi_{\Lambda}\big( \nabla f(z^*)-\nabla f(z^{t_k}) \big)\|_F+	\|\Pi_{\Lambda}\big( \nabla f(z^{t_k}) \big)\|_F \\
		&\leq L_f\|z^*-z^{t_k}\|_2+K_3\|z^{t_k}-z^{t_k+1}\|_2,
	\end{align*}
	which, together with taking limit $k\rightarrow\infty$ in the right hand side, shows $\Pi_{\Lambda}\big( \nabla f(z^*) \big)=0$, and hence $z^*$ is a stationary point of \eqref{modelcompute}.
\end{proof}

Theorems \ref{thmconverge1} and \ref{thmconverge11} show that the sequence $\{ z^t \}$ generated by QRTC is bounded and its any accumulation point is a stationary point of \eqref{modelcompute}. We next use the Kurdyka-{\L}ojasiewicz (KL) property \cite{KL1,KL2,KL3} to prove that $\{ z^t \}$ is convergent.
\begin{definition}[{\bf KL property}]\label{KLdf}
	Let $Z\in\mathbb{R}^n$ be an open set and $f : Z\rightarrow\mathbb{R}$ be a semi-algebraic function. For every critical point $z^*\in Z$ of $f$, there is a neighborhood $Z'\in Z$ of $z^*$, an exponent $\theta\in [0,1)$ and a positive constant $K_3$ such that
	\begin{equation}\label{klineq}
		|f(z)-f(z^*)|^{\theta}\leq K_3\|\Pi_{\Lambda}\big(\nabla f(z)\big)\|_F,\qquad \forall z\in Z'.
	\end{equation}
\end{definition}

It is obvious that $f(z)+\delta_{\Lambda}(z)$ is a semi-algebraic function. Then, from Definition \ref{KLdf} and \cite{KL3}, the KL inequality (\ref{klineq}) holds for the function $f(z)+\delta_{\Lambda}(z)$. Hence, for the given $\theta, K_3$  in Definition \ref{KLdf} and $K_1, K_2$ are defined as Theorems \ref{thmconverge1} and \ref{thmconverge11}, we have the following convergence results.

\begin{theorem}\label{thmconverge2}
Let $\{ z^t \}$ be the sequence generated by QRTC and $z^*$ be an accumulation point of $\{ z^t \}$. Assume $z^{0}\in B(z^*, \tau)\doteq \{ z|\ \|z-z^*\|_F<\tau \}\subset Z'$ with
	$$\tau>\frac{K_2 K_3}{K_1(1-\theta)}|f(z^0)-f(z^*)|^{1-\theta}+\|z^0-z^*\|_2,$$
then, $z^t\in B(z^*, \tau)$ for $t=0,1,2,\ldots$. Moreover,
\begin{equation}\label{resu2}\sum_{t=0}^{\infty}\|z^{t+1}-z^{t}\|_2\leq \frac{K_2 K_3}{K_1(1-\theta)}|f(z^0)-f(z^*)|^{1-\theta},\quad \lim_{t\to\infty}z^t=z^*.
\end{equation}
\end{theorem}
\begin{proof}
We show $\{z^t\}\subset B(z^*, \tau)$ by induction. When $t=0$, it holds obviously. Assume that $z^t\in B(z^*, \tau)$ holds for all $t\leq \hat{t}$, then KL property is true for $z^t$. Now we display that $z^t\in B(z^*, \tau)$ is true when $t=\hat{t}+1$. Let $$\phi(x)\doteq\frac{ K_3}{1-\theta}|x-f(z^*)|^{1-\theta},\quad x>f(z^*).$$Then $\phi(x)$ is concave and differentiable. Hence, we have
	$$\phi\big(f(z^t)\big)-\phi\big(f(z^{t+1})\big)\geq\phi^{'}\big(f(z^t)\big)\big( f(z^t)-f(z^{t+1}) \big)=\frac{K_3}{|f(z^t)-f(z^{*})|^{\theta}}\big( f(z^t)-f(z^{t+1}) \big).$$
By Definition \ref{KLdf}, \eqref{thm41eq} and \eqref{thm42eq}, we get
	\begin{equation*}
		\phi\big(f(z^t)\big)-\phi\big(f(z^{t+1})\big)\geq\frac{1}{\|\Pi_{\Lambda}\big( \nabla f(z^t) \big)\|_F}\big( f(z^t)-f(z^{t+1}) \big)\geq \frac{K_1}{K_2}\|z^{t}-z^{t+1}\|_2.
	\end{equation*}
	Therefore,
	\begin{align}\label{43thm1}
		\sum_{p=0}^{t}\|z^{p}-z^{p+1}\|_2\leq \frac{K_2}{K_1}\sum_{p=0}^{t}\Big( \phi\big(f(z^p)\big)-\phi\big(f(z^{p+1})\big) \Big)\leq\frac{K_2}{K_1}\phi\big(f(z^0)\big),
	\end{align}
which implies
	\begin{equation*}
		\|z^{\hat{t}+1}-z^*\|_2\leq\sum_{p=0}^{\hat{t}}\|z^{p+1}-z^p\|_2+\|z^{0}-z^*\|_2\leq\frac{K_2}{K_1}\phi\big(f(z^0)\big)+\|z^0-z^*\|_2\leq\tau .
	\end{equation*}
Thus, $\{z^t\}\subset B(z^*, \tau)$.
	
Taking $t\rightarrow\infty$ in \eqref{43thm1}, the first inequality in \eqref{resu2} is arrived. Without loss of generality, we assume that $\lim_{k\rightarrow\infty}z^{t_k}=z^*$. Then, for all $t>0$ and $t_{k+1}\geq t>t_{k}$,
	\begin{align*}
		\|z^t-z^*\|_2\leq \|z^{t_k}-z^*\|_2+\sum_{p=t_k}^{t-1}\|z^{p+1}-z^p\|_2,
	\end{align*}
which, together with the fact that  $\{ \|z^{t+1}-z^t\|_2 \}$ is Cauchy sequence, implies  $\lim_{t\to\infty}z^t=z^*$.
\end{proof}
\begin{theorem}
	Suppose that $\{ z^t \}$ is the sequence generated by QRTC and $z^*$ be its limit point. Then, the following statements hold.
	
	\emph{(i)} If $\theta\in( 0,\frac{1}{2} ]$, there exist $\gamma>0$ and $\xi\in(0,1)$ such that
	$$\|z^t-z^*\|_2\leq\gamma\xi^t.$$
	
	\emph{(ii)} If $\theta\in( \frac{1}{2},1)$, there exists $\gamma>0$ such that
	$$\|z^t-z^*\|_2\leq\gamma t^{-\frac{1-\theta}{2\theta-1}}.$$
\end{theorem}
\begin{proof}
	Since $\{ z^t \}$ converges to $z^*$, there exists an index $k_0$ such that $z^{k_0}\in B(z^*,\tau)$, where $\tau$ is given in Theorem \ref{thmconverge2}. Hence, we can regard $z^{k_0}$ as an initial point. Without loss of generality, we set $z^0\in B(z^*,\tau)$. Let
	\begin{equation}\label{44thm1}
		\Delta_{t}\doteq\sum_{p=t}^{\infty}\|z^p-z^{p+1}\|_2\geq\|z^t-z^*\|_2.
	\end{equation}
	It follows from \eqref{43thm1} that
	\begin{align*}
		\Delta_t\leq\frac{K_2}{K_1}\phi\big(f(z^t)\big)=\frac{K_2 K_3}{K_1(1-\theta)}|f(z^t)-f(z^*)|^{1-\theta}=\frac{K_2 K_3}{K_1(1-\theta)}\big(|f(z^t)-f(z^*)|^{\theta}\big)^{\frac{1-\theta}{\theta}}.
	\end{align*}
	Using KL inequality \eqref{klineq}, it holds
	\begin{align*}
		\Delta_t\leq\frac{K_2 K_3}{K_1(1-\theta)}\left( K_3\|\Pi_{\Lambda}\big( \nabla f(z^t) \big)\|_F\right)^{\frac{1-\theta}{\theta}}.
	\end{align*}
	Set $\xi_1=\frac{(K_2 K_3)^{\frac{1}{\theta}}}{K_1(1-\theta)}$, then the above inequality, together with \eqref{thm41eq} and \eqref{thm42eq}, implies
	\begin{equation}\label{44thm2}
		\Delta_t\leq\frac{K_2 K_3}{K_1(1-\theta)}\left( K_2K_3\|z^t-z^{t+1}\|_2\right)^{\frac{1-\theta}{\theta}}=\xi_1(\Delta_t-\Delta_{t+1})^{\frac{1-\theta}{\theta}}.
	\end{equation}
	
We now 	prove (i). If $\theta\in(0,\frac{1}{2}]$, then $\frac{1-\theta}{\theta}\geq 1$.  It holds for sufficiently large $t$,
	\begin{equation*}
		\Delta_{t}\leq\xi_1(\Delta_t-\Delta_{t+1}).
	\end{equation*}
	Hence,
	\begin{equation*}
		\Delta_{t+1}\leq\frac{\xi_1-1}{\xi_1},
	\end{equation*}
	which together with \eqref{44thm1} implies that (i) holds with $\xi=\frac{\xi_1-1}{\xi_1}$.
	
	We next to prove (ii). If $\theta\in(\frac{1}{2},1)$, let $q(x)=x^{-\frac{\theta}{1-\theta}}$. Then, $q(x)$ is monotonically decreasing on $x$. It follows from \eqref{44thm2} that
	\begin{equation*}
		\xi_1^{-\frac{\theta}{1-\theta}}\leq q(\Delta_{t})(\Delta_{t}-\Delta_{t+1})=\int_{\Delta_{t}}^{\Delta_{t+1}}q(\Delta_{t})dx\leq\int_{\Delta_{t}}^{\Delta_{t+1}}q(x)dx=-\frac{1-\theta}{2\theta-1}\left(\Delta_{t}^{-\frac{2\theta-1}{1-\theta}}-\Delta_{t+1}^{-\frac{2\theta-1}{1-\theta}}\right).
	\end{equation*}
Define $\nu\doteq-\frac{2\theta-1}{1-\theta}$. Then $\nu<0$ and hence
	\begin{equation*}
		\Delta_{t+1}^{\nu}-\Delta_{t}^{\nu}\geq-\nu\xi_1^{-\frac{\theta}{1-\theta}}>0.
	\end{equation*}
	Thus, there exists  $\bar{t}$ such that for all $t\geq2\bar{t}$,
	\begin{equation*}
		\Delta_{t}^{\nu}\geq\Delta_{\bar{t}}^{\nu}-\nu\xi_1^{-\frac{\theta}{1-\theta}}(t-\bar{t})\geq-\frac{\nu}{2}\xi_1^{-\frac{\theta}{1-\theta}}t,
	\end{equation*}
	which implies
	\begin{equation*}
		\Delta_{t}\leq\gamma t^{\frac{1}{\nu}}.
	\end{equation*}
Let $\gamma=\left( -\frac{\nu}{2}\xi_1^{-\frac{\theta}{1-\theta}} \right)^{\frac{1}{\nu}}$. Then, the above inequality shows that (ii) holds.
\end{proof}

\subsection{Low multi-gQt-rank tensor completion model}
\noindent
\par
In this subsection, we establish a novel low-rank tensor completion model based on multi-gQt-rank and then present the tensor factorization based solution method.

Different from \eqref{modelcompute}, we replace the loss function in mode-3, i.e., $\frac{1}{2}\|\A\pro\B-\C\|_F^2$, with a weighted sum of the loss functions in three modes, and consider the following model
\begin{equation}\label{sec5model}
    \min_{\C,\A_{u},\B_{w},w\in[3]}g(\C,\A_{1},\B_{1},\A_{2},\B_{2},\A_{3},\B_{3}), ~~\ \text{s.t.}\ P_{\Omega}(\C-\mathcal{M})=0,\ \Re(\C)=0,
\end{equation}
where,
\begin{equation*}
    g(\C,\A_{1},\B_{1},\A_{2},\B_{2},\A_{3},\B_{3})=\sum_{w=1}^{3}\left( \frac{\alpha_{w}}{2}\|\A_{w}*_{\mu}^{w}\B_{w}-\C\|_F^2+\frac{\lambda}{2}(\|\A_{w}\|_F^2+\|\B_{w}\|_F^2) \right)+\sum_{k=1}^{2}\lambda_k\|\C\times_{k}H_{n_k}\|_F^2,
\end{equation*}
and $\alpha_{w}$ ($w=1,2,3$) is the weighted coefficient.

Similar to Subsection \ref{subsecsolving}, we solve \eqref{sec5model} as the following steps. First, update $\C^t$ by
\begin{align}
    &\arg\min_{\C_{\ii},\C_{\jj},\C_{\kk}}\sum_{w=1}^{3} \frac{\alpha_{w}}{2}\|\A_{w}^{t-1}*_{\mu}^{w}\B_{w}^{t-1}-\C\|_F^2+\sum_{k=1}^{2}\lambda_k\|\C\times_{k}H_{n_k}\|_F^2+\delta_{\{ P_{\Omega}(\C-\mathcal{M})=0\} }(\C) \nonumber \\
    =&\arg\min_{\C_{\ii},\C_{\jj},\C_{\kk}}\frac{1}{2}\sum_{w=1}^{3}\alpha_w\|\X^{t-1}-\C\|_F^2+\sum_{k=1}^{2}\lambda_k\|\C\times_{k}H_{n_k}\|_F^2+\delta_{\{ P_{\Omega}(\C-\mathcal{M})=0\} }(\C), \label{sec5updateC}
\end{align}
where $\X^{t-1}=\frac{1}{\alpha_1+\alpha_2+\alpha_3}\sum_{w=1}^{3}\alpha_w\A_{w}^{t-1}*_{\mu}^{w}\B_{w}^{t-1}$. Let
\begin{align}
    f_{2}(\C_{\ii},\C_{\jj},\C_{\kk})&=\sum_{w=1}^{3} \frac{\alpha_{w}}{2}\|\Im(\X^{t-1})-\C\|_F^2+\sum_{k=1}^{2}\lambda_k\|\C\times_{k}H_{n_k}\|_F^2, \nonumber \\
    h_{2}(\C_{\ii},\C_{\jj},\C_{\kk})&=\delta_{\{ P_{\Omega}(\C-\mathcal{M})=0\} }(\C). \nonumber
\end{align}
Similar to Algorithm \ref{alg:2}, we can solve \eqref{sec5updateC} by the following PGM.
\begin{algorithm}
\caption{PGM for \eqref{sec5updateC}}\label{alg:3}
\begin{algorithmic}
\STATE {{\bf Input.} The tensor data $\mathcal{M},\X^{t-1}$, the observed set $\Omega$, initial step size $\alpha_0$, parameters $\lambda,\lambda_i,H_{n_i},1\leq i\leq 3$.}
\STATE {{\bf Step 0.} Initialize $z^0=[\C_{\ii}^0,\C_{\jj}^0,\C_{\kk}^0]$ which satisfies $P_{\Omega}(\C_{\ii}^0\ii+\C_{\jj}^0\jj+\C_{\kk}^0\kk-\mathcal{M})=0$, $z^{-1}=z^0$. Iterate the following steps for $k=0,1,2,\cdots$ while it does not satisfy stop criterion.}
\STATE {{\bf Step 1.} If $k\geq 1$, choose step size $$\alpha_k=\frac{\|z^{k}-z^{k-1}\|_F^2}{\langle z^{k}-z^{k-1}, \nabla f_2(z^{k})-\nabla f_2(z^{k-1})\rangle},\ \text{or } \alpha_k=\frac{\langle z^{k}-z^{k-1}, \nabla f_2(z^{k})-\nabla f_2(z^{k-1})\rangle}{\|\nabla f_2(z^{k})-\nabla f_2(z^{k-1})\|_F^2}.$$}
\STATE {{\bf Step 2.} Update $[\C_{\ii}^{k+1},\C_{\jj}^{k+1},\C_{\kk}^{k+1}]=\prox_{\alpha_kh_2}\big([\C_{\ii}^{k},\C_{\jj}^{k},\C_{\kk}^{k}]-\alpha_k\nabla f_2(\C_{\ii}^{k},\C_{\jj}^{k},\C_{\kk}^{k})\big)$ .}
\STATE {{\bf Output.} $\C_{\ii}^{k+1}\ii+\C_{\jj}^{k+1}\jj+\C_{\kk}^{k+1}\kk$.}
\end{algorithmic}
\end{algorithm}

It follows from Theorem \ref{lemma51} that \eqref{sec5model} is rewritten as
\begin{equation*}
    \min_{\hat{\C},\hat{\A}_w,\hat{\B}_w,w\in[3]}\sum_{w=1}^{3}\sum_{l=1}^{n_w}\left( \frac{\alpha_{w}}{2n_w}\|\hat{\A}_{w}^{(l)}\hat{\B}_{w}^{(l)}-\hat{\C}_{w}^{(l)}\|_F^2+\frac{\lambda}{2n_w}(\|\A_{w}^{(l)}\|_F^2+\|\B_{w}^{(l)}\|_F^2) \right)+\sum_{k=1}^{2}\lambda_k\|\C\times_{k}H_{n_k}\|_F^2.
\end{equation*}
To update $\hat{\A}_w^{t,(l)}$ and $\hat{\B}_w^{t,(l)}$, we consider the following problem
\begin{equation*}
    \min_{\hat{\A}_w^{t,(l)},\hat{\B}_w^{t,(l)}}\frac{\alpha_{w}}{2n_w}\|\hat{\A}_{w}^{(l)}\hat{\B}_{w}^{(l)}-\hat{\C}_{w}^{(l)}\|_F^2+\frac{\lambda}{2n_w}(\|\A_{w}^{(l)}\|_F^2+\|\B_{w}^{(l)}\|_F^2).
\end{equation*}
For $l\in[n_3]$ and $w\in[3]$, update $\hat{\A}_{w}^{(l)}$ by
\begin{align}
    \hat{\A}_{w}^{t,(l)}&=\mathop{\arg\min}_{\hat{\A}_{w}}\frac{\alpha_w}{2}\|\hat{\A}_{w}^{(l)}\hat{\B}_{w}^{t-1,(l)}-\hat{\C}_{w}^{t,(l)}\|_F^2+\frac{\lambda}{2}\|\hat{\A}_{w}^{(l)}\|_F^2+\frac{\beta}{2}\|\hat{\A}_{w}^{(l)}-\hat{\A}_{w}^{t-1,(l)}\|_F^2 \nonumber \\
    &=\big(\alpha_w\hat{\C}_{w}^{t,(l)}(\hat{\B}_{w}^{t-1,(l)})^*+\beta \hat{\A}_{w}^{t-1,(l)}\big)\big(\alpha_w\hat{\B}_{w}^{t-1,(l)}(\hat{\B}_{w}^{t-1,(l)})^*+(\lambda+\beta)I\big)^{-1},\ \label{sec5updatea}
\end{align}
and the updating of $\hat{\B}_{w}^{(l)}$ is given by
\begin{align}
    \hat{\B}_{w}^{t,(l)}&=\mathop{\arg\min}_{\hat{\B}_{w}}\frac{\alpha_w}{2}\|\hat{\A}_{w}^{t,(l)}\hat{\B}_{w}^{(l)}-\hat{\C}_{w}^{t,(l)}\|_F^2+\frac{\lambda}{2}\|\hat{\B}_{w}^{(l)}\|_F^2+\frac{\beta}{2}\|\hat{\B}^{(l)}_{w}-\hat{\B}_{w}^{t-1,(l)}\|_F^2 \nonumber \\
    &=\big(\alpha_w(\hat{\A}_{w}^{t,(l)})^*\hat{\A}_{w}^{t,(l)}+(\lambda+\beta)I\big)^{-1}\big(\alpha_w(\hat{\A}_{w}^{t,(l)})^*\hat{\C}_{w}^{t,(l)}+\beta \hat{\B}_{w}^{t-1,(l)}\big). \label{sec5updateb}
\end{align}
Consequently, we propose the following algorithm  to solve \eqref{sec5model}. The convergence analysis of Algorithm \ref{alg:4} is similar to that of Algorithm \ref{alg:1} and hence we omit it here.

\begin{algorithm}
\caption{Multi gQt-Rank Tensor Completion (MQRTC)}\label{alg:4}
\begin{algorithmic}
\STATE {{\bf Input.} The tensor data $\mathcal{M}\in\Q^{n_1\times n_2\times n_3}$, the observed set $\Omega$, the rank $\textbf{r}\in\mathbb{Z}_{+}^{n_3}$, parameters $\lambda,\lambda_i,H_{n_i},1\leq i\leq 2$ and $\epsilon$.}
\STATE {{\bf Step 0.} Initialize $\hat{\A}_w^0,\hat{\B}_w^0$ with $w\in[3]$ and $\C^0$ satisfying $P_{\Omega}(\C^0-\mathcal{M})=0,\ \Re(\C^0)=0$, and the rank of $\hat{\A}^0,\hat{\B}^0$ are less than $\textbf{r}$. Iterate the following steps for $t=1,2,\cdots$ while it does not satisfy stop criterion.}
\STATE {{\bf Step 1.} Compute
\begin{equation*}
	\C^{t}=\mathop{\arg\min}_{P_{\Omega}(\C-\mathcal{M})=0,\Re(\C)=0}\sum_{w=1}^{3}\frac{\alpha_w}{2}
\|\A_w^{t-1}*_{\mu}^{w}\B_w^{t-1}-\C\|_F^2+\sum_{k=1}^{2}\lambda_k\|\C\times_{k}H_{n_k}\|_F^2+\sum_{\alpha=\ii,\jj,\kk}\langle \delta^t_{\alpha}, \C_{\alpha}\rangle
\end{equation*}
via Algorithm \ref{alg:3}. That is, given parameters of Algorithm \ref{alg:3}, apply PGM to find an approximate solution $\C^t$ of \eqref{sec5updateC} such that the error vector $\delta^t$ satisfies $\Re(\delta^t)=0,P_{\Omega}(\delta^t)=0$ and the accuracy condition
\begin{equation*}
    \|P_{\Omega^c}(\delta^t)\|_F\leq \frac{1}{4}\|\C^t-\C^{t-1}\|_F.
\end{equation*}
}
\STATE {{\bf Step 2.} Compute $\hat{\A}^t_w$ by  \eqref{sec5updatea}.}
\STATE {{\bf Step 3.} Compute $\hat{\B}^t_w$ by  \eqref{sec5updateb}.}
\STATE {{\bf Step 5.} Check the stop criterion: $$\frac{|g(\C^{k},\A^{k}_1,\B^{k}_1,\A^{k}_2,\B^{k}_2,\A^{k}_3,\B^{k}_3)-g(\C^{k-1},\A^{k-1}_1,\B^{k-1}_1,\A^{k-1}_2,\B^{k-1}_2,\A^{k-1}_3,\B^{k-1}_3)|}{|g(\C^{k-1},\A^{k-1}_1,\B^{k-1}_1,\A^{k-1}_2,\B^{k-1}_2,\A^{k-1}_3,\B^{k-1}_3)|}<\epsilon .$$}
\STATE {{\bf Output.} $\C^t$.}
\end{algorithmic}
\end{algorithm}

\section{Numerical Experiments}\label{sec6}
\noindent
\par
In this section, we report some numerical results of our proposed algorithms QRTC and MQRTC to show the validity. Moreover, we compare them with several existing state-of-the-art methods, including TMac \cite{YXu2013}, TNN \cite{ZZhang2014}, TCTF \cite{PZhou2018} and LRQA-2 \cite{MiaoJi2019}. Note that TMac has two versions, i.e., TMac-dec and TMac-inc, and the former uses the rank-decreasing scheme to adjust its rank while the latter employs the rank-increasing scheme. The codes of TMac\footnote{\url{https://xu-yangyang.github.io/codes/TMac.zip}}, TNN\footnote{\url{http://www.ece.tufts.edu/~shuchin/tensor_completion_and_rpca.zip}}, TCTF\footnote{\url{https://panzhous.github.io/assets/code/TCTF_code.rar}} are open source and LRQA-2 is provided by the authors in \cite{MiaoJi2019}.

A color video is a 4-way tensor defined by two indices for spatial variables, one index for temporal variable and one index for color mode. All the videos in our simulations are initially represented by 4-way tensors $\C\in\mathbb{R}^{n_1\times n_2\times n_3\times 3}$, where $n_1\times n_2$ stands for the pixel scale of each frame, and $n_3$ is the number of frames. $\C(:,:,:,1),\C(:,:,:,2)$ and $\C(:,:,:,3)$ correspond to red, green and blue channels, respectively. The index set is $\Omega$ and the sampling ratio $\rho$ is defined by $$\rho=\frac{\text{numel}(\Omega)}{n_1\times n_2\times n_3}.$$

For TCTF, TNN and TMac3D, the three third-order real tensors for red, green and blue channels are first recovered. Then, the three recovered tensors are combined to form the integrated color video data. For TMac4D, we arrange the color video data as a fourth-order tensor and directly recover the incomplete part $P_{\Omega}(\C)$.
Totally there are four TMac-type methods, i.e., TMac3D-dec, TMac3D-inc, TMac4D-dec and TMac4D-inc. For LRQA-2 method, we recover each frame of video by LRQA-2 and finally  combine them into an integrated video tensor. In our two methods, each color video is reshaped as a pure quaternion tensor $\C\in\Q^{n_1\times n_2\times n_3}$ by using the following way:
\begin{equation*}
    \tilde{\C}=\C(:,:,:,1)\ii+\C(:,:,:,2)\jj+\C(:,:,:,3)\kk .
\end{equation*}
All the simulations are run in MATLAB 2020b under Windows 10 on a laptop with 1.30 GHz CPU and 16GB memory.

\subsection{Quantitative assessment and parameter settings}
\noindent
\par
In order to evaluate the performance of QRTC and MQRTC, we employ four quantitative quality indexes, including the relative square error (RSE), the peak signal-to-noise ratio (PSNR), the structure similarity (SSIM) and the feature similarity (FSIM), which are respectively defined as follows:
\begin{equation*}
    \text{RSE}=10\log10\left( \frac{\|\C-\hat{\C}\|_F}{\|\C\|_F} \right),
\end{equation*}
where $\hat{\C}$ and $\C$ are the recovered and truth data, respectively.
\begin{equation*}
    \text{PSNR}=10\log10\left( \frac{\text{numel}(\hat{\C})\text{Peakval}^2}{\|\hat{\C}-\C\|_F^2} \right),
\end{equation*}
where Peakval is taken from the range of the image datatype (e.g., for uint8 image, it is 255).
\begin{equation*}
    \text{SSIM}= \frac{(2\mu_{\C}\mu_{\hat{\C}}+C_1)(2\sigma_{\C\hat{\C}}+C_2)}{(\mu_{\C}^2+\mu_{\hat{\C}}^2+C_1)(\sigma_{\C}^2+\sigma_{\hat{\C}}^2+C_2)} ,
\end{equation*}
where $\mu_{\C},\mu_{\hat{\C}},\sigma_{\C},\sigma_{\hat{\C}}$ and $\sigma_{\C\hat{\C}}$ are the local means, standard deviations, and cross-covariance for video $\C$ and $\hat{\C}$, $C_1=(0.01L)^2$, $C_2=(0.03L)^2$, $L$ is the specified dynamic range of the pixel values.
\begin{equation*}
    \text{FSIM}=\frac{\sum_{x\in\Delta}S_{L}(x)PC_{m}(x)}{\sum_{x\in\Delta}PC_{m}(x)},
\end{equation*}
where $\Delta$ demotes the whole video spatial and temporal domain. The phase congruency for position $x$ of video $\C$ is denoted as $PC_{\C(x)}$, then $PC_{m}(x)=\max\{ PC_{\C(x)},PC_{\hat{\C}(x)} \}$, $S_{L}(x)$ is the gradient magnitude for position $x$.

Without special instructions, in the all experiments in Section \ref{sec6}, we set the initialized rank $\textbf{r}^{0}=[30,30,30]$ in TMac3D-dec, $\textbf{r}^{0}=[30,30,30,30]$ in TMac4D-dec, $\textbf{r}^{0}=[3,3,3]$ in TMac3D-inc and $\textbf{r}^{0}=[3,3,3,3]$ in TMac4D-inc, and set the weights for both versions as suggested in \cite{YXu2013}. For TCTF, we set the initialized rank $\textbf{r}^{0}=[30,\dots,30]\in\mathbb{R}^{n_3}$ the same as that in \cite{PZhou2018}. Following \cite{MiaoJi2019}, we use LRQA with Laplace function penalty, and set parameter $\gamma=20$. For all the methods except ours, the stopping criteria are built-in their codes. In our methods, the initial rank $\textbf{r}^0=[30,\dots,30]\in\mathbb{R}^{n_3}$. An then we will give the same setting of weight parameter $\alpha_1,\alpha_2,\alpha_3$, the penalty parameter $\lambda$ and TV penalty parameter $\lambda_1,\lambda_2$ in both QRTC and MQRTC. Noticing that two spatial dimension are symmetry, so we can naturally set $\alpha_1=\alpha_2$ and $\lambda_1=\lambda_2$. We Set $\alpha_3$ as cardinality 1, and then as a matter of experience, we set $\alpha_1=\alpha_2=10$, $\lambda_1=\lambda_2=5$ and $\lambda=\alpha_1+\alpha_2+\alpha_3$. In experiments, the maximum iteration number is set to be 20 and the termination precision $\epsilon$ is set to be 1e-3.

\subsection{Performances of methods based on Qt-SVD and t-SVD}\label{Compare}
\noindent
\par
Both t-SVD \cite{tproduct} and our novel factorization Qt-SVD depict the inherent low rank structure of a third order real or quaternion tensor. Here we conduct experiments to compare them in detail on real color video data. Other methods are not compared here since they are not based on matrix factorization of a Fourier transform result. In order to show that Qt-SVD explores the low rank property better of color video, we fairly compare TCTF and our method in the similar formulation. Notice that the model of TCTF is given as \eqref{sec1model5}, we set parameters $\lambda=\lambda_1=\lambda_2=0$ in QRTC \eqref{modelcompute} to get the similar formulation {\bf QRTC-1}, i.e.
\begin{align}\label{sec6model1}
    \min_{\C, \A, \B}\frac{1}{2}\|\A\pro\B-\C\|_F^2,\ \text{s.t.}\ P_{\Omega}(\C-\mathcal{M})=0,\ \Re(\C)=0.
\end{align}

We test TCTF and QRTC-1 on fifteen real color videos data of YUV Video Sequences. The frame size of each video is 288 $\times$ 352, and only the first 30 frames of each video are extracted as experimental data due to the computational limitation. The initialized rank of TCTF and QRTC-1 are set as $\textbf{r}^0=[30,\dots,30]\in\mathbb{R}^{n_3}$.
\begin{figure}[H]
    \centering
    \subfigure[$\rho=0.1$]{
		\begin{minipage}[t]{0.48\linewidth}
			\centering
			\includegraphics[width=1\linewidth]{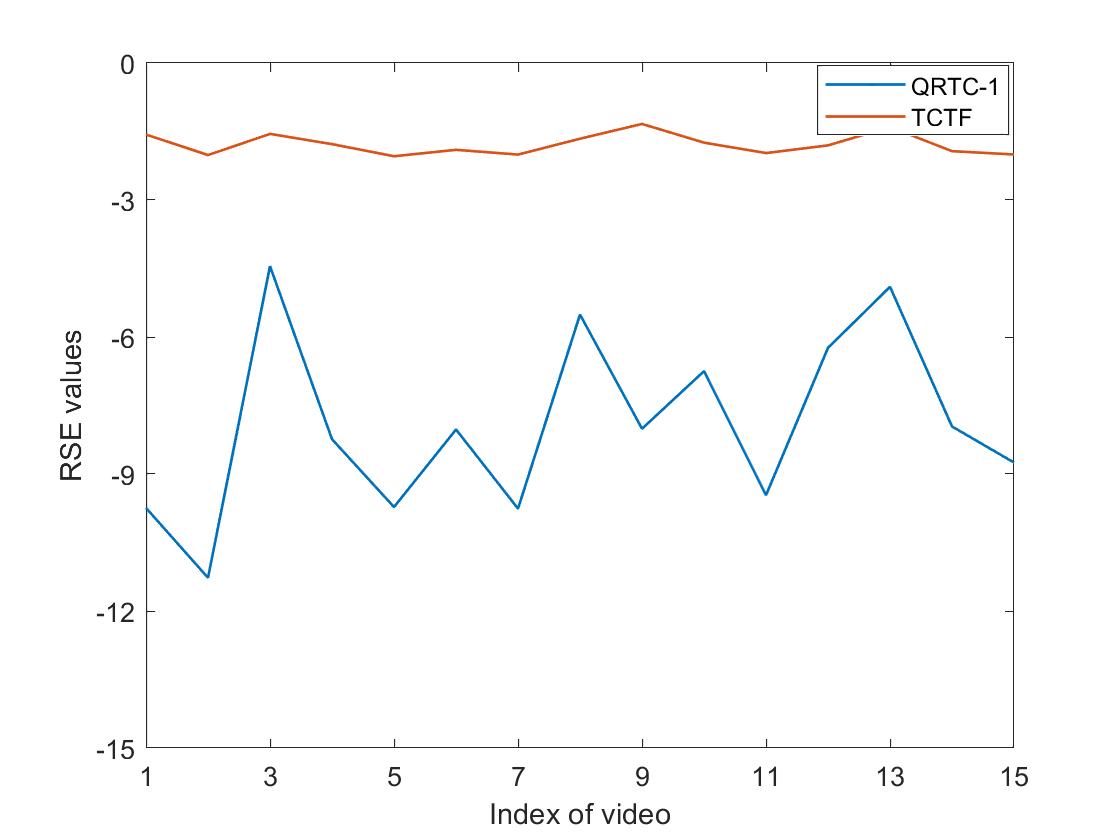}
		\end{minipage}
	}
	\subfigure[$\rho=0.3$]{
		\begin{minipage}[t]{0.48\linewidth}
			\centering
			\includegraphics[width=1\linewidth]{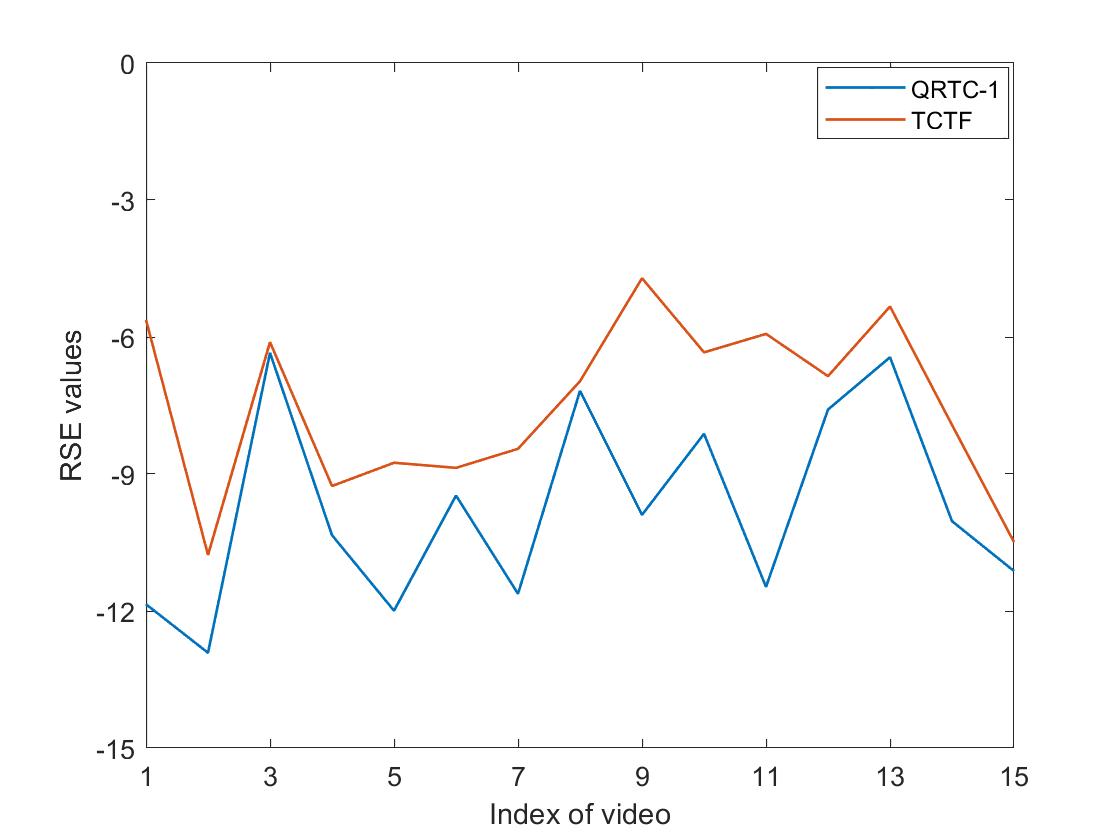}
		\end{minipage}
	}
	\qquad
    \subfigure[$\rho=0.5$]{
		\begin{minipage}[t]{0.48\linewidth}
			\centering
			\includegraphics[width=1\linewidth]{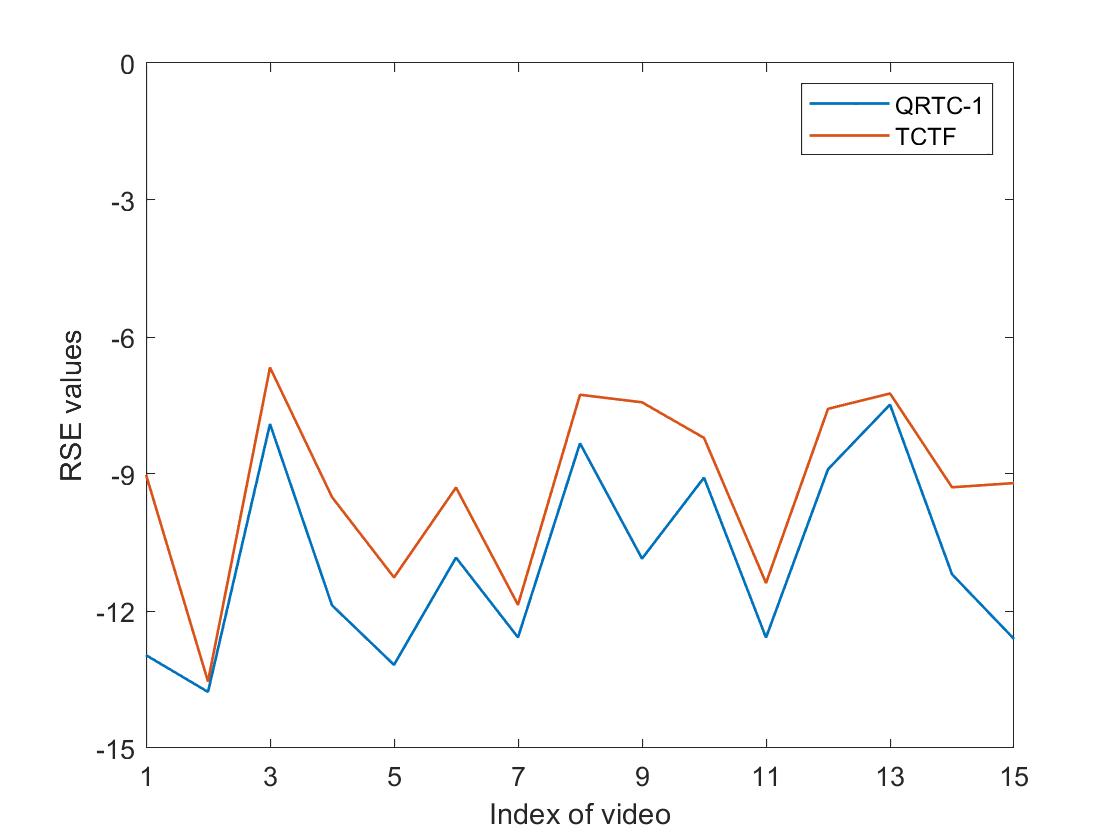}
		\end{minipage}
	}
	\subfigure[$\rho=0.7$]{
		\begin{minipage}[t]{0.48\linewidth}
			\centering
			\includegraphics[width=1\linewidth]{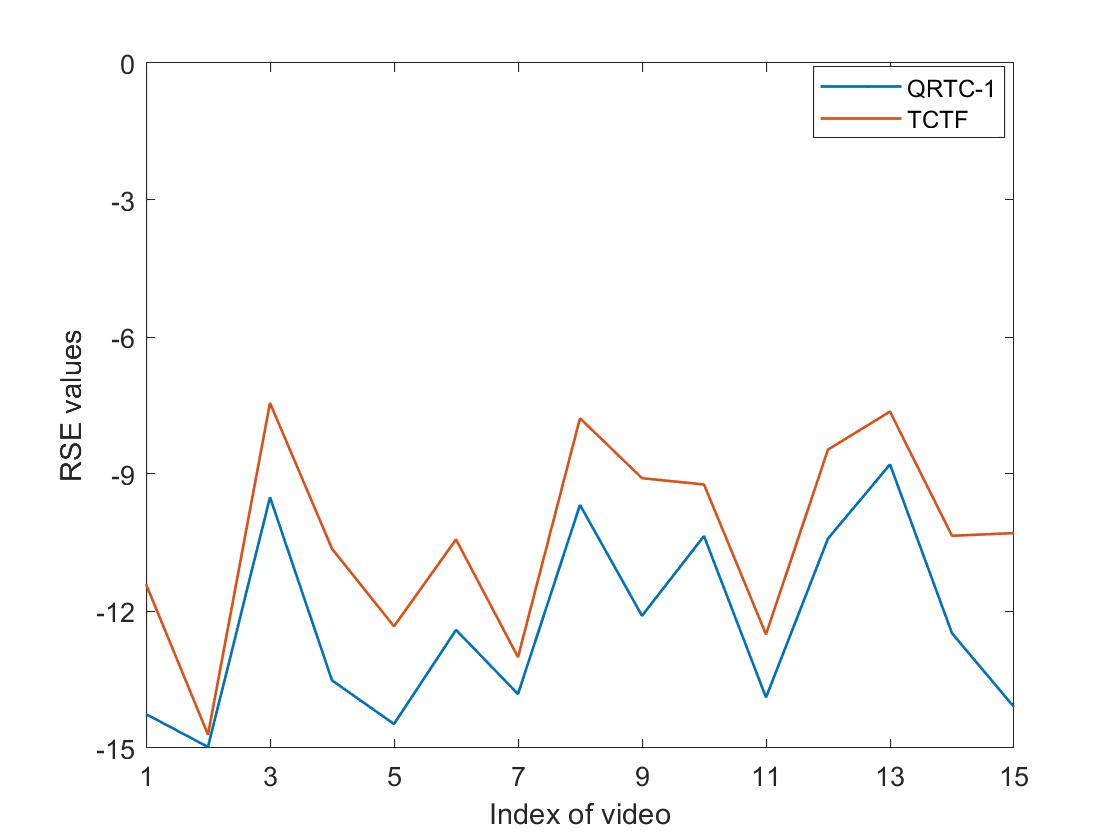}
		\end{minipage}
	}
    \caption{{\small Comparison of RSE result of TCTF and QRTC-1 for color video recovery on 15 videos, sample ratio $\rho=0.1,\ 0.3,\ 0.5$ and $0.7$. From RSE, our Qt-SVD based QRTC outperforms over t-SVD based TCTF. Moreover, the smaller of sample size, the performance of QRTC-1 is better from (a).} }
    \label{fig:compare}
\end{figure}

As shown in Figure \ref{fig:compare}, we display the RSE values of the recovery of fifteen video data with four sample ratios, $\rho=0.1,\ 0.3,\ 0.5$ and $0.7$, respectively. We can see that the RSE values of our Qt-SVD based QRTC-1 are always less than of TCTF with all sample ratios. This shows the better performance and robustness of methods based on Qt-SVD. And when sample ratio $\rho=0.1$, our method has a greater RSE, which shows QRTC-1 has better performance.

Table \ref{tab:compare} displays the PSNR values of recovery of fifteen videos data with four sample ratios, $\rho=0.1,\ 0.3,\ 0.5$ and $0.7$, respectively. The bold values in Table \ref{tab:compare} is the best values of two methods. It is shown in Table \ref{tab:compare} that our method always achieves the best. Especially with $\rho=0.1$, the average PSNR values of our method is two times better than of TCTF, this shows our method has a greater development on exploring the low rank structure of color video data. Table \ref{tab:comparetime} shows the average running time of two methods. We can see that the running time of QRTC-1 not longer than an order of magnitude with TCTF, which is an acceptable cost in practice.

\begin{table}[ht]
    \centering
    \begin{tabular}{|c|c|c|c|c|c|c|c|c|}
    \toprule
    \multirow{2}{*}{\makecell*[c]{Index}} & \multicolumn{2}{c|}{$\rho=0.1$} & \multicolumn{2}{c|}{$\rho=0.3$} & \multicolumn{2}{c|}{$\rho=0.5$} & \multicolumn{2}{c|}{$\rho=0.7$} \\  \cline{2-9}
	& TCTF & QRTC-1 & TCTF & QRTC-1 & TCTF & QRTC-1 & TCTF & QRTC-1 \\ \bottomrule
	\hline
	Akiyo &10.0340 &\textbf{26.3871} & 18.1500 &\textbf{30.6066} & 24.9302 &\textbf{32.8357} & 29.7239 &\textbf{35.4168} \\ \hline
	Bridge(Close)&6.6788 &\textbf{25.1930} & 24.1929 &\textbf{28.4788} & 29.7362 &\textbf{30.1917} &32.0634 &\textbf{32.6023} \\ \hline
	Bus& 11.6938 &\textbf{17.4861} & 20.8127 &\textbf{21.2821} & 21.9240 &\textbf{24.3991} &23.4783 &\textbf{27.6053} \\ \hline
	Coastguard & 9.2534 &\textbf{22.1785} & 24.2265 &\textbf{26.3815} & 24.7123 &\textbf{29.4515} &26.9737 &\textbf{32.7388} \\ \hline
	Container & 8.1161 &\textbf{23.4795} & 21.5396 &\textbf{28.0197} & 26.5626 &\textbf{30.3911} &28.7022 &\textbf{32.9819} \\ \hline
	Flower & 5.3197 &\textbf{17.5676} & 19.2486 &\textbf{20.4618} & 20.1102 &\textbf{23.1790} &22.3841 &\textbf{26.3466} \\ \hline
	Hall Monitor & 8.6508 &\textbf{24.1589} & 21.5368 &\textbf{27.8881} & 28.3716 &\textbf{29.7974} &30.6626 &\textbf{32.2779} \\ \hline
	Mobile & 7.5761 &\textbf{15.2794} & 18.1969 &\textbf{18.6211} & 18.7906 &\textbf{20.9175} &19.8210 &\textbf{23.6180} \\ \hline
	News & 11.1349 &\textbf{24.4833} & 17.8918 &\textbf{28.2573} & 23.3263 &\textbf{30.1733} &26.6495 &\textbf{326758} \\ \hline
	Paris & 10.1639 &\textbf{20.1727} & 19.3490 &\textbf{22.9161} & 23.0965 &\textbf{24.8344} &25.1367 &\textbf{27.3943} \\ \hline
	Silent & 9.3471 &\textbf{24.3320} & 17.2662 &\textbf{28.3484} & 28.1783 &\textbf{30.5686} &30.4278 &\textbf{33.1890} \\ \hline
	Stefan & 8.6749 &\textbf{17.5420} & 18.7819 &\textbf{20.2443} & 20.2135 &\textbf{22.8605} &21.9987 &\textbf{25.8947} \\ \hline
	Tempete & 12.3150 &\textbf{19.3154} & 20.1789 &\textbf{22.4004} & 23.9860 &\textbf{24.4744} &24.7784 &\textbf{27.0957} \\ \hline
	Waterfall & 11.9192 &\textbf{23.9763} & 23.9104 &\textbf{28.1241} & 26.6389 &\textbf{30.4502} &28.7643 &\textbf{30.0320} \\ \hline
	Foreman & 7.1351 &\textbf{20.6288} & 24.1126 &\textbf{25.3807} & 21.5267 &\textbf{28.3540} &23.7150 &\textbf{31.3344} \\ \hline
	\textbf{Average} & 9.2009 &\textbf{21.4787} & 20.6263 &\textbf{25.1607} & 24.1403 &\textbf{27.5252} &26.3520 &\textbf{30.2802} \\ \hline
	
    \end{tabular}
    \caption{{\small Comparison of PSNR result of TCTF and QRTC-1 for color video recovery on 15 videos, sample ratio $\rho=0.1,\ 0.3,\ 0.5,\ 0.7$.}}
    \label{tab:compare}
\end{table}

\begin{table}[ht]
    \centering
    \begin{tabular}{|c|c|c|c|c|c|c|c|}
    \toprule
         \multicolumn{2}{|c|}{$\rho=0.1$} & \multicolumn{2}{c|}{$\rho=0.3$} & \multicolumn{2}{c|}{$\rho=0.5$} & \multicolumn{2}{c|}{$\rho=0.7$} \\  \hline
	TCTF & QRTC-1 & TCTF & QRTC-1 & TCTF & QRTC-1 & TCTF & QRTC-1 \\ \bottomrule
	\hline
	206&219&166&214&169 & 226 &168 & 240 \\ \hline
    \end{tabular}
    \caption{{\small Average running time (seconds) of TCTF and QRTC-1 for color video recovery on 15 videos, sample ratio $\rho=0.1,\ 0.3,\ 0.5,\ 0.7$.}}
    \label{tab:comparetime}
\end{table}

\subsection{Video inpainting for different methods}
\noindent
\par
We first evaluate our method on the {\sl videoSegmentationData} dataset, which can be downloaded in \cite{KFukuchi2009}. We test all the above mentioned methods on color video datasets  ``AN119T", ``BR128T", ``DO01-013", ``DO01-030" and ``M07-058". The frame size of all videos is 288 $\times$ 352. We set the sampling ratio $\rho=0.3$ and uniformly sample the videos to construct the observable index set $\Omega$. All parameters are set as mentioned above. The first frames of five examples of selected videos are shown in Figure \ref{fig:differentvideo}. From Figure \ref{fig:differentvideo}, it is seen that in all videos our two methods have better recovery of derails  on the marginal area between the main target and the surrounding environment. Table \ref{table: differentvideo} summaries the RSE, PSNR, SSIM, FSIM values and the running time of all the algorithms on the five testing videos displayed in Figure \ref{fig:differentvideo}. In Table \ref{table: differentvideo}, the bold values and the values in brackets stands for the best and the second best values of items RSE, PSNR, SSIM and FSIM, respectively. From the results, it is shown that the overall performance of QRTC and MQRTC are vastly superior to the others: the best of the above four quality assessments is consistently of QRTC or MQRTC, while MQRTC behaves  better than MARTC in most of cases.
The running time of QRTC is longer than several methods of TMac, but not longer than an order of magnitude. The running time of MQRTC is approximately fourfold as long as that of QRTC, but no more  than that of TNN. All these outcomes demonstrate that in terms of color video inpainting problems, our methods have better recovery accuracy than others and runs also very efficiently.
\begin{figure}[htbp]
	\centering
	
	\subfigure{
		
        \rotatebox{90}{\bf{\scriptsize{Original frame}}}
		\begin{minipage}[t]{0.144\linewidth}
			\centering
			\includegraphics[width=1\linewidth]{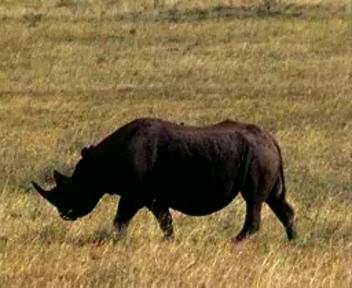}
		\end{minipage}
	}
	\subfigure{
		\begin{minipage}[t]{0.144\linewidth}
			\centering
			\includegraphics[width=1\linewidth]{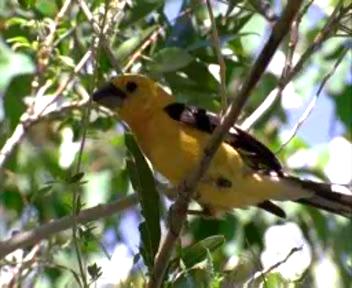}
		\end{minipage}
	}
	\subfigure{
		\begin{minipage}[t]{0.144\linewidth}
			\centering
			\includegraphics[width=1\linewidth]{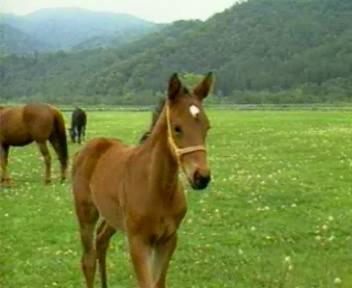}
		\end{minipage}
	}
	\subfigure{
		\begin{minipage}[t]{0.144\linewidth}
			\centering
			\includegraphics[width=1\linewidth]{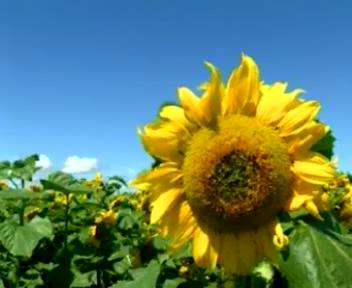}
		\end{minipage}
	}
	\subfigure{
		\begin{minipage}[t]{0.144\linewidth}
			\centering
			\includegraphics[width=1\linewidth]{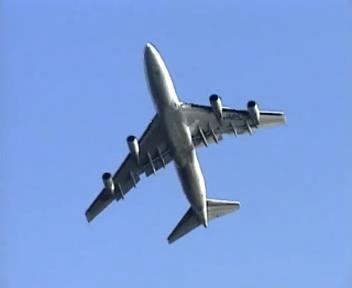}
		\end{minipage}
	}

	\vspace{-3mm}
	\setcounter{subfigure}{0}
	\subfigure{
		
        \rotatebox{90}{\bf{\scriptsize{~Observation}}}
		\begin{minipage}[t]{0.144\linewidth}
			\centering
			\includegraphics[width=1\linewidth]{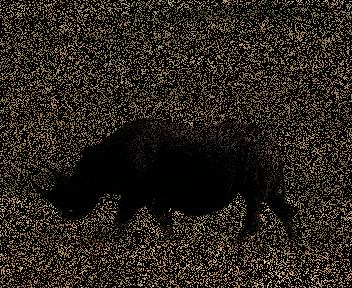}
		\end{minipage}
	}
	\subfigure{
		\begin{minipage}[t]{0.144\linewidth}
			\centering
			\includegraphics[width=1\linewidth]{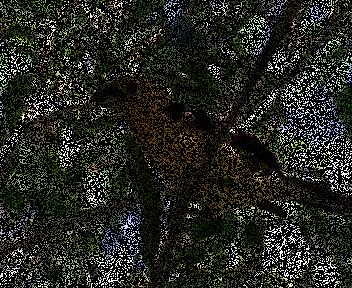}
		\end{minipage}
	}
	\subfigure{
		\begin{minipage}[t]{0.144\linewidth}
			\centering
			\includegraphics[width=1\linewidth]{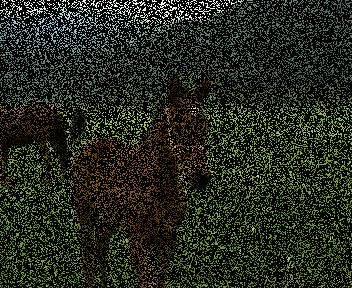}
		\end{minipage}
	}
	\subfigure{
		\begin{minipage}[t]{0.144\linewidth}
			\centering
			\includegraphics[width=1\linewidth]{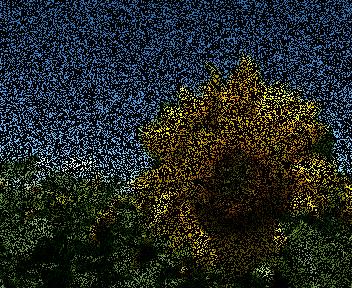}
		\end{minipage}
	}
	\subfigure{
		\begin{minipage}[t]{0.144\linewidth}
			\centering
			\includegraphics[width=1\linewidth]{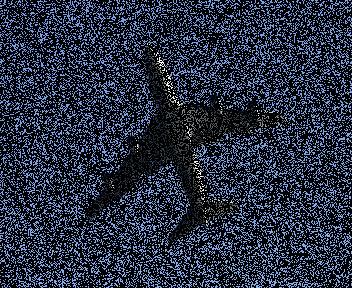}
		\end{minipage}
	}

	\vspace{-3mm}
	\setcounter{subfigure}{0}
	\subfigure{
		
        \rotatebox{90}{\bf{\scriptsize{~~~~TCTF}}}
		\begin{minipage}[t]{0.144\linewidth}
			\centering
			\includegraphics[width=1\linewidth]{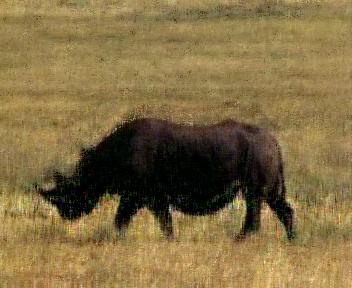}
		\end{minipage}
	}
	\subfigure{
		\begin{minipage}[t]{0.144\linewidth}
			\centering
			\includegraphics[width=1\linewidth]{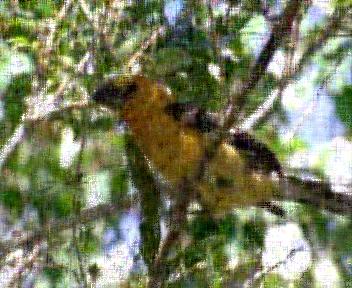}
		\end{minipage}
	}
	\subfigure{
		\begin{minipage}[t]{0.144\linewidth}
			\centering
			\includegraphics[width=1\linewidth]{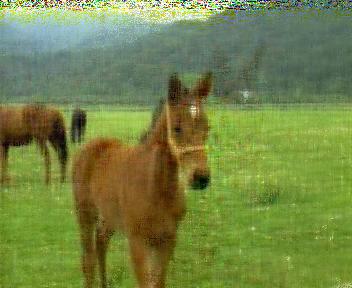}
		\end{minipage}
	}
	\subfigure{
		\begin{minipage}[t]{0.144\linewidth}
			\centering
			\includegraphics[width=1\linewidth]{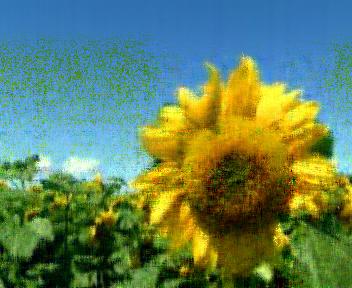}
		\end{minipage}
	}
	\subfigure{
		\begin{minipage}[t]{0.144\linewidth}
			\centering
			\includegraphics[width=1\linewidth]{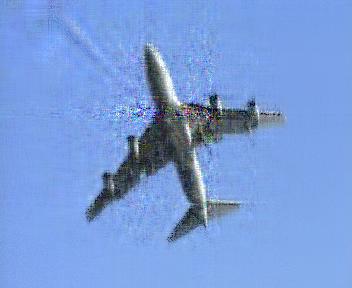}
		\end{minipage}
	}
	
	\vspace{-3mm}
	\setcounter{subfigure}{0}
	\subfigure{
		
        \rotatebox{90}{\bf{\scriptsize{~~~~~TNN}}}
		\begin{minipage}[t]{0.144\linewidth}
			\centering
			\includegraphics[width=1\linewidth]{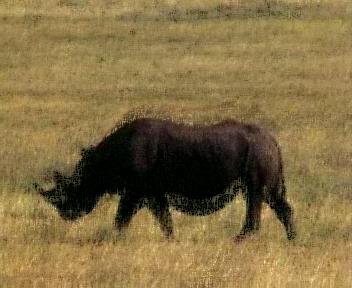}
		\end{minipage}
	}
	\subfigure{
		\begin{minipage}[t]{0.144\linewidth}
			\centering
			\includegraphics[width=1\linewidth]{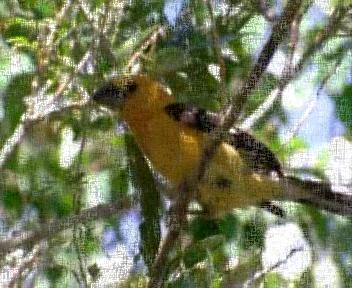}
		\end{minipage}
	}
	\subfigure{
		\begin{minipage}[t]{0.144\linewidth}
			\centering
			\includegraphics[width=1\linewidth]{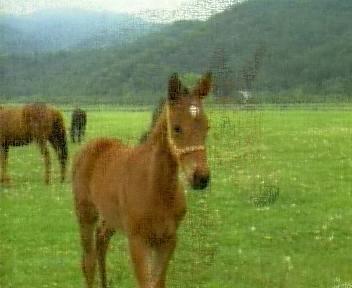}
		\end{minipage}
	}
	\subfigure{
		\begin{minipage}[t]{0.144\linewidth}
			\centering
			\includegraphics[width=1\linewidth]{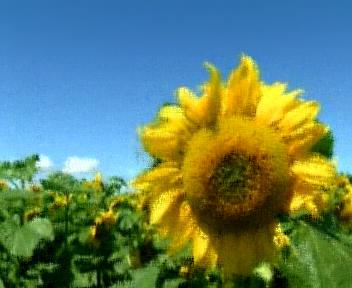}
		\end{minipage}
	}
	\subfigure{
		\begin{minipage}[t]{0.144\linewidth}
			\centering
			\includegraphics[width=1\linewidth]{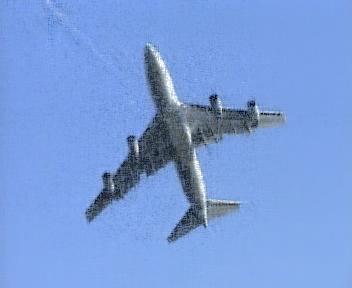}
		\end{minipage}
	}
	
	\vspace{-3mm}
	\setcounter{subfigure}{0}
	\subfigure{
		
        \rotatebox{90}{\bf{\scriptsize{~TMac3D-dec}}}
		\begin{minipage}[t]{0.144\linewidth}
			\centering
			\includegraphics[width=1\linewidth]{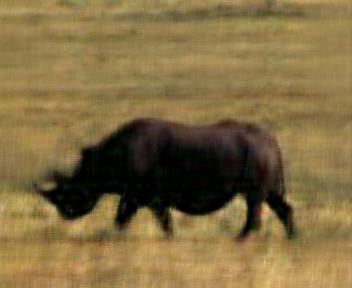}
		\end{minipage}
	}
	\subfigure{
		\begin{minipage}[t]{0.144\linewidth}
			\centering
			\includegraphics[width=1\linewidth]{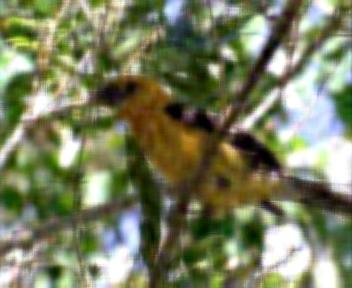}
		\end{minipage}
	}
	\subfigure{
		\begin{minipage}[t]{0.144\linewidth}
			\centering
			\includegraphics[width=1\linewidth]{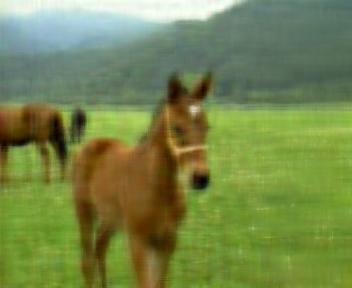}
		\end{minipage}
	}
	\subfigure{
		\begin{minipage}[t]{0.144\linewidth}
			\centering
			\includegraphics[width=1\linewidth]{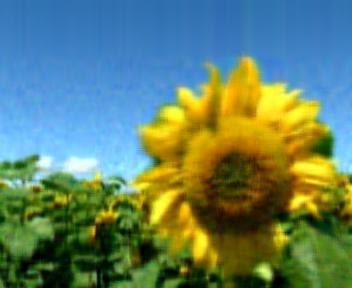}
		\end{minipage}
	}
	\subfigure{
		\begin{minipage}[t]{0.144\linewidth}
			\centering
			\includegraphics[width=1\linewidth]{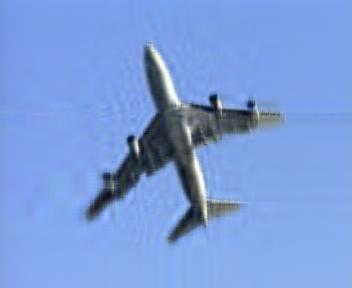}
		\end{minipage}
	}
	
	\vspace{-3mm}
	\setcounter{subfigure}{0}
	\subfigure{
		
        \rotatebox{90}{\bf{\scriptsize{~TMac3D-inc}}}
		\begin{minipage}[t]{0.144\linewidth}
			\centering
			\includegraphics[width=1\linewidth]{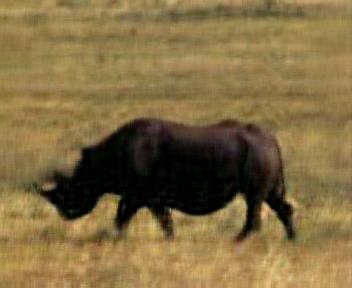}
		\end{minipage}
	}
	\subfigure{
		\begin{minipage}[t]{0.144\linewidth}
			\centering
			\includegraphics[width=1\linewidth]{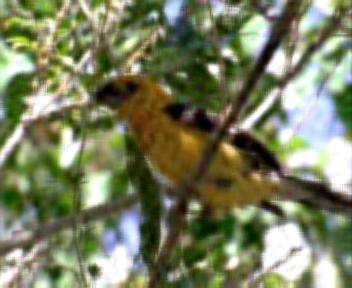}
		\end{minipage}
	}
	\subfigure{
		\begin{minipage}[t]{0.144\linewidth}
			\centering
			\includegraphics[width=1\linewidth]{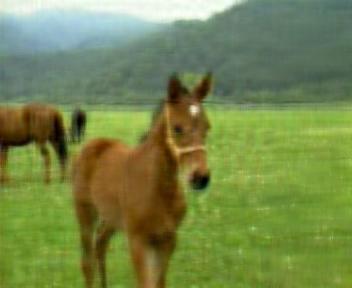}
		\end{minipage}
	}
	\subfigure{
		\begin{minipage}[t]{0.144\linewidth}
			\centering
			\includegraphics[width=1\linewidth]{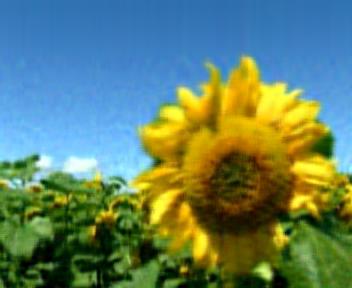}
		\end{minipage}
	}
	\subfigure{
		\begin{minipage}[t]{0.144\linewidth}
			\centering
			\includegraphics[width=1\linewidth]{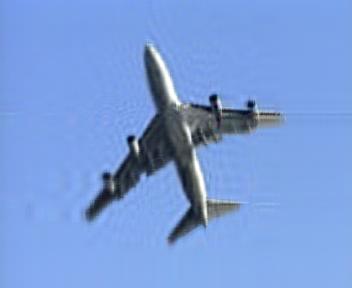}
		\end{minipage}
	}
	
	\vspace{-3mm}
	\setcounter{subfigure}{0}
	\subfigure{
		
        \rotatebox{90}{\bf{\scriptsize{~TMac4D-dec}}}
        \begin{minipage}[t]{0.144\linewidth}
			\centering
			\includegraphics[width=1\linewidth]{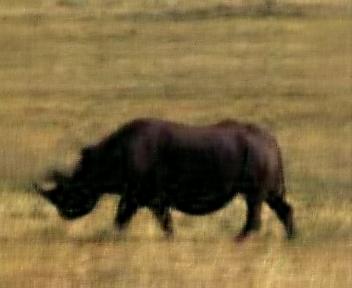}
		\end{minipage}
	}
	\subfigure{
		\begin{minipage}[t]{0.144\linewidth}
			\centering
			\includegraphics[width=1\linewidth]{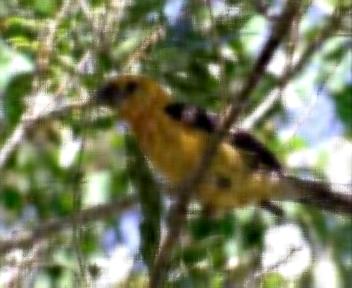}
		\end{minipage}
	}
	\subfigure{
		\begin{minipage}[t]{0.144\linewidth}
			\centering
			\includegraphics[width=1\linewidth]{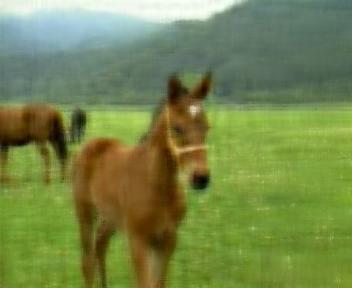}
		\end{minipage}
	}
	\subfigure{
		\begin{minipage}[t]{0.144\linewidth}
			\centering
			\includegraphics[width=1\linewidth]{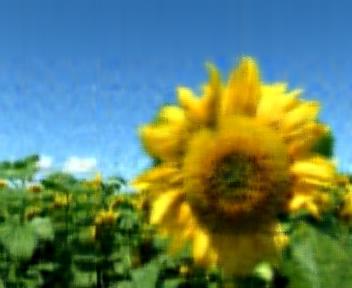}
		\end{minipage}
	}
	\subfigure{
		\begin{minipage}[t]{0.144\linewidth}
			\centering
			\includegraphics[width=1\linewidth]{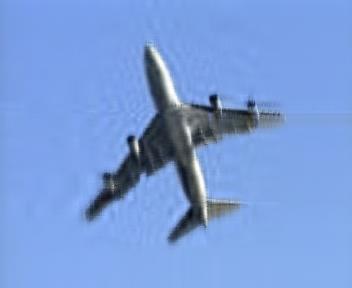}
		\end{minipage}
	}

	\vspace{-3mm}
	\setcounter{subfigure}{0}
	\subfigure{
		
        \rotatebox{90}{\bf{\scriptsize{~TMac4D-inc}}}
		\begin{minipage}[t]{0.144\linewidth}
			\centering
			\includegraphics[width=1\linewidth]{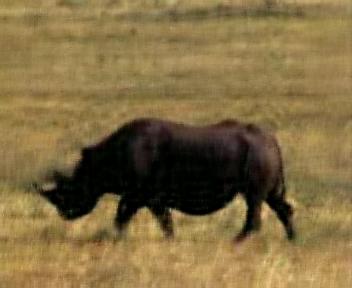}
		\end{minipage}
	}
	\subfigure{
		\begin{minipage}[t]{0.144\linewidth}
			\centering
			\includegraphics[width=1\linewidth]{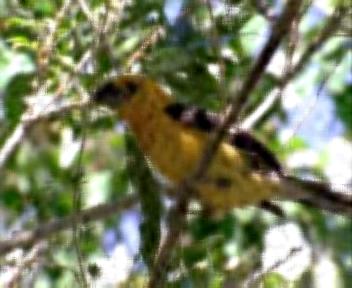}
		\end{minipage}
	}
	\subfigure{
		\begin{minipage}[t]{0.144\linewidth}
			\centering
			\includegraphics[width=1\linewidth]{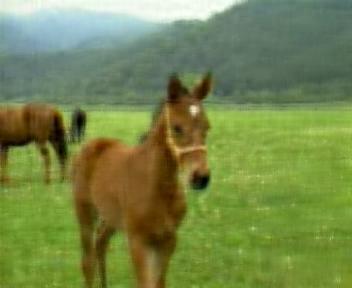}
		\end{minipage}
	}
	\subfigure{
		\begin{minipage}[t]{0.144\linewidth}
			\centering
			\includegraphics[width=1\linewidth]{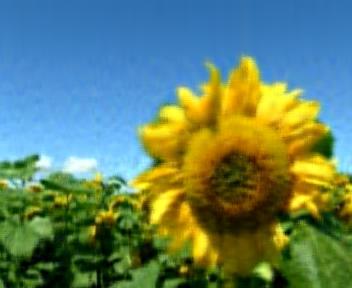}
		\end{minipage}
	}
	\subfigure{
		\begin{minipage}[t]{0.144\linewidth}
			\centering
			\includegraphics[width=1\linewidth]{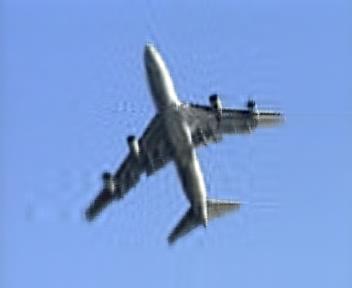}
		\end{minipage}
	}
	
	\vspace{-3mm}
	\setcounter{subfigure}{0}
	\subfigure{
		
        \rotatebox{90}{\bf{\scriptsize{~~~~LRQA-2}}}
		\begin{minipage}[t]{0.144\linewidth}
			\centering
			\includegraphics[width=1\linewidth]{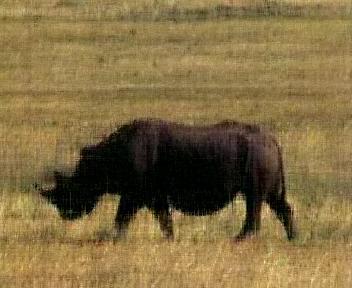}
		\end{minipage}
	}
	\subfigure{
		\begin{minipage}[t]{0.144\linewidth}
			\centering
			\includegraphics[width=1\linewidth]{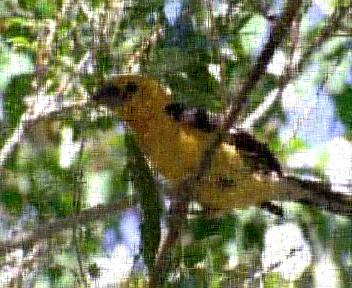}
		\end{minipage}
	}
	\subfigure{
		\begin{minipage}[t]{0.144\linewidth}
			\centering
			\includegraphics[width=1\linewidth]{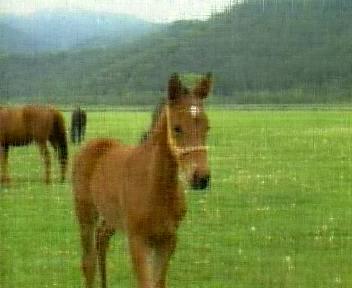}
		\end{minipage}
	}
	\subfigure{
		\begin{minipage}[t]{0.144\linewidth}
			\centering
			\includegraphics[width=1\linewidth]{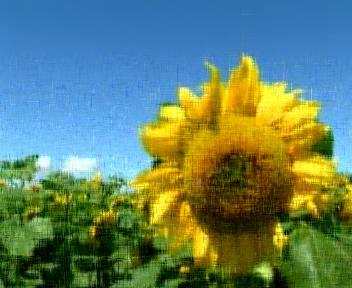}
		\end{minipage}
	}
	\subfigure{
		\begin{minipage}[t]{0.144\linewidth}
			\centering
			\includegraphics[width=1\linewidth]{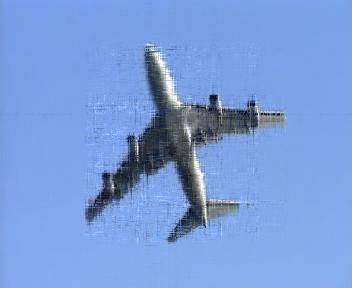}
		\end{minipage}
	}
	
	\vspace{-3mm}
	\setcounter{subfigure}{0}
	\subfigure{
		
        \rotatebox{90}{\bf{\scriptsize{~~~~QRTC}}}
		\begin{minipage}[t]{0.144\linewidth}
			\centering
			\includegraphics[width=1\linewidth]{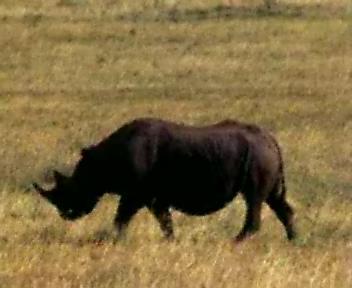}
		\end{minipage}
	}
	\subfigure{
		\begin{minipage}[t]{0.144\linewidth}
			\centering
			\includegraphics[width=1\linewidth]{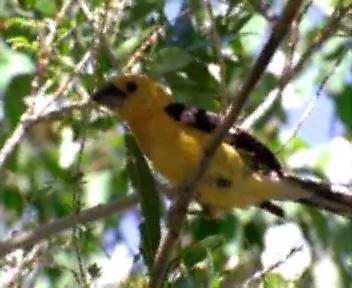}
		\end{minipage}
	}
	\subfigure{
		\begin{minipage}[t]{0.144\linewidth}
			\centering
			\includegraphics[width=1\linewidth]{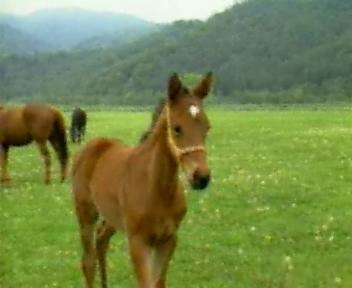}
		\end{minipage}
	}
	\subfigure{
		\begin{minipage}[t]{0.144\linewidth}
			\centering
			\includegraphics[width=1\linewidth]{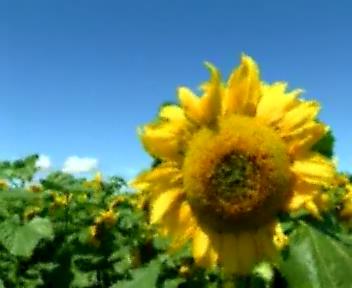}
		\end{minipage}
	}
	\subfigure{
		\begin{minipage}[t]{0.144\linewidth}
			\centering
			\includegraphics[width=1\linewidth]{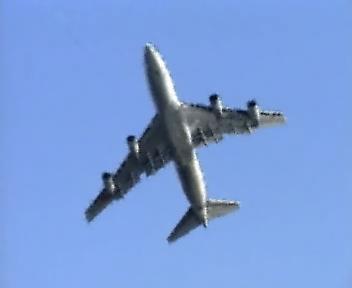}
		\end{minipage}
	}
	
	\vspace{-3mm}
	\setcounter{subfigure}{0}
	\subfigure[AN119T]{
		
        \rotatebox{90}{\bf{\scriptsize{~~~MQRTC}}}
		\begin{minipage}[t]{0.144\linewidth}
			\centering
			\includegraphics[width=1\linewidth]{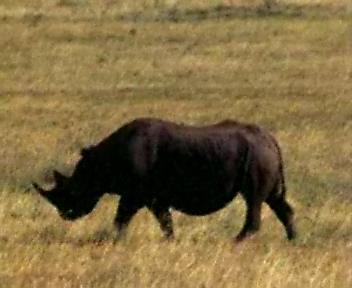}
		\end{minipage}
	}
	\subfigure[BR128T]{
		\begin{minipage}[t]{0.144\linewidth}
			\centering
			\includegraphics[width=1\linewidth]{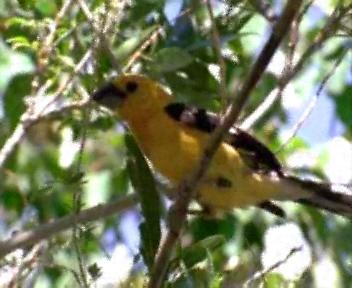}
		\end{minipage}
	}
	\subfigure[DO01-013]{
		\begin{minipage}[t]{0.144\linewidth}
			\centering
			\includegraphics[width=1\linewidth]{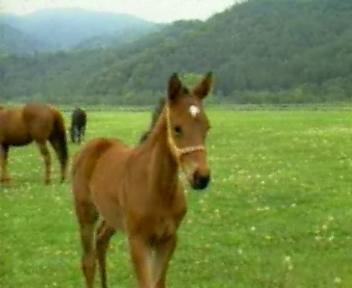}
		\end{minipage}
	}
	\subfigure[DO01-030]{
		\begin{minipage}[t]{0.144\linewidth}
			\centering
			\includegraphics[width=1\linewidth]{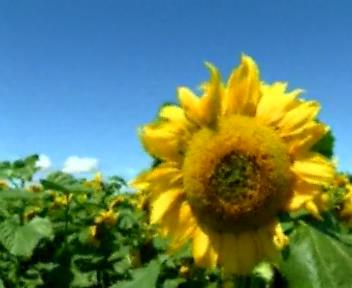}
		\end{minipage}
	}
	\subfigure[M07-058]{
		\begin{minipage}[t]{0.144\linewidth}
			\centering
			\includegraphics[width=1\linewidth]{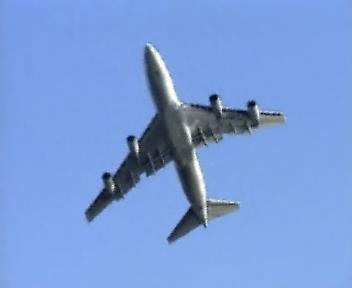}
		\end{minipage}
	}

    \caption{{\small First frame of color video recovery using different algorithms ($\rho=0.3$).}}
    \label{fig:differentvideo}
\end{figure}

\begin{sidewaystable}
\centering
\caption{Quantitave quality indexes and running time (seconds) of different algorithm on the five videos displayed in Figure \ref{fig:differentvideo} ($\rho=0.3$).}
\label{table: differentvideo}

\begin{tabular}{ |c|c|c|c|c|c|c|c|c|c|c| }
    \toprule
	Videos & Indexes & ~~~TCTF~~~ & ~~~~TNN~~~~ & TMac3D-dec & TMac3D-inc & TMac4D-dec & TMac4D-inc & ~~~LRQA-2~~~ & ~~~QRTC~~~ & ~~~MQRTC~~~ \\ \bottomrule
	\hline
	\multirow{5}{*}{AN119T} & RSE & -10.1719 & -11.3749 & -10.5061 & -10.9905 & -10.6873 & -11.1481 & -10.3217 & (-12.5983) & \bf{-12.7885}  \\ \cline{2-11}
	 & PSNR & 26.7931 & 29.1990 & 27.4615 & 28.4303 & 27.8239 & 28.7456 & 27.0927 & (31.6458) & \bf{32.0263}  \\\cline{2-11}
	 & SSIM & 0.8681 & 0.8984 & 0.8739 & 0.8864 & 0.8814 & 0.8922 & 0.8748 & \bf{0.9681} & (0.9498)  \\\cline{2-11}
	 & FSIM & 0.8687 & 0.8978 & 0.7984 & 0.8288 & 0.8170 & 0.8428 & 0.8689 & (0.9175) & \bf{0.9249}  \\\cline{2-11}
	 & time(s) & 241 & 2182 & 49 & 141 & 103 & 185 & 2460 & 173 & 744  \\ \hline
	\multirow{5}{*}{BR128T} & RSE & -8.2424 & -9.7408 & -7.5778 & -8.4865 & -7.7757 & -8.6600 & -8.3008 & (-11.0355) & \bf{-11.2164}  \\\cline{2-11}
	 & PSNR & 22.3657 & 25.3624 & 21.0364 & 22.8538 & 21.4322 & 23.2008 & 22.4825 & (27.9519) & \bf{28.3137}  \\\cline{2-11}
	 & SSIM & 0.7691 & 0.8684 & 0.7612 & 0.8168 & 0.7836 & 0.8311 & 0.7840 & \bf{0.9376} & (0.9299)  \\\cline{2-11}
	 & FSIM & 0.8444 & 0.8917 & 0.8063 & 0.8438 & 0.8193 & 0.8519 & 0.8445 & \bf{0.9443} & (0.9436)  \\\cline{2-11}
	 & time(s) & 402 & 1813 & 60 & 180 & 174 & 256 & 3908 & 221 & 923  \\ \hline
	\multirow{5}{*}{DO01-013} & RSE & -7.6358 & -13.0349 & -11.5504 & -12.1381 & -11.7081 & -12.2760 & -11.5340 & (-13.8539) & \bf{-14.4249}  \\\cline{2-11}
	 & PSNR & 22.1400 & 32.9382 & 29.9691 & 31.1446 & 30.2846 & 31.4203 & 29.9363 & (34.5761) & \bf{35.7182}  \\\cline{2-11}
	 & SSIM & 0.8758 & 0.9752 & 0.9600 & 0.9667 & 0.9638 & 0.9696 & 0.9502 & (0.9863) & \bf{0.9883}  \\\cline{2-11}
	 & FSIM & 0.8575 & 0.9405 & 0.8553 & 0.8804 & 0.8679 & 0.8888 & 0.8983 & (0.9462) & \bf{0.9582}  \\\cline{2-11}
	 & time(s) & 310 & 3241 & 59 & 140 & 122 & 186 & 2262 & 169 & 700  \\ \hline
	\multirow{5}{*}{DO01-030} & RSE & -6.9388 & -12.0369 & -10.6581 & -11.3983 & -10.6586 & -11.3957 & -10.1583 & (-12.7410) & \bf{-13.2912}  \\\cline{2-11}
	 & PSNR & 19.4925 & 29.6887 & 26.9312 & 28.4115 & 26.9322 & 28.4062 & 25.9315 & (31.0969) & \bf{32.1972}  \\\cline{2-11}
	 & SSIM & 0.8509 & 0.9763 & 0.9535 & 0.9646 & 0.9535 & 0.9642 & 0.9431 & (0.9837) & \bf{0.9862}  \\\cline{2-11}
	 & FSIM & 0.8436 & 0.9175 & 0.8551 & 0.8786 & 0.8568 & 0.8798 & 0.9527 & (0.9370) & \bf{0.9481}  \\\cline{2-11}
	 & time(s) & 364 & 5770 & 59 & 161 & 152 & 231 & 3116 & 196 & 793  \\ \hline
	\multirow{5}{*}{M07-058} & RSE & -11.8070 & -13.9020 & -13.8192 & -14.5612 & -13.9278 & -14.6486 & -12.7418 & (-15.4564) & \bf{-15.9781}  \\\cline{2-11}
	 & PSNR & 27.1223 & 31.3122 & 31.1466 & 32.6305 & 31.3638 & 32.8055 & 28.9919 & (34.4211) & \bf{35.4644}  \\\cline{2-11}
	 & SSIM & 0.9572 & 0.9805 & 0.9787 & 0.9839 & 0.9806 & 0.9849 & 0.9691 & (0.9910) & \bf{0.9925}  \\\cline{2-11}
	 & FSIM & 0.7969 & 0.8766 & 0.9155 & 0.9255 & 0.9155 & 0.9278 & 0.9033 & \bf{0.9754} & (0.9584)  \\\cline{2-11}
	 & time(s) & 203 & 4206 & 47 & 131 & 80 & 170 & 1854 & 135 & 542  \\
	\hline
\end{tabular}

\end{sidewaystable}

To verify the robustness of our methods to the sampling ratio $\rho$, we test the  video ``Stefan" which is of YUV Video Sequences. The frame size of the video is 288 $\times$ 352. We set the sampling ratio $\rho$ ranging from $0.1$ to $0.5$ and uniformly sample the pixels to construct $\Omega$. All the other  parameters are set as mentioned. The first frame of the selected video with different sampling ratios are shown in Figure \ref{fig:differentrho}. Figure \ref{fig:differentrho} indicates that the recovered videos of our methods are the clearest under all sampling ratios. Table \ref{table: differentrho} summaries the RSE, PSNR, SSIM, FSIM values and the running time of all the algorithms on the selected video with all sampling ratios which are displayed in Figure \ref{fig:differentrho}. In Table \ref{table: differentrho}, the bold values and the values in brackets are the best and the second best values of RSE, PSNR, SSIM and FSIM, respectively. From the results, the best and the second best of PSNR, RSE or SSIM are of either QRTC or MQRTC. For FSIM value, MQRTC and QRTC performs better than others except the situation when  $\rho=0.1$. The running time of QRTC is longer than several methods of TMac, but not longer than an order of magnitude. The running time of MQRTC is approximately fourfold as long as that of QRTC, but not exceeds that of TNN or LRQA-2.  Thus, we can conclude that our methods are rather robust to the sampling ratio and have considerably better performance than others.

\begin{figure}[htbp]
	\centering
	
	\subfigure{
		
        \rotatebox{90}{\bf{\scriptsize{Original frame}}}
		\begin{minipage}[t]{0.144\linewidth}
			\centering
			\includegraphics[width=1\linewidth]{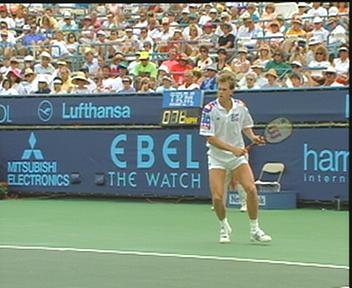}
		\end{minipage}
	}
	\subfigure{
		\begin{minipage}[t]{0.144\linewidth}
			\centering
			\includegraphics[width=1\linewidth]{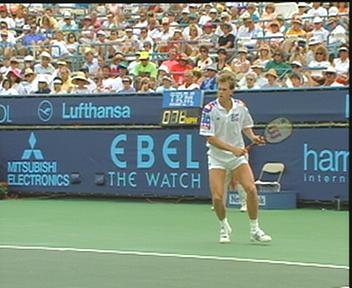}
		\end{minipage}
	}
	\subfigure{
		\begin{minipage}[t]{0.144\linewidth}
			\centering
			\includegraphics[width=1\linewidth]{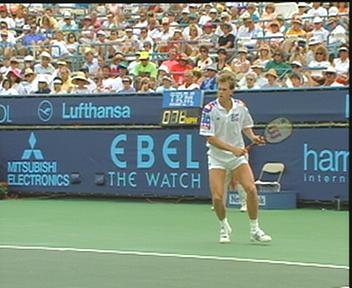}
		\end{minipage}
	}
	\subfigure{
		\begin{minipage}[t]{0.144\linewidth}
			\centering
			\includegraphics[width=1\linewidth]{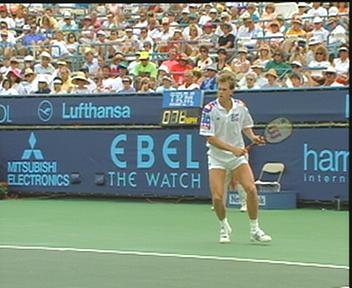}
		\end{minipage}
	}
	\subfigure{
		\begin{minipage}[t]{0.144\linewidth}
			\centering
			\includegraphics[width=1\linewidth]{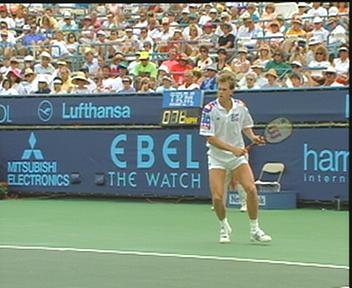}
		\end{minipage}
	}

	\vspace{-3mm}
	\setcounter{subfigure}{0}
	\subfigure{
		
        \rotatebox{90}{\bf{\scriptsize{~Observation}}}
		\begin{minipage}[t]{0.144\linewidth}
			\centering
			\includegraphics[width=1\linewidth]{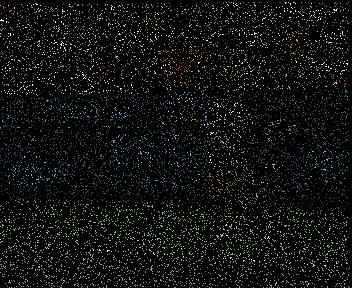}
		\end{minipage}
	}
	\subfigure{
		\begin{minipage}[t]{0.144\linewidth}
			\centering
			\includegraphics[width=1\linewidth]{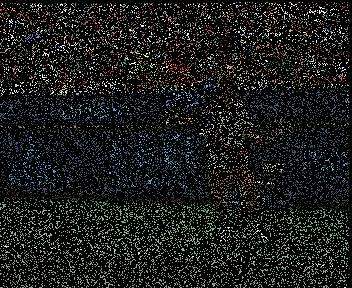}
		\end{minipage}
	}
	\subfigure{
		\begin{minipage}[t]{0.144\linewidth}
			\centering
			\includegraphics[width=1\linewidth]{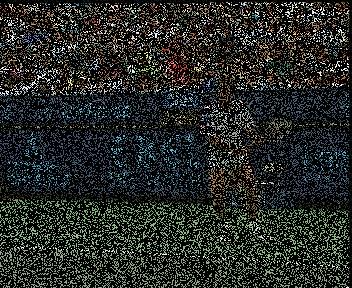}
		\end{minipage}
	}
	\subfigure{
		\begin{minipage}[t]{0.144\linewidth}
			\centering
			\includegraphics[width=1\linewidth]{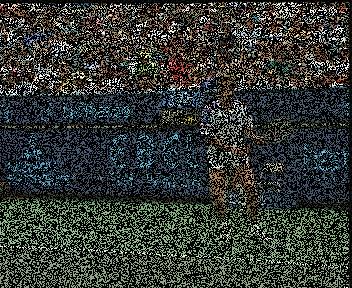}
		\end{minipage}
	}
	\subfigure{
		\begin{minipage}[t]{0.144\linewidth}
			\centering
			\includegraphics[width=1\linewidth]{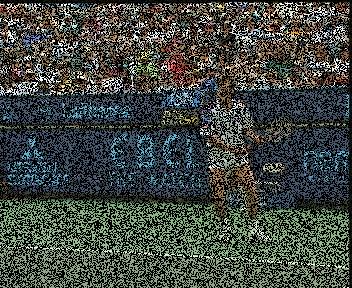}
		\end{minipage}
	}

	\vspace{-3mm}
	\setcounter{subfigure}{0}
	\subfigure{
		
        \rotatebox{90}{\bf{\scriptsize{~~~~TCTF}}}
		\begin{minipage}[t]{0.144\linewidth}
			\centering
			\includegraphics[width=1\linewidth]{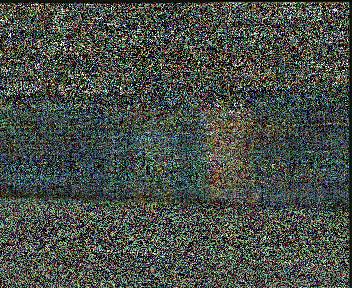}
		\end{minipage}
	}
	\subfigure{
		\begin{minipage}[t]{0.144\linewidth}
			\centering
			\includegraphics[width=1\linewidth]{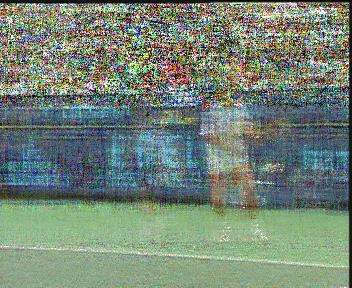}
		\end{minipage}
	}
	\subfigure{
		\begin{minipage}[t]{0.144\linewidth}
			\centering
			\includegraphics[width=1\linewidth]{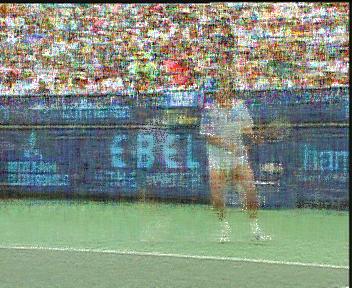}
		\end{minipage}
	}
	\subfigure{
		\begin{minipage}[t]{0.144\linewidth}
			\centering
			\includegraphics[width=1\linewidth]{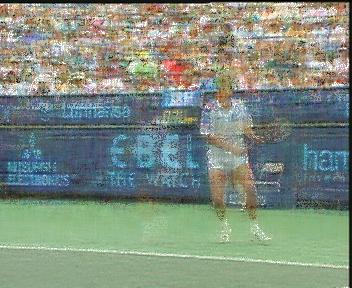}
		\end{minipage}
	}
	\subfigure{
		\begin{minipage}[t]{0.144\linewidth}
			\centering
			\includegraphics[width=1\linewidth]{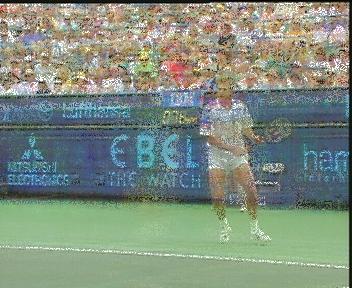}
		\end{minipage}
	}
	
	\vspace{-3mm}
	\setcounter{subfigure}{0}
	\subfigure{
		
        \rotatebox{90}{\bf{\scriptsize{~~~~~TNN}}}
		\begin{minipage}[t]{0.144\linewidth}
			\centering
			\includegraphics[width=1\linewidth]{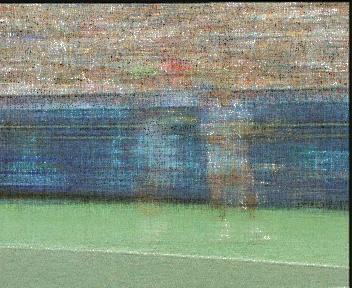}
		\end{minipage}
	}
	\subfigure{
		\begin{minipage}[t]{0.144\linewidth}
			\centering
			\includegraphics[width=1\linewidth]{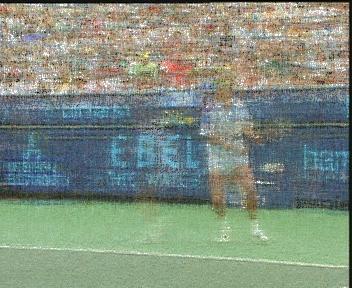}
		\end{minipage}
	}
	\subfigure{
		\begin{minipage}[t]{0.144\linewidth}
			\centering
			\includegraphics[width=1\linewidth]{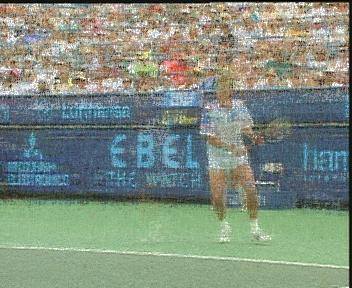}
		\end{minipage}
	}
	\subfigure{
		\begin{minipage}[t]{0.144\linewidth}
			\centering
			\includegraphics[width=1\linewidth]{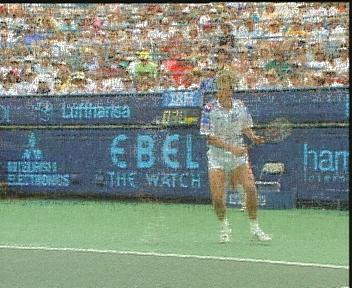}
		\end{minipage}
	}
	\subfigure{
		\begin{minipage}[t]{0.144\linewidth}
			\centering
			\includegraphics[width=1\linewidth]{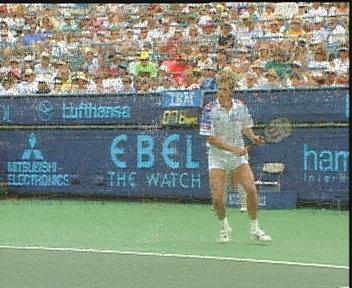}
		\end{minipage}
	}
	
	\vspace{-3mm}
	\setcounter{subfigure}{0}
	\subfigure{
		
        \rotatebox{90}{\bf{\scriptsize{~TMac3D-dec}}}
		\begin{minipage}[t]{0.144\linewidth}
			\centering
			\includegraphics[width=1\linewidth]{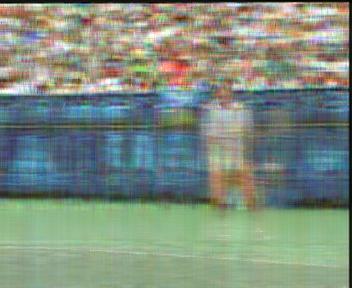}
		\end{minipage}
	}
	\subfigure{
		\begin{minipage}[t]{0.144\linewidth}
			\centering
			\includegraphics[width=1\linewidth]{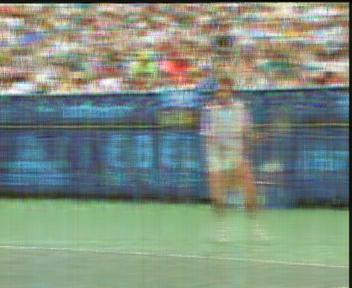}
		\end{minipage}
	}
	\subfigure{
		\begin{minipage}[t]{0.144\linewidth}
			\centering
			\includegraphics[width=1\linewidth]{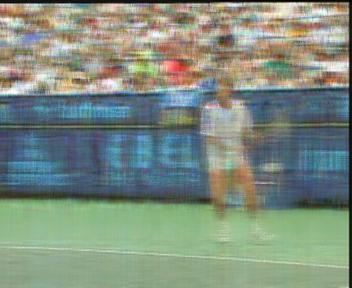}
		\end{minipage}
	}
	\subfigure{
		\begin{minipage}[t]{0.144\linewidth}
			\centering
			\includegraphics[width=1\linewidth]{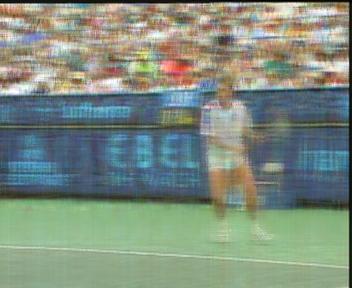}
		\end{minipage}
	}
	\subfigure{
		\begin{minipage}[t]{0.144\linewidth}
			\centering
			\includegraphics[width=1\linewidth]{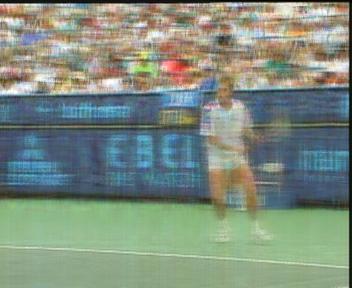}
		\end{minipage}
	}
	
	\vspace{-3mm}
	\setcounter{subfigure}{0}
	\subfigure{
		
        \rotatebox{90}{\bf{\scriptsize{~TMac3D-inc}}}
		\begin{minipage}[t]{0.144\linewidth}
			\centering
			\includegraphics[width=1\linewidth]{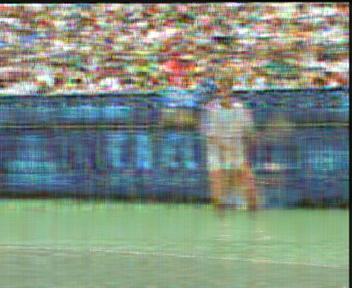}
		\end{minipage}
	}
	\subfigure{
		\begin{minipage}[t]{0.144\linewidth}
			\centering
			\includegraphics[width=1\linewidth]{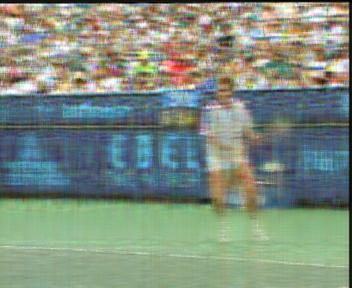}
		\end{minipage}
	}
	\subfigure{
		\begin{minipage}[t]{0.144\linewidth}
			\centering
			\includegraphics[width=1\linewidth]{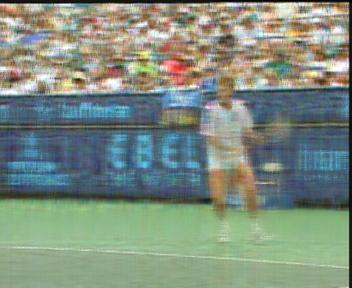}
		\end{minipage}
	}
	\subfigure{
		\begin{minipage}[t]{0.144\linewidth}
			\centering
			\includegraphics[width=1\linewidth]{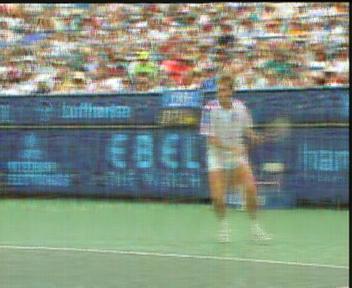}
		\end{minipage}
	}
	\subfigure{
		\begin{minipage}[t]{0.144\linewidth}
			\centering
			\includegraphics[width=1\linewidth]{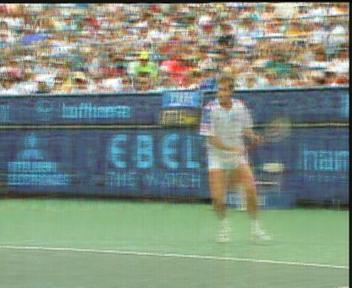}
		\end{minipage}
	}
	
	\vspace{-3mm}
	\setcounter{subfigure}{0}
	\subfigure{
		
        \rotatebox{90}{\bf{\scriptsize{~TMac4D-dec}}}
        \begin{minipage}[t]{0.144\linewidth}
			\centering
			\includegraphics[width=1\linewidth]{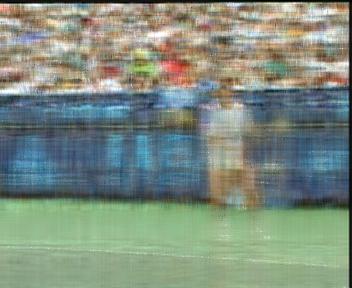}
		\end{minipage}
	}
	\subfigure{
		\begin{minipage}[t]{0.144\linewidth}
			\centering
			\includegraphics[width=1\linewidth]{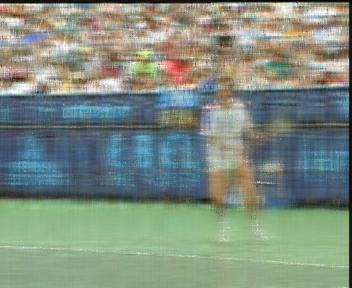}
		\end{minipage}
	}
	\subfigure{
		\begin{minipage}[t]{0.144\linewidth}
			\centering
			\includegraphics[width=1\linewidth]{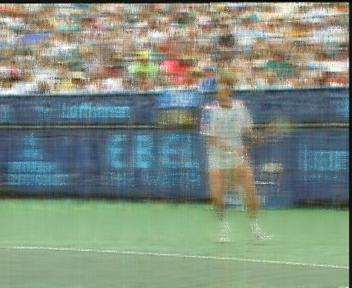}
		\end{minipage}
	}
	\subfigure{
		\begin{minipage}[t]{0.144\linewidth}
			\centering
			\includegraphics[width=1\linewidth]{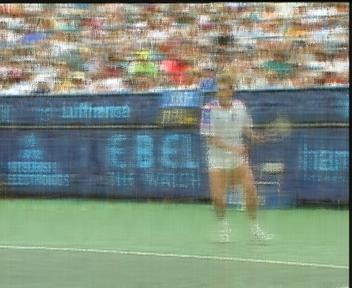}
		\end{minipage}
	}
	\subfigure{
		\begin{minipage}[t]{0.144\linewidth}
			\centering
			\includegraphics[width=1\linewidth]{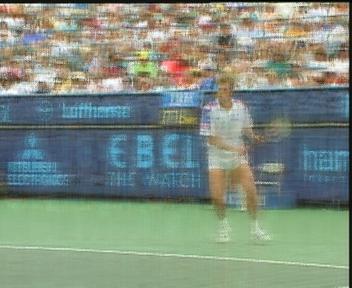}
		\end{minipage}
	}
	
	\vspace{-3mm}
	\setcounter{subfigure}{0}
	\subfigure{
		
        \rotatebox{90}{\bf{\scriptsize{~TMac4D-inc}}}
		\begin{minipage}[t]{0.144\linewidth}
			\centering
			\includegraphics[width=1\linewidth]{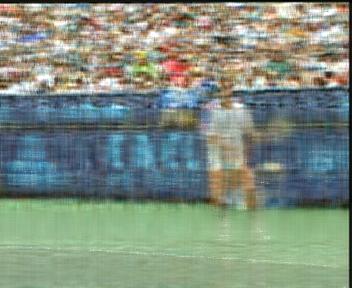}
		\end{minipage}
	}
	\subfigure{
		\begin{minipage}[t]{0.144\linewidth}
			\centering
			\includegraphics[width=1\linewidth]{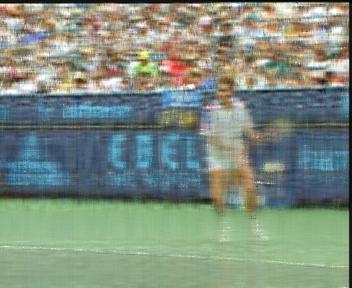}
		\end{minipage}
	}
	\subfigure{
		\begin{minipage}[t]{0.144\linewidth}
			\centering
			\includegraphics[width=1\linewidth]{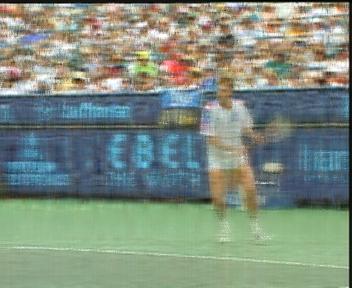}
		\end{minipage}
	}
	\subfigure{
		\begin{minipage}[t]{0.144\linewidth}
			\centering
			\includegraphics[width=1\linewidth]{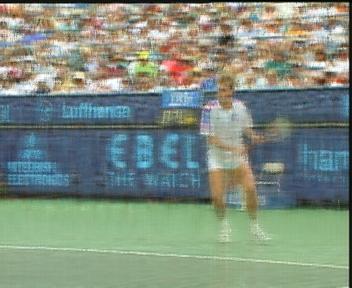}
		\end{minipage}
	}
	\subfigure{
		\begin{minipage}[t]{0.144\linewidth}
			\centering
			\includegraphics[width=1\linewidth]{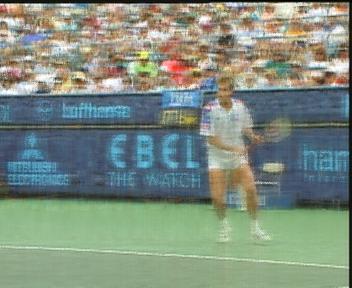}
		\end{minipage}
	}
	
	\vspace{-3mm}
	\setcounter{subfigure}{0}
	\subfigure{
		
        \rotatebox{90}{\bf{\scriptsize{~~~~LRQA-2}}}
		\begin{minipage}[t]{0.144\linewidth}
			\centering
			\includegraphics[width=1\linewidth]{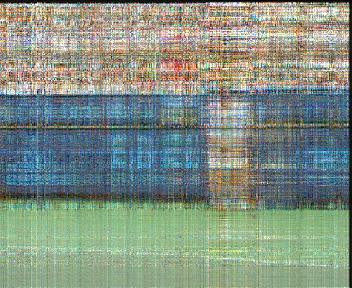}
		\end{minipage}
	}
	\subfigure{
		\begin{minipage}[t]{0.144\linewidth}
			\centering
			\includegraphics[width=1\linewidth]{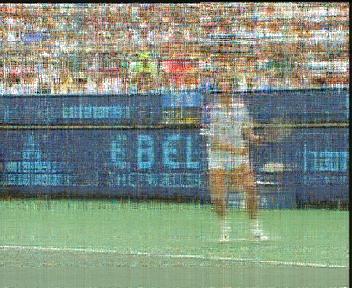}
		\end{minipage}
	}
	\subfigure{
		\begin{minipage}[t]{0.144\linewidth}
			\centering
			\includegraphics[width=1\linewidth]{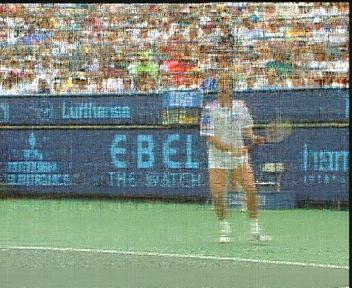}
		\end{minipage}
	}
	\subfigure{
		\begin{minipage}[t]{0.144\linewidth}
			\centering
			\includegraphics[width=1\linewidth]{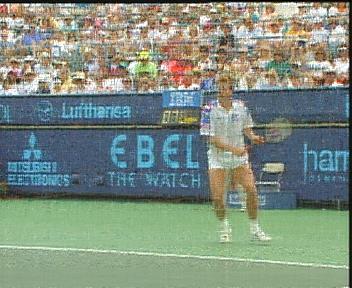}
		\end{minipage}
	}
	\subfigure{
		\begin{minipage}[t]{0.144\linewidth}
			\centering
			\includegraphics[width=1\linewidth]{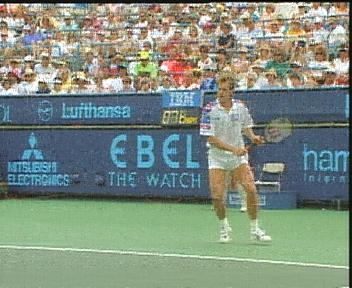}
		\end{minipage}
	}
	
	\vspace{-3mm}
	\setcounter{subfigure}{0}
	\subfigure{
		
        \rotatebox{90}{\bf{\scriptsize{~~~~QRTC}}}
		\begin{minipage}[t]{0.144\linewidth}
			\centering
			\includegraphics[width=1\linewidth]{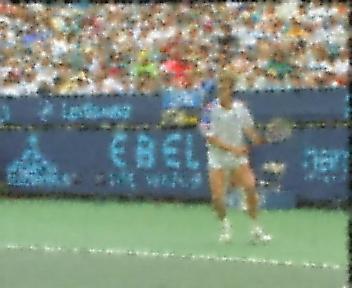}
		\end{minipage}
	}
	\subfigure{
		\begin{minipage}[t]{0.144\linewidth}
			\centering
			\includegraphics[width=1\linewidth]{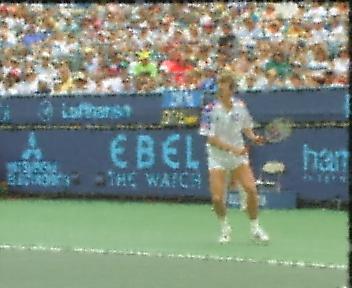}
		\end{minipage}
	}
	\subfigure{
		\begin{minipage}[t]{0.144\linewidth}
			\centering
			\includegraphics[width=1\linewidth]{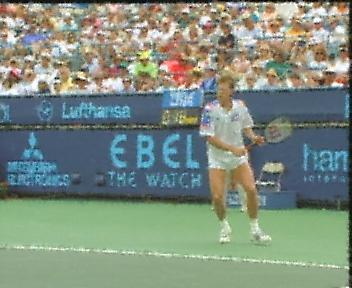}
		\end{minipage}
	}
	\subfigure{
		\begin{minipage}[t]{0.144\linewidth}
			\centering
			\includegraphics[width=1\linewidth]{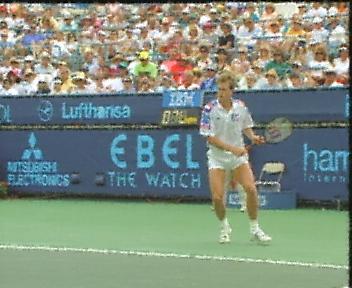}
		\end{minipage}
	}
	\subfigure{
		\begin{minipage}[t]{0.144\linewidth}
			\centering
			\includegraphics[width=1\linewidth]{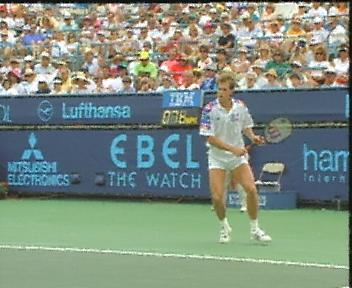}
		\end{minipage}
	}
	
	\vspace{-3mm}
	\setcounter{subfigure}{0}
	\subfigure[$\rho=0.1$]{
		
        \rotatebox{90}{\bf{\scriptsize{~~~MQRTC}}}
		\begin{minipage}[t]{0.144\linewidth}
			\centering
			\includegraphics[width=1\linewidth]{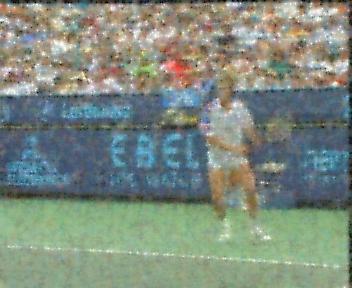}
		\end{minipage}
	}
	\subfigure[$\rho=0.2$]{
		\begin{minipage}[t]{0.144\linewidth}
			\centering
			\includegraphics[width=1\linewidth]{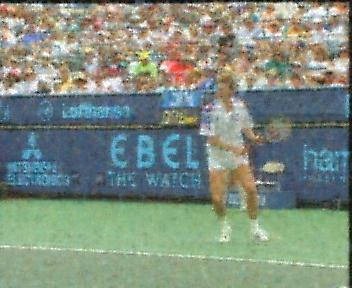}
		\end{minipage}
	}
	\subfigure[$\rho=0.3$]{
		\begin{minipage}[t]{0.144\linewidth}
			\centering
			\includegraphics[width=1\linewidth]{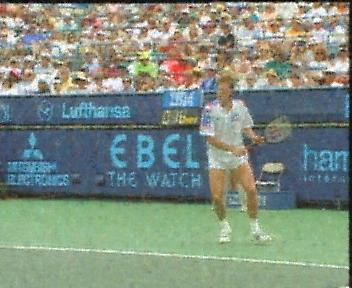}
		\end{minipage}
	}
	\subfigure[$\rho=0.4$]{
		\begin{minipage}[t]{0.144\linewidth}
			\centering
			\includegraphics[width=1\linewidth]{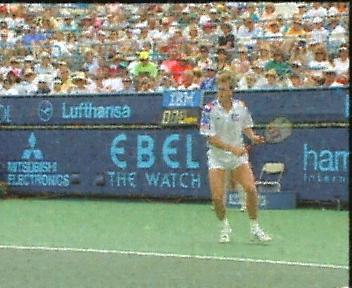}
		\end{minipage}
	}
	\subfigure[$\rho=0.5$]{
		\begin{minipage}[t]{0.144\linewidth}
			\centering
			\includegraphics[width=1\linewidth]{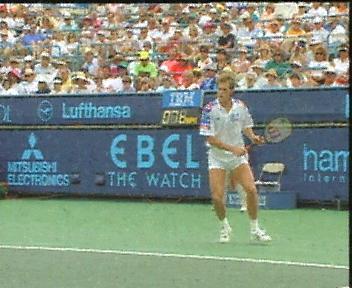}
		\end{minipage}
	}

    \caption{First frame of color video recovery while sample ratio $\rho$ from 0.1 to 0.5 using different methods.}
    \label{fig:differentrho}
\end{figure}

\begin{sidewaystable}[htbp]
\centering
\caption{{\small Quantitave quality indexes and running time (seconds) of different algorithm on the ``Stefan" video from $\rho=0.1$ to 0.5 which are displayed in Figure \ref{fig:differentrho}.}}
\label{table: differentrho}

\begin{tabular}{ |c|c|c|c|c|c|c|c|c|c|c| }
    \toprule
	$\rho$ & Indexes & ~~~TCTF~~~ & ~~~~TNN~~~~ & TMac3D-dec & TMac3D-inc & TMac4D-dec & TMac4D-inc & ~~~LRQA-2~~~ & ~~~QRTC~~~ & ~~~MQRTC~~~ \\ \bottomrule
	\hline
	\multirow{5}{*}{$\rho=0.1$} & RSE & -1.79 & -6.4133 & -6.7203 & -6.8304 & -6.7406 & -6.8705 & -5.0637 & (-6.9780) & \bf{-7.0966}  \\ \cline{2-11}
	 & PSNR & 8.6417 & 17.8723 & 18.4864 & 18.7066 & 18.5270 & 18.7869 & 15.1732 & (19.0018) & \bf{19.2389}  \\\cline{2-11}
	 & SSIM & 0.0818 & 0.6153 & 0.6620 & 0.6846 & 0.6687 & 0.6948 & 0.4232 & \bf{0.7158} & (0.7078)  \\\cline{2-11}
	 & FSIM & 0.6093 & 0.7068 & 0.6791 & (0.7163) & 0.6847 & \bf{0.7190} & 0.6926 & 0.6905 & 0.7046  \\\cline{2-11}
	 & time(s) & 268 & 1201 & 141 & 183 & 328 & 331 & 3330 & 326 & 693  \\ \hline
	\multirow{5}{*}{$\rho=0.2$} & RSE & -4.3148 & -7.2013 & -7.0963 & -7.3819 & -7.1831 & -7.4486 & -6.2521 & (-7.9459) & \bf{-8.2292}  \\\cline{2-11}
	 & PSNR & 13.6755 & 19.4483 & 19.2385 & 19.8096 & 19.4120 & 19.9429 & 17.5500 & (20.9377) & \bf{21.5042}  \\\cline{2-11}
	 & SSIM & 0.4807 & 0.7218 & 0.7091 & 0.7432 & 0.7240 & 0.7549 & 0.6079 & (0.8156) & \bf{0.8256}  \\\cline{2-11}
	 & FSIM & 0.7386 & 0.7787 & 0.6968 & 0.7374 & 0.7102 & 0.7460 & 0.7617 & (0.7875) & \bf{0.8060}  \\\cline{2-11}
	 & time(s) & 273 & 1241 & 70 & 138 & 131 & 185 & 3178 & 177 & 689  \\ \hline
	\multirow{5}{*}{$\rho=0.3$} & RSE & -6.7642 & -7.9653 & -7.3641 & -7.7258 & -7.5270 & -7.8525 & -7.2777 & (-8.7676) & \bf{-9.0605}  \\\cline{2-11}
	 & PSNR & 18.5741 & 20.9765 & 19.7741 & 20.4974 & 20.0998 & 20.7508 & 19.6011 & (22.5809) & \bf{23.1668}  \\\cline{2-11}
	 & SSIM & 0.7206 & 0.7995 & 0.7416 & 0.7785 & 0.7643 & 0.7950 & 0.7327 & (0.8752) & \bf{0.8815}  \\\cline{2-11}
	 & FSIM & 0.8162 & 0.8310 & 0.7166 & 0.7578 & 0.7394 & 0.7736 & 0.8162 & (0.8482) & \bf{0.8633}  \\\cline{2-11}
	 & time(s) & 270 & 1319 & 48 & 123 & 117 & 170 & 2902 & 160 & 562  \\ \hline
	\multirow{5}{*}{$\rho=0.4$} & RSE & -7.9171 & -8.7597 & -7.5969 & -8.0123 & -7.8549 & -8.2227 & -8.2366 & (-9.5559) & \bf{-9.8308}  \\\cline{2-11}
	 & PSNR & 20.8800 & 22.5653 & 20.2396 & 21.0704 & 20.7556 & 21.4911 & 21.5189 & (24.1576) & \bf{24.7074}  \\\cline{2-11}
	 & SSIM & 0.8125 & 0.8586 & 0.7677 & 0.8051 & 0.7977 & 0.8271 & 0.8206 & (0.9149) & \bf{0.9174}  \\\cline{2-11}
	 & FSIM & 0.8388 & 0.8728 & 0.7363 & 0.7775 & 0.7659 & 0.7987 & 0.8628 & (0.8915) & \bf{0.9025}  \\\cline{2-11}
	 & time(s) & 250 & 1448 & 42 & 115 & 104 & 160 & 2939 & 148 & 471  \\ \hline
	\multirow{5}{*}{$\rho=0.5$} & RSE & -7.4885 & -9.6432 & -7.8107 & -8.2800 & -8.1682 & -8.5860 & -9.2639 & (-10.3641) & \bf{-10.6105}  \\\cline{2-11}
	 & PSNR & 20.0227 & 24.3323 & 20.6671 & 21.6058 & 21.3821 & 22.2178 & 23.5737 & (25.7741) & \bf{26.2668}  \\\cline{2-11}
	 & SSIM & 0.7902 & 0.9047 & 0.7891 & 0.8275 & 0.8257 & 0.8541 & 0.8805 & \bf{0.94252} & (0.94251)  \\\cline{2-11}
	 & FSIM & 0.8193 & 0.9076 & 0.7528 & 0.7942 & 0.7898 & 0.8215 & 0.9022 & (0.9237) & \bf{0.9309}  \\\cline{2-11}
	 & time(s) & 244 & 1746 & 39 & 120 & 85 & 146 & 2814 & 143 & 397  \\
	\hline
\end{tabular}

\end{sidewaystable}

\section{Conclusions}\label{sec7}
\noindent
\par
We introduce the concept of gQt-product for third-order quaternion tensor and then define QDFT. Based the newly-defined QDFT, we introduce gQt-SVD of third-order quaternion tensors. We define gQt-rank for third-order quaternion tensor via its gQt-SVD and show the existence of low gQt-rank optimal approximation. We also generalize these results from mode-3 (QRTC) to three modes (MQRTC) of third-order quaternion tensor, and obtain multi-gQt-rank. Numerical experiments indicate that third-order quaternion tensors generated by color videos in real life have an inherent low-rank property.

Therefore, we establish low-rank quaternion tensor completion models based on gQt-rank and multi-gQt-rank to recover color videos with partial data loss. Using TV-regularization to capture the spatial stability feature, we obtain our novel tensor recovery models for color video inpainting. We present  two ALS algorithms (Algorithms \ref{alg:1} and \ref{alg:4}) to solve our models. Their convergence is established (see Subsection 4.3).  Extensive numerical experiments indicate that our approaches QRTC and MQRTC outperform some existing state-of-the-arts methods on various video datasets with different sample ratios, which also demonstrate the robustness of our methods. Especially, MQRTC outperforms QRTC.

\section*{Acknowledgement}
We are very grateful to Dr. Yongyong Chen for sharing us with the code in \cite{MiaoJi2019}. Also many thanks go to the authors for their open-source codes of their studies. The author Liping Zhang  was supported by the National Natural Science Foundation of China (Grant No. 12171271).


\bibliographystyle{unsrt}   
\bibliography{ref}

\clearpage

\end{document}